\documentclass[a4paper,reqno]{amsart}
\newif\ifjournalversion

\usepackage[english]{babel}
\newcommand{\changelocaltocdepth}[1]{%
  \addtocontents{toc}{\protect\setcounter{tocdepth}{#1}}%
  \setcounter{tocdepth}{#1}%
}

\usepackage{thmtools}
\usepackage{amsmath}
\usepackage{amssymb}
\usepackage{amsthm}
\usepackage{amsfonts}
\usepackage{graphicx}
\usepackage{comment}
\usepackage{enumitem}
\usepackage{color}
\usepackage[colorlinks=true]{hyperref}
\usepackage{cleveref}
\usepackage{todonotes}
\usepackage{oldgerm}
\usepackage{longtable}

\newcommand{\dummy}{\blacktriangle} 

\declaretheorem[name=Theorem,numberwithin=section]{thm}
\newtheorem{lem}[thm]{Lemma}
\newtheorem{Lemma}[thm]{Lemma}
\newtheorem{prop}[thm]{Proposition}
\newtheorem{Prop}[thm]{Proposition}
\newtheorem{claim}{Claim}[thm]
\newtheorem*{claim*}{Claim}
\declaretheorem[name=Corollary,sibling=thm]{Cor}
\newtheorem{Question}[thm]{Question}
\newtheorem{RP}[thm]{Research Proposal}

\theoremstyle{definition}
\newtheorem{defn}[thm]{Definition}
\newtheorem{Def}[thm]{Definition}
\newtheorem{exmp}[thm]{Example}
\newtheorem{remark}[thm]{Remark}
\newtheorem{Remark}[thm]{Remark}

\newtheorem{fact_n_def}[thm]{Fact and Definition}
\newtheorem{Fact}[thm]{Fact}

\newcommand{\subd}{\mathrel{\triangleleft}}
\newcommand{\PartEmb}{\operatorname{PartEmb}}

\newcommand{\Lk}{\operatorname{Lk}}
\newcommand{\id}{\operatorname{id}}
\newcommand{\supp}{\operatorname{supp}}

\DeclareMathOperator{\graph}{graph}

\DeclareMathOperator{\prj}{pr}

\DeclareMathOperator{\dom}{dom}

\DeclareMathOperator{\pr}{pr}
\DeclareMathOperator{\Emb}{Emb}

\renewcommand{\mid}{\,:\,}
\newenvironment{enumerate-(a)}{\begin{enumerate}[label={\upshape (\alph*)}, leftmargin=2pc]}{\end{enumerate}}
\newenvironment{enumerate-(a)-r}{\begin{enumerate}[label={\upshape (\alph*)}, leftmargin=2pc,resume]}{\end{enumerate}}
\newenvironment{enumerate-(a)-5}{\begin{enumerate}[label={\upshape (\alph*)}, leftmargin=2pc,start=5]}{\end{enumerate}}
\newenvironment{enumerate-(A)}{\begin{enumerate}[label={\upshape (\Alph*)}, leftmargin=2pc]}{\end{enumerate}}
\newenvironment{enumerate-(A)-r}{\begin{enumerate}[label={\upshape (\Alph*)}, leftmargin=2pc,resume]}{\end{enumerate}}
\newenvironment{enumerate-(i)}{\begin{enumerate}[label={\upshape (\roman*)}, leftmargin=2pc]}{\end{enumerate}}
\newenvironment{enumerate-(i)-r}{\begin{enumerate}[label={\upshape (\roman*)}, leftmargin=2pc,resume]}{\end{enumerate}}
\newenvironment{enumerate-(I)}{\begin{enumerate}[label={\upshape (\Roman*)}, leftmargin=2pc]}{\end{enumerate}}
\newenvironment{enumerate-(I)-r}{\begin{enumerate}[label={\upshape (\Roman*)}, leftmargin=2pc,resume]}{\end{enumerate}}
\newenvironment{enumerate-(1)}{\begin{enumerate}[label={\upshape (\arabic*)}, leftmargin=2pc]}{\end{enumerate}}
\newenvironment{enumerate-(1)-r}{\begin{enumerate}[label={\upshape (\arabic*)}, leftmargin=2pc,resume]}{\end{enumerate}}

\newenvironment{enumerate-(star)}{\begin{enumerate}[label={\upshape{(\( \star_{ \arabic*} \))}}, leftmargin=2pc]}{\end{enumerate}}

\renewcommand{\b}{\beta}
\renewcommand{\o}{\omega}
\newcommand{\lo}{<\!\omega{}}
\newcommand{\f}{\varphi}

\newcommand{\bW}{\mathbf{W}}
\newcommand{\vv}{\mathbf{v}}

\newcommand{\ee}{\mathbf{e}}

\renewcommand{\d}{\delta}

\newcommand{\Q}{\mathbb{Q}} 
\newcommand{\R}{\mathbb{R}} 
 
\newcommand{\N}{\mathbb{N}} 
\newcommand{\U}{\mathbb{U}} 

\newcommand{\UU}{\mathcal{U}}

\newcommand{\PP}{\mathcal{P}}
\newcommand{\DD}{\mathcal{D}}
\newcommand{\FF}{\mathcal{F}}

\newcommand{\MM}{\mathcal{M}}
\newcommand{\BB}{\mathcal{B}}

\newcommand{\RR}{\mathcal{R}}

\renewcommand{\AA}{\mathcal{A}}

\newcommand{\TT}{\mathcal{T}} 
\newcommand{\CC}{\mathcal{C}} 

\newcommand{\AAA}{\mathfrak{A}} %
\newcommand{\BBB}{\mathfrak{B}} 
\newcommand{\MMM}{\mathfrak{M}} 
\newcommand{\SSS}{\mathfrak{S}} 
\newcommand{\TTT}{\mathfrak{T}} 
\newcommand{\PPP}{\mathfrak{P}} %
\newcommand{\CCC}{\mathfrak{C}} 

\newcommand{\bzero}{\mathbf{0}}
\newcommand{\bone}{\mathbf{1}}

\newcommand{\bDelta}{\mathbf{\Delta}}

\newcommand{\Carr}{\operatorname{Carr}}
\newcommand{\interior}{\operatorname{int}}
\newcommand{\scl}{\operatorname{scl}}
\renewcommand{\int}{\operatorname{int}}

\newcommand{\Mod}{\operatorname{Mod}}

\allowdisplaybreaks

\mathchardef\mhyphen="2D
\newcommand{\rest}{\!\restriction\!}
\newcommand{\homeo}{\approx}

\newcommand{\PL}{\operatorname{PL}}
\newcommand{\St}{\operatorname{St}}

\newcommand{\es}{\varnothing}
\renewcommand{\ge}{\geqslant}
\renewcommand{\le}{\leqslant}
\newcommand{\nle}{\nleqslant}
\renewcommand{\geq}{\geqslant}
\renewcommand{\leq}{\leqslant}
\newcommand{\e}{\varepsilon}
\renewcommand{\Cap}{\bigcap}
\renewcommand{\Cup}{\bigcup}
\newcommand{\la}{\langle}
\newcommand{\ra}{\rangle}
\newcommand{\cat}{{}^\frown}
\newcommand{\ran}{\operatorname{ran}}

\renewcommand{\subseteq}{\subset}

\newcommand{\obDelta}{\overset{\circ}{\mathbf{\Delta}}{}}
\newcommand{\uobDelta}{\overset{\circ}{\underline{\mathbf{\Delta}}}{}}
\newcommand{\ubDelta}{{\underline{\mathbf{\Delta}}}{}}

\title[Simplicial complexes and non-compact 2- and 3-manifolds]{Borel classification of simplicial complexes and non-compact 2- and 3-manifolds} 

\author[M.~Iannella]{Martina Iannella} 
\address{Department of Discrete Mathematics and Geometry, TU Wien.~Wiedner Hauptstr.~8--10, 1040 Vienna, Austria}
\email{martina.iannella@tuwien.ac.at}
\author[V.~Weinstein]{Vadim Weinstein} 
\address{Center for Ubiquitous Computing\\
Erkki Koiso-Kanttilan katu 3\\ 
door E P.O Box 4500\\
FI-90014 University of Oulu} 
\email{vadim.weinstein@iki.fi}

\subjclass[2020]{Primary: 03E15, 06E75, 57K20, 57K30, 57Q05. Secondary: 54H05, 57Q15, 57M30}

\thanks{The first author was partially supported by the Italian PRIN 2017 Grant ``Mathematical Logic: models, sets, computability", and by the Austrian Science Fund (FWF) [10.55776/ESP1625325]. The second author was supported by a European Research Council Advanced Grant (ERC AdG, ILLUSIVE: Foundations of Perception Engineering, 101020977)}

\begin{document}

\begin{abstract}
  We generalize the Stone space of ultrafilters on Boolean algebras and prove a generalization of Stone duality which is applicable to locally compact Polish spaces. Using this, we obtain complete invariants for simplicial complexes up to PL-homeomorphism and for non-compact $2$- and $3$-manifolds up to homeomorphism. We prove that the homeomorphism relation on non-compact $2$-manifolds without boundary, the homeomorphism relation on non-compact $3$-manifolds with or without boundary, the homeomorphism relation on open subsets of $\mathbb{R}^2$ and $\mathbb{R}^3$, and conjugacy of Cantor sets in $\mathbb{R}^3$ are classifiable by countable structures. Together with known lower bounds, this implies that these relations are Borel bireducible with isomorphism of countable graphs. We also show that PL-homeomorphism of Heine-Borel simplicial complexes and PL-homeomorphism of PL $n$-manifolds, for every $n$, are classifiable by countable structures.
\end{abstract}

\maketitle

\tableofcontents

\section{Introduction}




    

The classification of non-compact manifolds has long been a subject of extensive mathematical interest. An algebraic classification of non-compact topological $2$-manifolds without boundary was obtained by Goldman in 1971 \cite{Goldman71}, and a more general result encompassing topological $2$-manifolds with boundary was proved by Brown and Messer in 1979 \cite{BM79}. The corresponding problem for $3$-manifolds, however, has long been regarded as substantially more difficult, largely because of the existence of Whitehead manifolds and other pathological examples; see \cite{whitehead1935open,scott1983geometries} for classical results and \cite{Maillot2008} for a discussion. This naturally raises the question of whether non-compact topological $3$-manifolds admit complete algebraic invariants.

Descriptive set theory provides a framework in which such questions can be formulated precisely. The complexity of classification problems is measured using \emph{Borel reducibility} (see \cite{Hjo00,Gao09}). Given equivalence relations $E$ on a standard Borel space $X$ and $F$ on a standard Borel space $Y$, we say that $E$ is \emph{Borel reducible} to $F$ if there exists a Borel map $f\colon X\to Y$ such that
\(
x\,E\,y \Longleftrightarrow f(x)\,F\,f(y).
\)
When $F$ is the isomorphism relation on a class of countable structures, we say that $E$ is \emph{classifiable by countable structures}. 
 We say that $E$ and $F$ are \textit{Borel bireducible} if $E$ is Borel reducible to $F$ and $F$ is Borel reducible to $E$.
 Borel reducibility thus yields a hierarchy of classification problems according to their complexity and provides a framework for proving that classification by certain classes of invariants is ``impossible".

The isomorphism relation on countable structures, denoted $\cong$, occupies a distinguished place in this hierarchy and is often regarded as the dividing line between ``classifiable'' and ``non-classifiable'' equivalence relations. Among countable structures, the isomorphism relation $\cong_{\mathcal{G}}$ between countable graphs has maximal Borel complexity. Many natural classification problems have been shown to lie strictly above this benchmark. For example, neither conjugacy of autohomeomorphisms of the unit square nor homeomorphism of compact Polish spaces is classifiable by countable structures (see \cite{Hjo00}); the same holds for equivalence of wild knots, proved by one of the present authors in \cite{Kul17}\footnote{Note the author's name change from Kulikov to Weinstein.}. On the other hand, Riemann surfaces up to conformal equivalence are classifiable by countable sets of reals, and hence by countable structures (see \cite{HjoKec2000}).

Throughout the paper, manifolds are assumed to be topological unless stated otherwise. 
From the perspective of descriptive set theory, the classification problem for manifolds up to homeomorphism has received little attention. It appears only tangentially in \cite{GKB13,Kul17}. While the classification problem of compact manifolds up to homeomorphism has a trivial solution from the descriptive-set-theoretic perspective, since there are only countably many homeomorphism types \cite{CK70}, it is known that for $n\ge2$ graph isomorphism Borel reduces to the homeomorphism relation on non-compact $n$-manifolds (see Proposition \ref{thm:CongToHomeo}). This is based on the fact that every homeomorphism between totally disconnected subsets of $\mathbb{R}^n$ extends to a homeomorphism of the ambient space (a consequence of \cite[Theorem~4.1]{CM83}), while homeomorphism of such subsets has exactly the complexity of $\cong_{\mathcal{G}}$ (see \cite{CamGao2001}). 
See \cite[Corollary 6.5]{GKB13} for a different proof in dimension $3$. 
We emphasize that the classifications of Goldman and Brown-Messer predate the development of Borel reducibility, and therefore do not address the descriptive complexity of these problems. Prior to the results presented here and the related recent developments discussed below, no matching upper bound was known. In independent work, Bergfalk and Smythe proved that the homeomorphism relation for non-compact connected $2$-manifolds without boundary is classifiable by countable structures, and so Borel bireducible with graph isomorphism (see \cite{BS25}); they also investigated classification problems involving complete hyperbolic manifolds up to isometry. Hoganson and Zomback proved that the homeomorphism relation for orientable non-compact surfaces equipped with a pants decomposition is Borel bireducible with $\cong_{\mathcal{G}}$ (see \cite{HZ24}). In ongoing work, Gompf and Panagiotopoulos study the descriptive complexity of the classification of manifolds in several dimensions and categories. In particular, they have announced an approach to the upper bound for the homeomorphism relation on topological $n$-manifolds and the diffeomorphism relation for smooth $n$-manifolds, for $n>1$, by isomorphism of countable structures (see \cite{GP26}).

Our methods provide a uniform algebraic framework for the classification of non-compact 2- and 3-manifolds. They yield complete invariants for manifolds both with and without boundary. At the level of Borel reducibility, they establish classification by countable structures for non-compact 2-manifolds without boundary and for non-compact 3-manifolds, both with and without boundary. Our main theorems show that, despite the striking differences between the classical classification theories in dimensions 2 and 3, the corresponding homeomorphism relations have the same descriptive-set-theoretic complexity. While the $2$-dimensional case was also obtained by Bergfalk and Smythe, our methods yield the first such result for $3$-manifolds. 
Denote by $\homeo_2$, $\homeo_3$ and $\homeo_3^\partial$ the homeomorphism relation on non-compact $2$-manifolds without boundary, non-compact $3$-manifolds without boundary, and non-compact $3$-manifolds with boundary, respectively.

\begin{thm}
The relations $\homeo_2$, $\homeo_3$ and $\homeo_3^\partial$ are classifiable by countable structures. Equivalently, they are Borel bireducible with $\cong_{\mathcal{G}}$.
\end{thm}

Our proof proceeds by establishing a more general classification theorem for polyhedral objects. We first show the following for the classification of simplicial complexes up to PL-homeomorphism.

\begin{thm}
Simplicial complexes up to PL-homeomorphism are classifiable by countable structures.
\end{thm}

This result yields, in particular, the corresponding classification theorem for PL-manifolds in arbitrary dimension.

\begin{thm}
For every $n$, the PL-homeomorphism relation on PL $n$-manifolds is classifiable by countable structures.
\end{thm}

For the classification of $2$- and $3$-manifolds, we refine this approach by assigning a \emph{sorted complemented algebra} (Section~\ref{ssec:SCA}) to each manifold and proving that this algebra is a complete invariant. Equivalently, the homeomorphism type of a $2$- or $3$-manifold is completely determined by the partial order of compact polyhedra with ``rational'' vertex points of one of their triangulations ordered by $p<q$ if and only if the closure of $p$ is contained in the interior of~$q$ (Corollary~\ref{cor:Manifolds-class-simple}). We further show that this assignment can be carried out in a Borel way, yielding the descriptive-set-theoretic classification above.

These results have several consequences. The relations $\homeo_2$ and $\homeo_3$ are strictly simpler than homeomorphism of compact or locally compact Polish spaces and than equivalence of wild knots (see \cite{ZIELINSKI2016635,Kul17}). They are Borel bireducible with graph isomorphism and, by \cite{PaoliniShelah2024}, also with isomorphism of torsion-free abelian groups. In particular, $\homeo_2$ and $\homeo_3$ have the same Borel complexity. For $n>3$, the methods developed here do not determine the descriptive complexity of $\homeo_n$. The upper bound problem is being investigated in ongoing work of Gompf and Panagiotopoulos (see Section~\ref{sec:Final}).

Moreover, our methods yield several further classification results. We answer \cite[Question~5.5]{GKB13} by proving that the conjugacy relation on Cantor sets in $\mathbb{R}^3$ is Borel reducible to isomorphism of countable structures (Theorem~\ref{thm:CantorBorelRed}). As an additional application, we show that
open subsets of $\mathbb{R}^2$ and $\mathbb{R}^3$ up to homeomorphism are classifiable by countable structures (Corollary \ref{cor:RedFromOpenToMan}).

The conceptual core of the paper is a generalization of Stone duality, which we call the \emph{blurry duality theorem} (Theorem~\ref{thm:BlurryDuality}). Classical Stone duality is based on the observation that a homeomorphism between compact totally disconnected spaces preserves their Boolean algebras of clopen sets. Consequently, two such spaces are homeomorphic if and only if their Boolean algebras of clopen sets are isomorphic \cite{CamGao2001}. Our approach follows the same philosophy, replacing Boolean algebras of clopen sets by sorted complemented algebras arising from polyhedral bases.

Unlike clopen bases, a countable polyhedral basis of a manifold is not generally preserved by an arbitrary homeomorphism. In dimensions $2$ and $3$, however, the uniqueness of PL structures, due to the classical work of Moise and Bing \cite{Bi59,Mo52,Mo54} (see Fact~\ref{fact:MoiseBing}), implies that every homeomorphism can be ``tweaked" so as to preserve rational polyhedra (for a formal statement, see Theorem~\ref{thm:QPL-homeo}).

A second aspect of classical Stone duality is the reconstruction of a compact totally disconnected space from the Boolean algebra of its clopen basis. The Stone space of this Boolean algebra, consisting of its ultrafilters, is homeomorphic to the original space. This construction cannot be applied directly to polyhedral bases, since the Stone space is always totally disconnected. To overcome this difficulty, we introduce \emph{blurry filters} (Section~\ref{ssec:SCA}), a natural generalization of ultrafilters. The space of blurry filters associated with a sorted complemented algebra reconstructs the corresponding topological space whenever suitable axioms are satisfied (Section~\ref{ssec:Duality}). These axioms hold for $2$- and $3$-manifolds because of their special PL-topological properties. See Section~\ref{sec:Final} for further discussion.

Another essential ingredient of the proof is a Borel version of triangulation for $2$- and $3$-manifolds (Section~\ref{ssec:BorelTriangulation}). In dimension $2$, we use the triangulation result obtained in \cite{BS25}. In dimension $3$, we adapt the classical triangulation arguments of Moise and Bing, \cite{Mo52,Bing54}, to the descriptive-set-theoretic setting, thereby strengthening the classical theorem by producing triangulations in a Borel way.

At present, we do not know whether the invariants introduced here admit a direct geometric interpretation. Thus, as far as we are aware, Maillot's observation that ``For open $3$-manifolds [...] there is not even a conjectural description [...] in terms of geometric ones'' \cite{Maillot2008} remains valid. Nevertheless, our results remove a possible descriptive-set-theoretic obstruction to such a classification: since non-compact $2$- and $3$-manifolds are classifiable by countable structures, there is no obstruction of this kind to the existence of geometrically meaningful complete invariants.

In Section~\ref{sec:Final} we discuss further research directions suggested by this work. More broadly, our results show that manifolds admit canonical presentations as countable structures. This suggests that classical algebraic invariants, such as homology and fundamental groups, may become logically definable within these presentations. It is therefore conceivable that major classification results, such as the classification of simply connected closed $3$-manifolds (now a theorem by Perelman \cite{Perelman2002,Perelman2003a,Perelman2003b,CaoZhu2006}), could admit alternative proofs using methods of this kind.

The paper is structured as follows. In
Section~\ref{sec:Classification_non_borel} we prove that the sorted
complemented algebras (SCA) obtained from polyhedra are a complete
invariant for the homeomorphism on $2$- and $3$-manifolds with or without
boundary. We do not do any descriptive set theoretic work in this
section.  First we introduce the algebras and blurry filters
(Section~\ref{ssec:SCA}), then we introduce \emph{basis spaces}
(Section~\ref{ssec:BasisSpaces}) and prove the blurry duality
theorem~(Section \ref{ssec:Duality}) which is a generalization of the
classical Stone duality. In Section~\ref{ssec:PointAdjustment} we
develop the machinery of PL-geometry and show that PL-homeomorphic
manifolds are $\Q$-PL-homeomorphic, i.e.  that the homeomorphism can
always be chosen so that it preserves a suitably chosen countable set
of polyhedra. In Section~\ref{ssec:Classification_NON_BOREL} we prove
that this leads to complete
invariants. Section~\ref{sec:Borelification} presents the same results
in the paradigm of Borel reducibility. We start by introducing the
relevant concepts and preliminary results from descriptive set theory
(Section~\ref{ssec:BasicDST}), then we show how to present SCAs and simplicial complexes as standard Borel spaces and how to obtain a Borel classification by countable structures of the PL-homeomorphism on Heine-Borel simplicial complexes
(Sections~\ref{ssec:ManToBS}). In Section \ref{sec:manifolds} we show that many classes of non-compact $n$-manifolds are standard Borel spaces and how to obtain Borel
classifications by countable structures of their corresponding homeomorphism relations. In Section \ref{ssec:BorelTriangulation} we also prove that $3$-manifolds can be triangulated in a
Borel way (this is perhaps technically the most
challenging part of the paper). We then prove that the homeomorphism relation on open subsets of $\R^n$, $n\in\{2,3\}$, and
conjugacy relation of Cantor sets in \(\R^3\) are classifiable by countable structures (Sections \ref{ssec:OpenSubsets} and \ref{ssec:BorelClassCantor}). Finally, we discuss other
applications and open problems in Section~\ref{sec:Final}.

\vspace{10pt}

\section{Classification without Borel reducibility}
\label{sec:Classification_non_borel}

The goal of this section is to obtain complete invariants for
simplicial complexes up to PL-homeomorphism, and for 2- and
3-manifolds up to homeomorphism. At first invariants will be
\emph{sorted complemented algebras}
(Theorems~\ref{thm:ClassificationOfSimplicialComplexes} and
\ref{thm:ClassificationOfManifolds}) but then simplified to partial
orders (Theorems~\ref{thm:PL-class-simple} and Corollary
\ref{cor:Manifolds-class-simple}).  In Section \ref{ssec:SCA} we
introduce sorted complemented algebras and blurry filters, in Section
\ref{ssec:BasisSpaces} basis spaces are defined, and in Section
\ref{ssec:Duality} we prove the blurry duality theorem connecting the
two and generalizing the Stone duality
(Theorem~\ref{thm:BlurryDuality}). Then, in Section
\ref{ssec:PointAdjustment} we develop the necessary theory of
PL-topology and in Section~\ref{sssec:PLClass} we show how to obtain
sorted complemented algebras from simplicial complexes and prove that
they are complete invariants of PL-homeomorphism. We then apply this
to characterize $2$- and $3$-manifolds with and without boundary up to
homeomorphism (Section~\ref{sssec:Classification23man}), and Cantor
sets up to conjugacy (Section~\ref{sssec:Cantor1}).  We will turn to
the descriptive theoretic treatment in
Section~\ref{sec:Borelification}.

\vspace{10pt}

\subsection{Notation}

By $A\subset B$ we mean that $A$ is a subset of $B$, possibly $A=B$.
If we want to emphasize that $A$ is a proper subset, we write
$A\subsetneq B$. If $f\colon X\to Y$ is a map and $A$ is a set, we
denote by $fA$ or $f[A]=f[A\cap X]$ the image of $A$ in~$Y$, $fA=\{f(x)\mid x\in A\cap X\}$.
Interchangeably, we
denote the composition of functions $f\colon X\to Y$ and
$g\colon Y\to Z$ either by $g\circ f$ or~$gf$. We denote both finite
and infinite sequences as $\bar z=(z_i)_{i\in I}$ for e.g. $I=\N$. 
By $l(\bar z)$ we denote the length of $\bar z$,
and we start indexing from
$0$ unless mentioned otherwise.
If
each element of the sequence is an element of $A$, we denote
$\bar z\subset A$. Slightly abusing notation, $|\bar z|$ is
the range of the sequence, $|\bar z|=\{z_i\mid i\in I\}$.
The concatenation of sequences is denoted by $\bar z\cat \bar w$.
If $\bar s=\bar z\cat \bar w$, we write $\bar z\subset \bar s$
and say that $\bar s$ is an \textbf{end-extension} of $\bar z$.
If $\bar s^0\subset \bar s^1\subset \cdots$ is a sequence of
sequences such that $\bar s^{i+1}$ is an end-extension of $\bar s^i$,
then we denote by $\Cup_{i\in\N}\bar s^i$ the sequence $\bar s$
whose $j$-th element equals the $j$-th element of some (any) $\bar s_i$
with $l(s_i)>j$.

If $X$ is a topological space and $Y \subseteq X$, then by $\interior_X(Y)$ and $\partial_X Y$ we denote, respectively, the interior and boundary of $Y$ in $X$. If $X$ is clear from the context, it is dropped from subscript. Sometimes the interior of $Y$ is denoted also by $\overset{\circ}{Y}$. The closure of $Y$ is denoted by~$\overline{Y}$.

Open and closed intervals in $\R$ are denoted respectively by
$(a,b)$ and $[a,b]$.
Given a metric space $(X,d)$, by 
\begin{align*}
    & B_X(x,\e)=\{y \in X \mid d(x,y)<\e\}, \text{ and}\\
    & \overline{B_X(x,\e)}=\{y \in X \mid d(x,y)\le \e\}
\end{align*}
we always denote, respectively, the open and closed ball in $X$ with center $x \in X$ 
and radius $\e \in \R_+$, where $\R_+$ is the space of positive real numbers. 
We drop 
$X$ from the lower case when it is clear from the context.

\vspace{10pt}

\subsection{Sorted complemented algebras}
\label{ssec:SCA}


\begin{Def}\label{def:algebra}
  Let $L=\{\le, \bzero, \bone, c\}$ be a first-order vocabulary
  with one binary predicate symbol $\le$, two constant symbols
  $\bzero$ and $\bone$, and one unary function symbol~$c$. A 
  \textbf{complemented algebra}, briefly \textbf{CA},
  is an $L$-model
  $\AA=(A,\le^{\AA},\bzero^{\AA}, \bone^{\AA}, c^{\AA})$
  satisfying the following axioms:
  \begin{enumerate}[label={\upshape (CA\arabic*)}, leftmargin=3pc]
  \item\label{alg1} The relation $\le^{\AA}$ is a partial order on $A$, i.e., a reflexive, antisymmetric, and transitive relation on $A$.
  \item\label{alg2} $\bone^{\AA}$ is the unique $\le^{\AA}$-maximal element, 
    and $\bzero^{\AA}$ is the unique $\le^{\AA}$-minimal element.
  \item\label{alg3} For each $a\in A$, there is a unique 
    $x\in A$ which satisfies the two formulas
    \begin{align*}
      &\forall y((y\le a\land y\le x)\rightarrow y=\bzero)\\
      &\forall y((y\ge a\land y\ge x)\rightarrow y=\bone)
    \end{align*}
    and this element $x$ equals~$c^{\AA}(a)$. It is called the ``complement'' of $a$.
  \item\label{alg4} For every $a \in A$, $c^\AA(c^\AA(a))=a$.
  \item\label{alg5} For all $a_0,a_1\in A$, $a_0\le^{\AA}a_1\iff c^{\AA}(a_1)\le^{\AA}c^{\AA}(a_0)$.
  \end{enumerate} 
  We denote by $<^\AA$ the strict relation associated with $\leq^\AA$,
  which is defined by $a <^\AA a'$ if and only if $a \le^\AA a'$ and
  $a' \nle^\AA a$. To simplify the notation, we write
  $\AA=(A,\le,\bzero, \bone, c)$ omitting the superscript $\AA$ when
  it is clear from the context.

  Let $L^+=L\cup \{K\}$ where $K$ is a unary predicate. Let
  $\AA=(A,\le^{\AA},\bzero^{\AA}, \bone^{\AA}, c^{\AA},K^{\AA})$ be an $L^+$-model
  such that $\AA\rest L$ is a~CA. Then we say that $\AA$ is a
  \textbf{sorted complemented algebra}, briefly
  \textbf{SCA}, if the following additional axioms hold:
  \begin{enumerate}[label={\upshape (CA\arabic*)}, leftmargin=3pc, resume]
  \item\label{alg5_} Exactly one of the following holds:
    \begin{enumerate-(i)}
    \item $\forall a\in A\big(a\in K^\AA\leftrightarrow c^\AA(a)\notin K^\AA\big)$, and there is no $\le^\AA$-largest element in $K^\AA$,\label{alg5_i}
    \item $K^\AA=A$.\label{alg5_ii}
    \end{enumerate-(i)}
  \item \label{alg6_} $\bzero^\AA\in K^\AA$ and
   $K^\AA$ is downward closed with respect to $\le^\AA$.
  \item\label{alg7_} For every $a_0\in K^\AA$ and $a_1\in A$, the following equivalence
  holds:
  $a_0<^\AA a_1$ if and only if there is $b\in K^\AA$ such that
  $a_0<^\AA b$ and the only element $\le^\AA$-below both $b$ and $c(a_1)$ is~$\bzero^\AA$. 
  \end{enumerate}
\end{Def}

If $(B,\lor,\land,\bzero,\bone,c)$ is a Boolean algebra,
then $(B,\le,\bzero,\bone,c)$ is a complemented algebra
where $\le$ is defined by
$b_0\le b_1\iff b_0\land b_1=b_0$, for all \(b_0, b_1 \in B\). Thus,
a Boolean algebra is a special case of a CA.
The converse fails as demonstrated by Example~\ref{ex:SCA} below.
The axioms \ref{alg1}--\ref{alg7_} are first-order over the
vocabulary~$L^+$ (and \ref{alg1}--\ref{alg5} are first-order
over~$L$).  The constants $\bone$, $\bzero$, and the function $c$ are
all first-order definable from $\le$, so we could have defined the CA
over the vocabulary~$\{\le\}$ and SCA over $\{\le,K\}$. We have added
the extra symbols for the sake of clarity, but see
Theorem~\ref{thm:PL-class-simple} and Corollary \ref{cor:Manifolds-class-simple} where this feature is exploited.  We now
introduce the notion of blurry filters on complemented algebras, which
is one of the main ingredients leading to a generalization of the
Stone duality.

\begin{Def}\label{def:blurry_filter}
  Let $\AA=(A,\le,\bzero,\bone,c)$ be a CA.
  A set $F\subseteq A$ is a \textbf{filter} if:
  \begin{enumerate}[label={\upshape (F\arabic*)}, leftmargin=2pc]
  \item\label{blurry_filter_1} The $\le$-maximal element is in the filter: $\bone\in F$.
  \item\label{blurry_filter_2} For all $a_0,a_1\in A$, if
    $a_0\le a_1$ and $a_0\in F$, then $a_1\in F$.
  \item\label{blurry_filter_3} For all $a_0, a_1\in F$
    there is $a_2\in F$ such that $a_2\le a_0$ and $a_2\le a_1$.   
  \end{enumerate}   
  $F$ is \textbf{proper}, if:
  \begin{enumerate}[label={\upshape (F\arabic*)}, leftmargin=2pc,resume]
  \item\label{blurry_filter_4} The minimal element is not in the filter: $\bzero\notin F$.
  \end{enumerate}
  A proper filter $F$ is \textbf{blurry}, if:
  \begin{enumerate}[label={\upshape (F\arabic*)}, leftmargin=2pc,resume]
  \item\label{blurry_filter_5} For all $a\in A$, if $a\notin F$ and $c(a)\notin F$, 
    then for all $a_0\in A$, if $a<a_0$, then $a_0\in F$.
  \end{enumerate}
  If $\AA$ is an SCA, we say that $F$ is a (proper, blurry) filter
  on $\AA$, if $F$ is a (proper, blurry) filter on $\AA\rest L$.
  It is a \textbf{$K$-filter} if:
  \begin{enumerate}[label={\upshape (F\arabic*)}, leftmargin=2pc,resume]
  \item \label{blurry_filter_6} $F\cap K\ne \es$.
  \end{enumerate}
  We denote the set of all proper blurry \(K\)-filters on \(\AA\) by \(\FF(\AA)\).
\end{Def}

\begin{Remark}\label{remark:FilterFinite}
  The property \ref{blurry_filter_3} can be iteratively applied
  to obtain a more general formulation: 
  \begin{itemize}
  \item[(F3')] For all $a_0,\dots,a_{n-1}\in F$ there is $a\in F$ such that
  $a\le a_i$ for all $i\in \{0,\dots,n-1\}$. 
  \end{itemize}
\end{Remark}

Let us examine an example highlighting
the difference between a Boolean 
algebra and a (sorted) complemented
algebra as well as between a blurry filter
and an ultrafilter.

\begin{exmp}\label{ex:SCA}
  We will first define a CA and then turn it into an SCA. 
  Let $A_0$ be the set of bounded non-empty 
  open intervals of $\R$ with rational endpoints,
  and let 
  $$A=A_0\cup\{\R\setminus \bar a\mid a\in A_0\}\cup\{\es,\R\}$$
  where $\bar a$ is the closure of $a$. Let $\le^\AA$ be defined by
  $a\le a'$ if and only if either $a=a'$ or the closure of $a$ is
  contained in $a'$.  Let
  $c^\AA(a)=\R\setminus \bar a$.  
  Let $\bone^\AA = \R$
  and $\bzero^\AA=\es$.  Let $\AA=(A,\le^\AA,\bzero^\AA,\bone^\AA,c^\AA)$. 
  One easily
  checks conditions \ref{alg1}--\ref{alg5}, so $\AA$ is a CA. 
  It is not a Boolean algebra for the following
  reason. Let $x=(0,2)$ and $y=(1,3)$. Then
  there is no $\le^\AA$-greatest lower bound of $x$ and $y$:
  every element $\le^\AA$-below both $x$ and $y$
  is of the form $z=(\alpha,\beta)$ where
  $1<\alpha<\beta<2$, so 
  $z'=((1+\alpha)/2,(2+\beta)/2)$ is $\le^\AA$-greater
  than $z$ but still below both $x$ and $y$.
  Similarly, $x$ and $y$ do not have a least upper bound.
  As usual, call a
  filter $F\subseteq A$ an \textbf{ultrafilter}, if it satisfies
  conditions \ref{blurry_filter_1}--\ref{blurry_filter_4} together
  with a stronger version of condition~\ref{blurry_filter_5}, namely
  \begin{itemize}
  \item[(F5')] For all $a\in A$, either $a\in F$ or $c(a)\in F$.
  \end{itemize}
  Fix $x\in\R$ and let $F=\{a\in A\mid x\in a\}$. Then $F$ is a proper
  blurry filter. If $x$ is irrational, then $F$ is an ultrafilter. If
  $x$ is rational, then $F$ is not an ultrafilter, because then
  $a=(x,x+1)\notin F$, but also $\R\setminus [x,x+1]=c(a)\notin F$.
  We can make $\AA$ into an SCA by interpreting $K$ either as $K=A$,
  or as $K=A_0\cup \{\es\}$.  In both cases it is easy to verify
  conditions \ref{alg5_}--\ref{alg7_}.  In both cases $F$, as defined
  above, is a $K$-filter, because $(\alpha,\beta)\in F\cap K$ for all
  $\alpha,\beta\in\Q$ such that $\alpha<x<\beta$.
\end{exmp}

\begin{Lemma}\label{lemma:bonecbzero}
  Let $\AA=(A,\le,\bzero,\bone, c)$ be a CA. Then $c(\bone)=\bzero$
  and $c(\bzero)=\bone$.
\end{Lemma}
\begin{proof}
  By \ref{alg2}, for any $a$ we have $c(a)\le \bone$,
  so by \ref{alg5}, $c(\bone)\le c(c(a))=a$ where the
  latter equality is by \ref{alg4}. By \ref{alg2},
  $c(\bone)=\bzero$. Symmetrically, $c(\bzero)=\bone$.
\end{proof}

Next we will show that the isomorphism
type of the restricted structure $(K^\AA,\le^\AA)$
completely determines the isomorphism type of~$\AA$.

\begin{Prop}\label{prop:IsoOnKIsEnough}
  Let $\AA=(A,\le^{\AA},\bzero^{\AA},\bone^{\AA},c^{\AA},K^{\AA})$ and $\BB=(B,\le^{\BB},\bzero^{\BB},\bone^{\BB},c^{\BB},K^{\BB})$
  be two sorted complemented algebras. The partial orders
  $(K^\AA,\le^\AA)$ and $(K^\BB,\le^\BB)$ are isomorphic if and only if $\AA\cong\BB$.
\end{Prop}
\begin{proof}
  For the non-obvious direction, let $f\colon (K^\AA,\le^\AA)\to (K^\BB,\le^\BB)$ be an isomorphism.
  We will show that it extends to an isomorphism
  $\hat f$ from $\AA$ to $\BB$. We will
  divide the analysis into cases depending on 
  which one of the conditions in \ref{alg5_} is 
  satisfied. If $K^\AA$ satisfies \ref{alg5_}\ref{alg5_i} in $\AA$,
  then it has no largest element in $K^{\AA}$. Since $f$
  is an isomorphism, $K^\BB$ has no largest element, so
  $\BB\models \bone^\BB\notin K^\BB$. Therefore
  $K^\BB\ne B$ 
  and hence also satisfies \ref{alg5_}\ref{alg5_i} (in $\BB$).
  By symmetry, $K^\AA$ satisfies \ref{alg5_}\ref{alg5_i}
  if and only if $K^\BB$ satisfies \ref{alg5_}\ref{alg5_i}.
  Let us call \textbf{case (i)} the case in which
  both satisfy \ref{alg5_}\ref{alg5_i} and \textbf{case (ii)}
  in which both satisfy \ref{alg5_}\ref{alg5_ii}.
  
  In case (ii), let $\hat f:=f$. 
  Otherwise, in case (i), define $\hat f$ by
  \begin{equation}\label{eq:defofhatf}
  \hat f(x)=
  \begin{cases}
      f(x),&\text{ if }x\in K^\AA,\\
      c^\BB(f(c^\AA(x))&\text{ if }x\notin K^\AA.
  \end{cases}
  \end{equation}
  By \ref{alg5_}\ref{alg5_i}, $\hat f$ is well-defined.
  We claim
  that $\hat f$ is an isomorphism from $\AA$ to~$\BB$. 
  First observe
  that \ref{alg5_} guarantees that for all $a\in A$,
  \begin{equation}
    a\in K^{\AA}\text{ if and only if }\hat f(a)\in K^{\BB}.\label{eq:iffinK}
  \end{equation}
  Thus, $\hat f$ preserves
  the predicate~$K$.
  Let us next show that $\hat f$ is one-to-one. 
  In case (ii), it is clear. Assume case (i).
  Injectivity on $K^\AA$ follows because $\hat f\rest K^\AA=f$ is an isomorphism. Suppose $x\in K^{\AA}$ and
  $y\notin K^{\AA}$. Then it follows from \eqref{eq:iffinK}
  that $\hat f(x)\ne \hat f(y)$.
  Suppose $x,y\in A\setminus K^{\AA}$ and assume
  that $\hat f(x)=\hat f(y)$. Then applying \ref{alg4} we have
  $$c^{\BB}(f(c^{\AA}(x)))=\hat f(x)=\hat f(y)=c^{\BB}(f(c^{\AA}(y))) \ \Longrightarrow\ f(c^{\AA}(x))=f(c^{\AA}(y))$$
  which, by the fact that $f$ is an isomorphism, implies
  $c^{\AA}(x)=c^{\AA}(y)$, and applying \ref{alg4} again, we have
  $x=y$. Thus $\hat f$ is one-to-one. It is easy to see that \(\hat f\) is also surjective. Next, let us show that $f$ preserves the
  function~$c$. First consider case (ii).
  Let $x\in A$.
  The elements $\bzero$, $\bone$,
  and therefore by \ref{alg3} also $c^\AA(x)$, 
  are definable from $x$ 
  using a first-order
  formula over the vocabulary~$\{\le\}$. 
  Since $f$ is an isomorphism between the $\{\le\}$-structures 
  $(K^\AA,\le^\AA)$ and $(K^\BB,\le^\BB)$, 
  the element $c^\BB(f(x))$ is definable from $f(x)$ using
  the same formula. By the uniqueness assumption in \ref{alg3},
  $c^\BB(f(x))$ must be equal to $f(c^\AA(x))$.
  Consider now case (i). Assume $x\in K^\AA$.
  Then by \ref{alg5_}\ref{alg5_i}, $c^\AA(x)\notin K^\AA$, so we have
  $$\hat f(c^{\AA}(x))\stackrel{\eqref{eq:defofhatf}}{=}c^{\BB}( f(c^{\AA}(c^{\AA}(x))))\stackrel{\text{\ref{alg4}}}{=}c^{\BB}(f(x))\stackrel{\eqref{eq:defofhatf}}{=}c^\BB(\hat f(x)).$$
  On the other hand, suppose $x\notin K^{\AA}$. Then $c^{\AA}(x)\in K^{\AA}$, so  
  $$\hat f(c^{\AA}(x))\stackrel{\eqref{eq:defofhatf}}{=}f(c^{\AA}(x))\stackrel{\text{\ref{alg4}}}{=}c^{\BB}(c^{\BB}(f(c^{\AA}(x))))\stackrel{\eqref{eq:defofhatf}}{=}c^{\BB}(\hat f(x)).$$
  Next, let us show that $\bzero$ and $\bone$ are preserved.
  Since $\bzero^{\AA}\in K^{\AA}$ and $\bzero^{\BB}\in K^{\BB}$ (condition~\ref{alg6_})
  are the $\le_{\AA}$- and $\le_{\BB}$-minimal elements respectively, it follows
  that $\hat f(\bzero^{\AA})=f(\bzero^{\AA})=\bzero^{\BB}$
  by using that $f$ is an isomorphism.
  By Lemma~\ref{lemma:bonecbzero}, and by the fact that 
  $\hat f$ preserves $c$,
  we have
  $$\hat f(\bone^\AA)=\hat f(c^{\AA}(\bzero^{\AA}))=c^{\BB}(f(\bzero^{\AA}))=c^{\BB}(\bzero^{\BB})=\bone^{\BB}.$$
  It remains to show that $\hat f$ preserves the relation
  $\le$. Suppose $x,y\in \AA$ and $x\le y$.  There are three options:
  (a) $x,y\in K^\AA$, (b) $x,y\notin K^\AA$, and (c) $x\in K^\AA$,
  $y\notin K^\AA$, so in particular $x<y$; the fourth option with
  $x\notin K^\AA$ and $y\in K^\AA$ is excluded by \ref{alg6_}.  In
  cases (a) and (b), the preservation of $\le$ follows from
  \eqref{eq:defofhatf} and \ref{alg5}. In case (c), by
  \ref{alg7_} there is $z\in K^\AA$ such that $x<^\AA z$ and the only
  element $\le^\AA$-below both $z$ and $c^\AA(y)$ is
  $\bzero^\AA$. Since $x,z,c(y)\in K^\AA$, this property is preserved
  by $f$, so $f(z)$ witnesses the right side of the ``if and only if'' of \ref{alg7_} in
  $\BB$, so we have $f(x)<^\BB f(y)$.
\end{proof}
\vspace{10pt}

\subsection{Basis spaces}
\label{ssec:BasisSpaces}

\begin{defn}\label{def:LCBS}
  A \textbf{basis space} is a pair $(X,\beta)$, where $X$ is a set and
  $\beta\subseteq\PP(X)$ is any set of subsets.
  Sometimes, if $\beta$ is countable,
  we assume that it is indexed by natural numbers and $\beta=(b_n)_{n\in\N}$.
  If no confusion ensues, we use both the set and the sequence representation interchangeably.
  We say that two basis spaces $(X,\beta)$ and $(X',\beta')$, are
  \textbf{equivalent}, and write $(X,\beta)\equiv (X',\beta')$, if
  there is a bijection $h\colon X\to X'$ such that for all $b\in\beta$
  there is $b'\in \beta'$ such that $h[b]=b'$ and for all $b'\in\beta'$ 
  there is $b\in\beta$ such that $h^{-1}[b']=b$.
\end{defn} 

Given a basis space $(X,\beta)$, we denote by $\la\beta\ra$ the
topology generated by~$\beta$ on \(X\), i.e. the smallest topology
in which each $b\in\beta$ is open.  In the sequel, all the operations of closure and
interior for subsets of a basis space 
$X$ are defined using this topology. In
particular, we say that a subset $C\subseteq X$ is \textbf{compact} if
it is compact in the subspace topology induced 
by~$\la\beta\ra$.  The following is an immediate
consequence of the definitions.

\begin{Fact}\label{fact:Homeo}
  Suppose $h$ is a map witnessing 
  the equivalence between $(X,\beta)$ and $(X',\beta')$.
  Then $h$ is a homeomorphism from $(X,\la\beta\ra)$ to $(X',\la\beta'\ra)$. \qed
\end{Fact}

\begin{exmp}\label{ex:BasisSpacesNonEq}
  Let $X=\R$ and let $A_0$ be as in Example~\ref{ex:SCA}.
  Let $\beta=A_0\cup\{\es\}$. Then $(X,\beta)$ is a basis space.
  Let $\beta'=\{a\cup b\mid a,b\in \beta\}$. Then $(X,\beta')$
  is also a basis space. The topologies generated by $\beta$
  and $\beta'$ coincide: the sets in $\beta'$ are open in 
  $\la\beta\ra$ on the one hand and $\beta\subseteq\beta'$ on the other. 
  In fact, both $\beta$ and $\beta'$ are bases for
  the standard topology on~$\R$. However, 
  $(X,\beta)\not\equiv (X,\beta')$. To see this, let $h$
  be a bijection $\R\to\R$ witnessing the counterassumption
  that they are equivalent.
  By Fact~\ref{fact:Homeo}, $h$ is a homeomorphism.
  Therefore $h^{-1}[(0,1)\cup (2,3)]$ must be disconnected
  and bounded in~$\R$. But there is no disconnected
  and bounded set in $\beta$ while 
  $(0,1)\cup (2,3)\in \beta'$, a contradiction.
\end{exmp}

A \textbf{Polish space} is a topological space which is separable and admits a compatible complete metric.
For our purposes, let us define a special type of basis spaces.
\begin{Def}\label{def:CLCP}
  A basis space $(X,\beta)$ is \textbf{countable complemented locally compact
  Polish}, briefly \textbf{CCLCP}, if 
    $\beta=(b_n)_{n\in\N}$ is countable and satisfies the following
  conditions:
  \begin{enumerate}[label={\upshape (B\arabic*)}, leftmargin=2pc]
  \item \label{def:BC3} $(X,\la\beta\ra)$ is a Polish space.
  \item \label{def:BC1} For all $n \in \N$ there is $m \in \N$
    such that $X\setminus \overline{b_n}=b_m$.
  \item \label{def:BC8} For all $n \in \N$, either $\overline{b_n}$ or $X\setminus b_n$ is compact (possibly both are compact).
  \item \label{def:BC6} For each $x\in X$ there is $n \in \N$ such
    that $x\in b_n$ and $\overline{b_n}$ is compact.
  \item \label{def:BC2} $\es\in \beta$ and $X\in\beta$.
  \item \label{def:BC7} For each $n \in \N$, we have
    $X\setminus b_n=\overline{X\setminus \overline{b_n}}$.
  \item \label{def:BC4} For all $n\ne m$, $b_n\ne b_m$.
  \item \label{def:BC9} For all $n,m\in \N$, if $\overline{b_n}$ is
    compact and $\overline{b_n}\subset b_m$, then there is $k$ such
    that $\overline{b_k}$ is compact, $\overline{b_n}\subset b_k$, and
    $\overline{b_k}\subset b_m$.
  \item \label{def:BC5} For all $n_1,\dots,n_k\in\N$ and all
    $x\in b_{n_1}\cap\cdots\cap b_{n_k}$ there is $m \in \N$ such that
    $x\in b_m$ and
    $\overline{b_m}\subseteq b_{n_1}\cap\cdots\cap b_{n_k}$.
  \end{enumerate}
\end{Def}

\begin{Remark}\label{remark:RegStr}
  Notice that condition \ref{def:BC2}, which guarantees that
  \(\Cup_{n \in \N} b_n=X\), and condition \ref{def:BC5} imply that
  \(\beta\) is a basis for the topology generated by \(\beta\).  
\end{Remark}

A subset of a topological space is \textbf{regular open} if it is equal to the interior of its
closure and \textbf{regular closed} if it is equal to the closure of
its interior.

\begin{Fact}\label{fact:InteriorReg}
  Let $X$ be a topological space. The interior of any closed subset of $X$ is regular open and the closure of
  any open subset of $X$ is regular closed. Moreover condition~\ref{def:BC7} is
  equivalent to $b_n$ being regular open.
\end{Fact}
\begin{proof}
  This is a standard exercise but we will give the proof for the sake
  of completeness. Let $A\subset X$ be a closed set and let $B=\int(A)$. Since $B\subset \overline{B}$, we have
  $\int B\subset\int(\overline{B})$.  But $B$ is open, so $\int(B)=B$ and so
  $B\subset \int(\overline{B})$. On the other hand $B\subset A$ and since $A$
  is closed, $\overline{B}\subset A$, so $\int(\overline{B})\subset A$.  Since the
  former is open, it follows that $\int(\overline{B})\subset\int(A)=B$.  This
  proves that $B$ is regular open. 
  
  Suppose $A$ is open and $C=\overline{A}$. We want to show that $C=\overline{\int(C)}$. Since $\int(C)\subset C$, we have
  $\overline{\int(C)}\subset \overline{C}=C$. Since $A\subset C$,
  $A=\int(A)\subset \int(C)$. Taking closure on both sides, we have
  $\overline{A}\subset \overline{\int(C)}$, but $\overline{A}=C$. This proves that
  $C$ is regular closed. 
  
  Let us prove the last statement of the fact.
  Suppose $b_n$ is regular open. Then $x\notin X\setminus b_n$ iff
  $x\notin X\setminus \int(\overline{b_n})$ iff $x\in\int(\overline{b_n})$ iff $x$
  has a neighborhood inside $\overline{b_n}$ iff $x$ is not in the
  closure of $X\setminus \overline{b_n}$. For the other direction, suppose
  that $b_n$ is not regular and let
  $x\in \int(\overline{b_n})\setminus b_n$.  Then $x\in X\setminus b_n$, but
  $x\notin X\setminus \int(\overline{b_n})$, so it is in $\int(\overline{b_n})$
  and has a neighborhood inside $\overline{b_n}$.  So it is not in the
  closure of $X\setminus \overline{b_n}$. Thus,
  $X\setminus b_n\ne \overline{X\setminus\overline{b_n}}$.
\end{proof}

\begin{Lemma}\label{lemma:FindNinC}
    Suppose $(X,\beta)$ is a CCLCP basis space.
    Let $x\in X$ and $C\subset X$ be
    a closed neighborhood of $x$ in $\la\beta\ra$.
    Then there is $n\in \N$ with 
    $x\in b_n\subset C$. 
\end{Lemma}
\begin{proof}
    Let $U=\int(C)$. Then $U$ is an open neighborhood
    of~$x$.
    Since $\beta$ is a basis,
    $U$ is the union of some sets $b_{n_1}, b_{n_2},\ldots$,
    so $x\in b_{n_k}\subset U \subset C$ for some $k$. 
\end{proof}

\begin{Lemma}\label{lemma:Infinite}
    Suppose $(X,\beta)$ is a CCLCP basis space. Then $X$ is infinite.
\end{Lemma}
\begin{proof}
  From Condition \ref{def:BC4} and the definition of $\beta$
  as a sequence indexed by natural numbers, it follows
  that $\beta$ has infinitely many distinct sets.
  Therefore $X$ is infinite.  
\end{proof}
  


\vspace{10pt}

\subsection{Blurry duality theorem}
\label{ssec:Duality}


We define a map which connects basis spaces with sorted complemented algebras.
Recall that $L^+=\{\le,\bzero,\bone,c,K\}$ is the first-order vocabulary of
SCA (Definition~\ref{def:algebra}).

\begin{Def}\label{def:map_from_BBB_to_AAA}
  Given a CCLCP basis space $(X,\beta)$, let $\psi(X,\beta)$ be the $L^+$-model
  $$\AA=(A,\le^{\AA},\bzero^{\AA},\bone^{\AA},c^{\AA},K^{\AA})$$
  defined as follows:
  \begin{enumerate}[label={\upshape (D\arabic*)}, leftmargin=2.3pc]
  \item \label{def:Psi_item0} $A=\N$,
  \item \label{def:Psi_item5} $K^\AA=\{n \in \N\mid\overline{b_n}\text{ is compact}\}$.
  \end{enumerate}
  For all $n,m\in \N$:
  \begin{enumerate}[label={\upshape (D\arabic*)}, leftmargin=2.3pc,resume]
  \item \label{def:Psi_item1} $n\le^\AA m\iff b_n=b_m\lor \overline{b_n}\subseteq b_m$,
  \item \label{def:Psi_item2} $\bone^\AA=n\iff b_n=X$,
  \item \label{def:Psi_item3} $\bzero^\AA=n\iff b_n=\es$,
  \item \label{def:Psi_item4} $c^\AA(n)=m \iff X\setminus \overline{b_n}=b_m$.
    \end{enumerate}
\end{Def}


\begin{Lemma}\label{lemma:CCLCP_is_AAA}
  For any CCLCP basis space $(X,\beta)$, the $L^+$-model 
  $\AA=\psi(X,\beta)$ is well-defined and is a sorted complemented algebra.
\end{Lemma}
Below \ref{def:BC3}--\ref{def:BC5}
refer to Definition~\ref{def:CLCP}, 
\ref{def:Psi_item0}--\ref{def:Psi_item4} refer to Definition \ref{def:map_from_BBB_to_AAA}, and~\ref{alg1}--\ref{alg7_}
to Definition~\ref{def:algebra}.
\begin{proof}
  By \ref{def:BC2} and \ref{def:BC4}, 
  there are unique $n$ and $m$ 
  such that $b_n=\es$ and $b_m=X$, so $\bzero^\AA$ and $\bone^\AA$ are well-defined
  by \ref{def:Psi_item2} and \ref{def:Psi_item3}. By \ref{def:BC1} and 
  \ref{def:BC4}, for each $n$ there is a unique $m$ such that 
  $b_m=X\setminus \overline{b_n}$,
  so $c^\AA$ is also well-defined by \ref{def:Psi_item4}. This proves that $\psi(X,\beta)$
  is well-defined. It remains to check that \ref{alg1}--\ref{alg7_}
  are satisfied. 

  Reflexivity, transitivity, and antisymmetry of $\le^\AA$ clearly follow from
  \ref{def:Psi_item1} and \ref{def:BC3}.
  This proves \ref{alg1}.
  We already covered the existence and
  uniqueness of $\bone^\AA$, $\bzero^\AA$. That they are
  respectively $\le^\AA$-maximal and minimal follows from the
  definition of $\le^\AA$ (note that $\overline{\es}=\es$). This proves~\ref{alg2}.  

  Before \ref{alg3}, we will prove \ref{alg4} and \ref{alg5}.
  Let $n\in A$.
  Then by \ref{def:BC7}
  and \ref{def:Psi_item4}, we have
  $$m=c^\AA(c^\AA(n))\iff b_m=X\setminus\overline {X\setminus \overline{b_n} }=X\setminus (X\setminus b_n)=b_n$$
  which by \ref{def:BC4} implies $n=m$
  proving condition \ref{alg4}.
    For \ref{alg5}, suppose $n\le^\AA m$. Then $\overline{b_n}\subset b_m$,
  so $X\setminus b_m\subset X\setminus \overline{b_n}$.
  Now \ref{def:BC7} implies that $\overline{X\setminus \overline{b_m}}\subset X\setminus \overline{b_n}$ which means $c^{\AA}(m)\le^\AA c^\AA(n)$.
  
  To prove \ref{alg3}, let $n\in A$ be arbitrary. 
  By \ref{def:BC1} there is $m$ such that $b_m=X\setminus \overline{b_n}$
  and by \ref{def:BC4} it is unique. By \ref{def:Psi_item4}, 
  $m=c^{\AA}(n)$.
  Let us show that $m$ satisfies the formulas of \ref{alg3}.
  The first formula 
  $\forall y((y\le n\land y\le m)\rightarrow y=\bzero)$
  is satisfied because $b_n\cap b_m=\es$, and the second one
  $\forall y((n\le y\land m\le y)\rightarrow y=\bone)$
  is satisfied because $\overline{b_n}\cup \overline{b_m}=X$.
  On the other hand, if $m$ satisfies the first formula, then
  $b_n\cap b_m=\es$ by \ref{def:BC5}. Analogously, if $m$ satisfies the second formula, then $\overline{b_n}\cup \overline{b_m}=X$ because
  otherwise $c^{\AA}(n)$ and $c^{\AA}(m)$ 
  satisfy the first formula (using \ref{alg5}) but $(X\setminus \overline{b_n})\cap (X\setminus\overline{b_m})\ne\es$ which by \ref{def:Psi_item4} is a contradiction. 
  So if $m'$ is another element satisfying both formulas, 
  then $b_{m'}\cap b_n=\es$ implies $m'\le^\AA m$ and
  $\overline{b_{m'}}\cup\overline{b_n}=X$ implies $m\le^{\AA}m'$,
  so $b_m=b_{m'}$ which by \ref{def:BC4} implies $m=m'$. 
  
  Condition \ref{alg6_} follows from \ref{def:Psi_item5},
  because the empty set is compact and
  a closed subset of a compact set
  in a Polish space (see \ref{def:BC3}) is compact.
  For \ref{alg5_}, if $(X,\la \beta\ra)$ is compact, 
  then by \ref{def:Psi_item5}, $K^\AA=A=\N$, 
  and all elements in $\beta$ are pre-compact,
  so \ref{alg5_}\ref{alg5_ii} holds. 
  Otherwise, by \ref{def:BC8} and \ref{def:Psi_item4}, 
  exactly one of $n$ and $c^\AA(n)$ is an element of
  $K^\AA$, for each $n \in \N$. Let us show that in this case there is no largest element 
  in
  $K^\AA$. Let $n\in K^\AA$. Note
  that $n\ne \bone^\AA$, because $\bone^\AA\notin K^\AA$ by \ref{def:Psi_item2} and the fact that $(X,\langle \beta\rangle)$ is not compact. Thus, $n<^\AA \bone^\AA$.
  Now, apply \ref{def:BC9} to $n$ and $\bone^\AA$ to choose $k$ such that $\overline{b_k}$ is compact and
  $\overline{b_n}\subset b_k$. Then $k\in K^\AA$ and 
  $n\le^\AA k$. 
  Finally, let us show \ref{alg7_}. The direction from left-to-right
  follows from \ref{def:BC9}. Suppose that the right-hand-side holds
  for some $n_0\in K^\AA$ and $n_1\in A$, i.e. there is $m\in K^\AA$
  such that $n_0<^\AA m$ and the only element $\le^\AA$-below $m$
  and $c(n_1)$ is $\bzero^\AA$. This implies
  that $\overline{b_m}$ is compact and $\overline{b_{n_0}}\subset b_m$
  (using \ref{def:Psi_item5} and \ref{def:Psi_item1}). By
  \ref{def:BC5}, we have
  $b_m\cap (X\setminus \overline{b_{n_1}})=\es$. 
  But $b_m$ is open, so $b_m\cap \overline{X\setminus \overline{b_{n_1}}}=\es$.
  By \ref{def:BC7}, $b_m\cap (X\setminus b_{n_1})=\es$,
  so 
  $b_m \subset b_{n_1}$. 
  Since $\overline{b_{n_0}}\subset b_m$, we have 
  $n_0<^\AA n_1$ by applying \ref{def:Psi_item1}.
\end{proof}

Next, let us define a map $\f$ which attaches a basis space 
to a given sorted complemented algebra. We will show in 
Theorem~\ref{thm:BlurryDuality} that this $\f$ is a left inverse of~$\psi$.

\begin{Def}\label{def:phi_map}
  Given a sorted complemented algebra $\AA$, 
  let $X_\AA=\FF(\AA)$ be the set of all blurry
  $K$-filters on $\AA$ (Definition~\ref{def:blurry_filter}),
  and let $\beta_{\AA}=\{b[a]\mid a\in A\}$, where
  $$b[a]:=\{F\in\FF(\AA)\mid a\in F\}.$$ Then 
  define $\f(\AA)=(X_{\AA},\beta_{\AA})$.
\end{Def}

\begin{Lemma}\label{lem:from_alg_to_basisspace}
  If $\AA$ is a sorted complemented algebra, then $\f(\AA)$ is a basis space.
\end{Lemma}
\begin{proof}
   Clearly $\beta_\AA\subseteq \PP(X_\AA)$.
\end{proof}


\begin{Lemma}\label{lemma:atmost1}
  Suppose that $(X,\beta)$ is a CCLCP basis space, with $\beta=(b_n)_{n \in \N}$. Let
  $\AA=\psi(X,\beta)=(A,\le^\AA,\bzero^\AA,\bone^\AA,c^\AA,K^\AA)$,
  and assume that $F$ is a blurry
  $K$-filter on~$\AA$.
  Let $F^*=\{b_n\mid n\in F\}$ and $\overline{F^*}=\{\overline{b_n}\mid n\in F\}$.
  Then the following hold:
  \begin{enumerate-(i)}
  \item \label{lemma:atmost1:0} $\Cap F^*\subseteq \Cap \overline{F^*}$.
  \item \label{lemma:atmost1:1} There is at most one element in the intersection $\Cap\overline{F^*}$.
  \item \label{lemma:atmost1:2} The intersection $\Cap F^*$ is non-empty. 
  \item \label{lemma:atmost1:3} $F$ is of the form 
   $\{n\in A\mid x\in b_n\}$ for some~$x\in X$.
  \end{enumerate-(i)}
\end{Lemma}

Below \ref{blurry_filter_1}--\ref{blurry_filter_6}
refer to Definition~\ref{def:blurry_filter}.

\begin{proof}
  Recall that $A=\N$. 
  Observe that the following holds for all
  $n \in \N$:
  \begin{equation}
    \label{eq:Complementation}
    b_{c^\AA(n)}=X\setminus \overline{b_{n}}\quad\text{ and }\quad \overline{b_{c^\AA(n)}}=X\setminus b_{n},
  \end{equation}
  the first part by \ref{def:Psi_item4}, and the second part follows
  from \ref{def:BC7} (regularity).
  As before, denote $n<^\AA m$, if $n\le^\AA m$ and $n\ne m$.
  \begin{enumerate-(i)}
  \item This follows immediately from the definitions, because $b_n\subset \overline{b_n}$ for each \(n \in \N\).
  \item Suppose $x_0\ne x_1$ and $x_0,x_1\in \Cap\overline{F^*}$. 
    Since $X$ is infinite
    (Lemma~\ref{lemma:Infinite}), there is $x_2\in X$ distinct from both $x_0$ and~$x_1$. Using \ref{def:BC3}, pick disjoint closed neighborhoods 
    $C_0$, $C_1$, and $C_2$ of $x_0$, $x_1$, and $x_2$ respectively
    in the topology generated by $\beta$. 
    By \ref{def:BC6}, Lemma \ref{lemma:FindNinC} and \ref{def:BC5},
    there are indices
    $n_0,n_1,n_2\in \N$ such that $x_k\in b_{n_k}\subseteq C_k$
    and $\overline{b_{n_k}}$ is compact for
    $k\in\{0,1,2\}$. By \eqref{eq:Complementation}
    we have for all $i,j\in \{0,1,2\}$:
    \begin{equation}
        x_i\in \overline{b_{n_j}}
        \iff x_i\notin 
        X\setminus b_{n_j} 
        \stackrel{\eqref{eq:Complementation}}{\iff} x_i\notin \overline{b_{c^\AA(n_j)}}
        \iff i=j.\label{eq:x0x1x2}
    \end{equation}
    By the assumption, $x_0$ and $x_1$ belong to all elements of
    $\overline{F^*}$. Since $x_1\notin \overline{b_{n_0}}$, $\overline{b_{n_0}}$ cannot be an element
    of $\overline{F^*}$. Therefore $b_{n_0}$ is not an element of
    $F^*$ and $n_0$ is not an element of $F$. We have
    established:
    \begin{equation*}
      \label{eq:n0_notin_F}
      n_0\notin F.
    \end{equation*}
    On the other hand, $x_0\notin \overline{b_{c^\AA(n_0)}}$, so 
    by the same chain of inferences, we also have:
    \begin{equation*}
      \label{eq:cn_1_notin_F}
      c^\AA(n_0)\notin F.
    \end{equation*}
    Since $\overline{b_{n_0}}\cap \overline{b_{n_1}}=\es$,
    we have $\overline{b_{n_0}}\subseteq X\setminus\overline{b_{n_1}}=b_{c^\AA(n_1)}$.
    By \ref{def:Psi_item1} we now have
    $n_0\le^\AA  c^\AA(n_1)$.  
    On the other hand $b_{n_0}\ne b_{c^\AA(n_1)}$
    because $x_2\in b_{c^\AA(n_1)}\setminus b_{n_0}$. 
    This implies that $n_0\ne c^\AA(n_1)$, and so
    we have established:
    \begin{equation*}
      \label{eq:n_0_ne_cn_1}
      n_0 <^\AA c^\AA(n_1).
    \end{equation*}
    By \ref{blurry_filter_5}
    we must have $c^\AA(n_1)\in F$. But by
    \eqref{eq:x0x1x2}, $x_1\notin \overline{b_{c^\AA(n_1)}}$, a contradiction.
  \item Let $n_K\in F\cap K$ which exists by \ref{blurry_filter_6}.
  Then by \ref{def:Psi_item5} $\overline{b_{n_K}}$ is compact.
    Let $Z=\{\overline{b}\cap \overline{b_{n_K}}\mid b\in F^*\}$.
    Let us show that $\Cap Z$ is non-empty. 
    If it is empty, then by 
    the compactness of $\overline{b_{n_K}}$ there
    are $n_0,\dots,n_{k-1}\in F$ such that
    \[
      \overline{b_{n_0}}\cap\cdots\cap \overline{b_{n_{k-1}}}\cap \overline{b_{n_K}}=\es,
    \]
    implying
    \[
      b_{n_0}\cap\cdots\cap b_{n_{k-1}}\cap b_{n_K}=\es.
    \]
    By \ref{blurry_filter_3} (see also Remark \ref{remark:FilterFinite}) there is $j\in F$ such that
    $j\le^{\AA} n_i$ for all $i<k$, and $j\le^{\AA} n_K$. 
    Then by \ref{def:Psi_item1}
    $b_j\subseteq \overline{b_j}\subseteq b_{n_0}\cap\cdots\cap b_{n_{k-1}}\cap b_{n_K}=\es$, so $b_j=\es$, and hence by \ref{def:Psi_item3} $j=\bzero^\AA$.
    This contradicts the properness of~$F$, \ref{blurry_filter_4}. 
    Thus, we have established that $\Cap Z\ne\es$.
    
    So fix $x\in \Cap Z$. We want to show that 
    $x\in\Cap F^*$. To this end, let $m_0\in F$ be arbitrary.
    It is now enough to prove that $x\in b_{m_0}$. By the definition of
    $Z$ we know that
    \begin{equation*}
      \label{eq:xinbm0stuff}
      x\in \overline{b_{m_0}}\cap \overline{b_{n_K}}\subseteq\overline{b_{m_0}}.
    \end{equation*}
    If $b_{m_0}=\overline{b_{m_0}}$, we are done. Otherwise $b_{m_0}$ is open and
    not closed, so it cannot be a singleton (using \ref{def:BC3}), and there must be
    $x_0,x_1\in b_{m_0}$ with $x_0\ne x_1$.  But then there is 
    $m_1\in F$ such that either $x_0\notin b_{m_1}$ or $x_1\notin b_{m_1}$ because otherwise
    $\{x_0,x_1\}\subseteq \Cap F^*\subseteq \Cap \overline{F^*}$, contradicting
    \ref{lemma:atmost1:1}. In particular 
    \begin{equation}
      m_0\not\le^\AA  m_1.\label{eq:m0m1}
    \end{equation}  
    Let $m_2\in F$ be such that
    $m_2\le^\AA  m_0$ and $m_2\le^\AA  m_1$ which exists by 
    \ref{blurry_filter_3}. By \eqref{eq:m0m1} 
    $m_2\ne m_0$, so we have established that
    $m_2<^\AA m_0$ and hence $\overline{b_{m_2}}\subseteq b_{m_0}$.    
    By the definition
    of $Z$ and the choice of $x$,
    we now have 
    $$
    x\in\overline{b_{m_2}}\cap \overline{b_{n_K}}\subseteq \overline{b_{m_2}}\subseteq b_{m_0},
    $$
    which was to be proven.
  \item By \ref{lemma:atmost1:0}, \ref{lemma:atmost1:1}, and \ref{lemma:atmost1:2} there is $x$ such that
    \begin{equation}
      \label{eq:singleton}
      \Cap F^*=\Cap \overline{F^*} =\{x\}.
    \end{equation}
    Thus $F^*\subseteq \{b_n \mid x\in b_n\}$, or equivalently
    $F\subseteq \{n\in\N\mid x\in b_n\}$.
    Let us show the converse, namely that
    $\{n\in\N\mid x\in b_n\}\subseteq F$ which will complete the proof.
    Suppose $n$ is such that $x\in b_n$, but $n\notin F$. 
    Then $x\notin b_{c^\AA(n)}$ and so we also have 
    $c^\AA(n)\notin F$.
    Suppose there is $k$ such that $x\in b_k$ and $n\not\le^\AA  k$. 
    By \ref{def:BC5} find $m\in\N$ such that 
    $x\in b_m$ and $\overline{b_m}\subseteq b_n\cap b_k$. In particular
    $m\le^\AA  n$ and  $m\le^\AA  k$. By the assumption $n\not\le^\AA  k$
    we also have $m\ne n$, so $m<^\AA n$.
    By \ref{def:BC7}, $c^\AA(n)<^\AA c^\AA(m)$.
    Applying \ref{blurry_filter_5} to $c^\AA(n)$ and \ref{alg4} we conclude that $c^\AA(m)\in F$.
    This is a contradiction because $x\notin b_{c^\AA(m)}$.
    So there is no $k$ with $x\in b_k$ and 
    $n\not\le^\AA  k$. This means that every element of the basis
    $\beta$ which contains $x$ must also contain $\overline{b_n}$,
    which means by \ref{def:BC5} that $\overline{b_n}$ is the singleton $\overline{b_n}=\{x\}$,
    and $b_n=\overline{b_n}$. In particular $x$ is an isolated point
    and $\{x\}$ is clopen. Consequently we also
    have $b_{c^\AA(n)}=\overline{b_{c^\AA(n)}}=X\setminus \{x\}$.
    By \eqref{eq:singleton} we have that
    \begin{equation}
        \overline{b_{c^\AA(n)}}\cap \Cap \overline{F^*}=\es.\label{eq:emptyinter}
    \end{equation}
    As in the previous proof, let $n_K\in F\cap K$ be an index
    such that $\overline{b_{n_K}}$ is compact. Then from \eqref{eq:emptyinter} we have
    $$\overline{b_{c^\AA(n)}}\cap \overline{b_{n_K}}\cap\Cap \overline{F^*}=\es.$$
    This can be re-written as
    $$\Cap Z=\es$$
    where
    $$Z=\{\overline{b_{c^\AA(n)}}\cap \overline{b_{n_K}}\cap\overline{b_m}\mid m\in F\}$$
    But each set in $Z$ is compact, so there is a finite sequence
    $m_1,\dots,m_{k-1}$ in $F$ such that
    $$\overline{b_{c^\AA(n)}}\cap \overline{b_{n_K}}\cap \overline{b_{m_0}}\cap\dots\cap \overline{b_{m_{k-1}}}=\es.$$
    By removing the closures we get
    $$b_{c^\AA(n)}\cap b_{n_K}\cap b_{m_0}\cap\dots\cap b_{m_{k-1}}=\es.$$
    The only set in this intersection which does not contain $x$
    is $b_{c^\AA(n)}=X\setminus \{x\}$. This means that
    $$b_{n_K}\cap b_{m_0}\cap\dots\cap b_{m_{k-1}}=\{x\}.$$
    Now each of the numbers $n_K,m_0,\dots,m_{k-1}$ is in $F$, so by
    \ref{blurry_filter_3} (see also Remark~\ref{remark:FilterFinite})
    there is $j\in F$ such that $j\le n_K$, and $j\le m_i$ for all
    $i<k$.  So we have
    $\overline{b_j}\subseteq b_{n_K}\cap b_{m_0}\cap\dots\cap
    b_{m_{k-1}}=\{x\}.$ By the properness of $F$, $j\ne\bzero$ and so
    $b_j\ne\es$ by \ref{def:Psi_item3}. Thus, we must have
    $b_j=b_n=\{x\}$ and finally by \ref{def:BC4} we have $j=n$, so
    $n\in F$, a contradiction.  \qedhere
  \end{enumerate-(i)}
\end{proof}

The following result is the analogue of the Stone's
representation theorem (see \cite{Ko89} for a 
review). 
In our setup a sorted complemented algebra plays the role of a Boolean algebra and
the set of its blurry $K$-filters the role of the set of ultrafilters.

\begin{thm}[Blurry duality theorem]%
  \label{thm:BlurryDuality} 
  Let $(X,\beta)$ be a CCLCP basis space. Then
  $(X,\beta)\equiv \f(\psi(X,\beta))$.
\end{thm}
\begin{proof}
  Let $\beta=(b_n)_{n\in\N}$, and set $\psi(X,\beta)= \AA$ and
  $\f(\AA)= (X_\AA,\beta_{\AA})$.  We need to show that
  $(X,\beta)\equiv (X_{\AA},\beta_{\AA})$.  Recall that
  $X_{\AA}=\FF(\AA)$.  Let \(h \colon X \to X_{\AA}\) be defined by
  \begin{equation}
    \label{eq:def_of_h}
    h(x)=\{n \in A \mid x \in b_n\}.
  \end{equation}
  We show that $h$ witnesses the equivalence between $(X,\beta)$ and
  $(X_{\AA},\beta_{\AA})$.  First we need to check that indeed
  \(h(x) \in X_{\AA}\) for all $x\in X$, that is $h(x)$ is a blurry
  $K$-filter on~$\AA$.
  \begin{claim}\label{ClaimBl}
    \(h(x) \in \FF(\AA)\), for
    every \(x \in X\).
  \end{claim}
  \begin{proof}
    Fix \(x \in X\). We check that \(h(x)\) satisfies all conditions
    \ref{blurry_filter_1}--\ref{blurry_filter_6}.  Let
    $n=\bone^\AA$. Then by \ref{def:Psi_item2} and \ref{def:BC2},
    $b_n=X$, so $x\in b_n$ and $\bone^\AA\in h(x)$. This proves
    \ref{blurry_filter_1}.  To prove \ref{blurry_filter_2}, let
    \(n_0\) and \(n_1\) be such that \(n_0 \le^\AA n_1\) and
    \(n_0 \in h(x)\). We want to show that $n_1\in h(x)$.  By
    \ref{def:Psi_item1}, we have either \(b_{n_0} = b_{n_1}\) or
    \(\overline{b_{n_0}} \subseteq b_{n_1}\).  In the former case we
    are done.  Otherwise, by \eqref{eq:def_of_h}, \(x \in b_{n_0}\),
    implying \(x \in b_{n_1}\), and so \(n_1 \in h(x)\).  To prove
    \ref{blurry_filter_3}, let \(n_0, n_1 \in h(x)\).  By definition
    of $h(x)$, we have $x\in b_{n_0}\cap b_{n_1}$.  By \ref{def:BC5}
    there is $m\in \N$ such that $x\in b_m$ and
    $\overline{b_m}\subset b_{n_0}\cap b_{n_1}$. So by
    \ref{def:Psi_item1} we have $m\le n_0$ and $m\le n_1$ as well as
    $m\in h(x)$.
    
    Let $n=\bzero^\AA$. Then $b_n=\es$ by \ref{def:Psi_item3}.
    Clearly $x\notin \es$, so $\bzero \notin h(x)$.  Thus, $h(x)$
    satisfies \ref{blurry_filter_4} and is a proper filter.
		
    To check that \(h(x)\) is a blurry filter, let \(n \in \N\) be
    such that \(n \notin h(x)\) and \(c^\AA(n)\notin h(x)\), and pick
    \(n_0 \in \N\) such that \( n <^\AA n_0\). Now
    $x\notin b_{c^\AA(n)}=X\setminus \overline{b_n}$ (by
    \ref{def:Psi_item4}), so $x\in \overline{b_n}\subset b_{n_0}$ (by
    \ref{def:Psi_item1}).  Thus, \(n_0\in h(x)\), as desired.  This
    proves \ref{blurry_filter_5}.  Finally, by \ref{def:BC6} and
    \ref{def:Psi_item5} there is $n\in h(x)\cap K$. Thus, \(h(x)\)
    satisfies \ref{blurry_filter_6} and is a \(K\)-filter.
  \end{proof}	
  
  Let us show next that \(h\) is a bijection.  Suppose $F\in X_{\AA}$.
  By Lemma~\ref{lemma:atmost1}\ref{lemma:atmost1:3} we have
  \(F=\{n \in A \mid x \in b_n\}\) for some \(x \in X\), i.e. $F=h(x)$
  by~\eqref{eq:def_of_h}.  This shows that \(h\) is onto. To see that
  $h$ is one-to-one, let $x_0, x_1\in X$ and assume that
  $h(x_0)=h(x_1)=F$. By Claim~\ref{ClaimBl}, $F$ is a blurry
  $K$-filter. Let $F^*=\{b_n\mid n\in F\}$ be as in
  Lemma~\ref{lemma:atmost1}. By \eqref{eq:def_of_h} we then have
  $\{x_0,x_1\}\in \Cap F^*$, so by Lemma
  \ref{lemma:atmost1}\ref{lemma:atmost1:0}-\ref{lemma:atmost1:1} it
  follows that $x_0=x_1$.  This completes the proof that $h$ is a
  bijection.
  
  Let us now prove that \(h[b] \in \b_{\AA}\) for every \(b \in \b\),
  and \(h^{-1}[b']\in \b\) for every \(b' \in \b_{\AA}\).
  Let \(b \in \b\). Then $b=b_n$ for some $n$ and
  \begin{align*}
    h[b_n] &=\{h(x)\mid x\in b_n\}\\
           &=\{h(x)\mid n\in h(x)\} &\text{by \eqref{eq:def_of_h}}\\
           &=\{F\in X_{\AA}\mid n\in F\} &\text{by surjectivity of }h\\
           &=b[n]\in \beta_{\AA}&\text{by Definition \ref{def:phi_map}}.
  \end{align*}
  On the other hand, let \(b' \in \beta_\AA\).  By
  Definition~\ref{def:phi_map}, $b'=b[n]=\{F\in\FF(\AA)\mid n\in F\}$
  for some $n\in A$.  We have:
  \begin{align*}
    h^{-1}[b[n]] &=\{x\in X\mid h(x)\in b[n]\}\\
                 &=\{x\in X\mid h(x)\in \{F\in\FF(\AA)\mid n\in F\}\}\\
                 &=\{x\in X\mid n\in h(x)\}\\
                 &=\{x\in X\mid x\in b_n\}&\text{by \eqref{eq:def_of_h}}\\
                 &=b_n\in \beta,
  \end{align*}
  which completes the proof.
\end{proof}

As in the classical Stone duality, one would like to prove also that, for every SCA $\AA$, $\AA \cong \psi(\f(\AA))$.
The main obstacle is that it is not known whether $\f(\AA)$ is a CCLCP basis space (see Lemma~\ref{lem:from_alg_to_basisspace}). Consequently, we do not obtain an analogue of Theorem~\ref{thm:PsiReduction} characterizing $\cong_{\AAA}$ in terms of $\equiv$.

\begin{Lemma}\label{lemma:BasisEqIso}
  Suppose that $(X,\beta)$ and $(X',\beta')$ are equivalent CCLCP
  basis spaces.  Then $\psi(X,\beta)$ and $\psi(X',\beta')$ are
  isomorphic sorted complemented algebras.
\end{Lemma}
\begin{proof}
  Suppose $\beta=(b_n)_{n\in\N}$ and $\beta'=(b'_n)_{n\in\N}$.  Denote
  $\AA=\psi(X,\beta)$ and $\AA'=\psi(X',\beta')$.  Let
  $h\colon X\to X'$ witness the equivalence.  Recall that $A=A'=\N$ by
  \ref{def:Psi_item0}.  Define $g\colon \N\to \N$ so that $g(n)$ is
  the unique (by \ref{def:BC4}) index $m$ such that $h[b_n]=b'_m$.  In
  particular, for all $n\in A$
  \begin{equation}
    \label{eq:commutehbn}
    h[b_n]=b'_{g(n)}.
  \end{equation}
  We check that $g$ is an isomorphism.  Since $h$ is a homeomorphism
  from $(X,\la\beta\ra)$ to $(X',\la\beta'\ra)$ (see
  Fact~\ref{fact:Homeo}), for all $n_0,n_1 \in A$ we have:
  \begin{align*}
         &n_0\le^{\AA}n_1\\
    \iff & \overline{b_{n_0}}\subseteq b_{n_1}&\text{by \ref{def:Psi_item1}}\\
    \iff &\overline{h[b_{n_0}]}\subseteq h[b_{n_1}]&h\text{ is a homeomorphism}\\
    \iff &\overline{b'_{g(n_0)}}\subseteq b'_{g(n_1)}&\text{by \eqref{eq:commutehbn}}\\
    \iff &g(n_0)\le^{\AA'} g(n_1)&\text{by \ref{def:Psi_item1}}.
  \end{align*}
  By \ref{def:BC2}, \ref{def:BC4}, \ref{def:Psi_item2} and
  \ref{def:Psi_item3}, and the fact that $h$ is a homeomorphism,
  $g(\bone^{\AA})=\bone^{\AA'}$ and $g(\bzero^{\AA})=\bzero^{\AA'}$.
  For complementation, fix $n\in A$ and let $m$ be such that
  $g(c^{\AA}(n))=m.$ So $h[X\setminus \overline{b_n}]=b'_m$ by
  \ref{def:Psi_item4} and \eqref{eq:commutehbn}. Since $h$ is a
  homeomorphism, $h[b_n]=X'\setminus \overline{b'_m}$.  By
  \ref{def:BC1}, let $k$ be such that
  $b'_k=X'\setminus \overline{b'_m}$.  Then $b'_k=h[b_n]$, so
  $g(n)=k$.  But also $k=c^{\AA'}(m)$ and $m=c^{\AA'}(k)$.  Thus,
  $g(c^\AA(n))=c^{\AA'}(g(n))$.
  
  Finally, since $h$ is a homeomorphism, we have that $\overline{b_n}$
  is compact if and only if $\overline{h[b_{n}]}=\overline{b'_{g(n)}}$
  is, so using \ref{def:Psi_item5}, we have
  $n\in K^{\AA}\iff g(n)\in K^{\AA'}$.
\end{proof}

\begin{Lemma}\label{lemma:BasisEqIsoBack}
  Suppose that two sorted complemented algebras $\AA$ and $\AA'$ are
  isomorphic. Then $\f(\AA)$ and $\f(\AA')$ are equivalent as basis
  spaces.
\end{Lemma}
\begin{proof}
  Suppose $f\colon A\to A'$ is an isomorphism from $\AA$ to~$\AA'$.
  As in Definition \ref{def:phi_map}, denote
  $\f(\AA)=(X_{\AA},\beta_{\AA})$ and
  $\f(\AA')=(X_{\AA'},\beta_{\AA'})$ where $X_{\AA}=\FF(\AA)$ and
  $X_{\AA'}=\FF(\AA')$.  Define $h\colon X_{\AA}\to X_{\AA'}$ by
  $h(F)=f[F]=\{f(n)\mid n\in F\}$.
  
  Let us check that this is an equivalence between the basis spaces.
  That $h$ is indeed a bijection follows from the fact that $f$ is an
  isomorphism between $\AA$ and~$\AA'$.  Suppose now that
  $b[n]\in \beta_{\AA}$ for some $n\in\N$ where
  $b[n]=\{F\in\FF(\AA)\mid n\in F\}$ (see
  Definition~\ref{def:phi_map}).  Let $m=f(n)$. We claim that
  $h[b[n]]=b'[m]$, where
  $b'[m]=\{F\in\FF(\AA')\mid m\in F\}\in \beta_{\AA'}$:
  \begin{align*}
    h[b[n]]&=\{f[F] \mid F\in \FF(\AA)\land F\in b[n]\}\\
           &=\{f[F] \mid F\in \FF(\AA)\land n\in F\}\\
           &=\{f[F] \mid F\in \FF(\AA)\land m\in f[F]\}\\
           &=\{h(F) \mid F\in \FF(\AA)\land m\in h(F)\} \\
           &=\{F \mid F\in \FF(\AA')\land m\in F\} &h\text{ is a bijection}\\
           &=b'[m].
  \end{align*}
  The other direction, namely that for all $b'\in \beta_{\AA'}$,
  $h^{-1}[b']\in\beta_{\AA}$ is shown symmetrically.
\end{proof}

\begin{thm}[Characterization of $\equiv$ by $\,\cong_\AAA$]%
  \label{thm:PsiReduction}
  For all CCLCP basis spaces $(X,\beta)$ and $(X',\beta')$, it holds
  that $(X,\beta)\equiv (X',\beta')$ if and only if
  $\psi(X,\beta)\cong \psi(X',\beta')$.
\end{thm}
\begin{proof}
  If $\psi(X,\beta)$ and $\psi(X',\beta')$ are isomorphic,
  then by Lemma~\ref{lemma:BasisEqIsoBack} we obtain
  \begin{equation}
    \label{eq:fpsiequiv}
    \f(\psi(X,\beta))\equiv\f(\psi(X',\beta')).
  \end{equation}
  Moreover, by Theorem~\ref{thm:BlurryDuality} we have
  $(X,\beta)\equiv \f(\psi(X,\beta))$ and
  $(X',\beta')\equiv \f(\psi(X',\beta'))$. Thus,
  $(X,\beta)\equiv (X',\beta')$.  Suppose on the other hand that
  $(X,\beta)$ and $(X',\beta')$ are equivalent as basis spaces. Then
  by Lemma~\ref{lemma:BasisEqIso} the algebras $\psi(X,\beta)$ and
  $\psi(X',\beta')$ are isomorphic.
\end{proof}

\begin{Def}\label{def:map_from_BBB_to_PPP}
  Given a CCLCP basis space $(X,\beta)$, let $\psi^*(X,\beta)$  
  be the partial order
  $(K^\AA,\le^\AA)$
  where $\AA=\psi(X,\beta)$.
\end{Def}

\begin{prop}\label{prop:BasisToPPP}
    Let $(X,\beta)$ and $(X',\beta')$ be basis spaces.
    Then $(X,\beta)\equiv (X',\beta')$ if and only if
    $\psi^*(X,\beta)\cong \psi^*(X',\beta)$.
\end{prop}
\begin{proof}
    Follows from Theorem~\ref{thm:PsiReduction}
    and Proposition~\ref{prop:IsoOnKIsEnough}.
\end{proof}



\vspace{10pt}

\subsection{Classification of simplicial complexes up to PL-homeomorphism}
\label{ssec:Classification_NON_BOREL}

\subsubsection{Simplicial complexes and point adjustment}
\label{ssec:PointAdjustment}

We begin this section by recalling some basic definitions of
PL-geometry and fixing notation.

\begin{defn}\label{def:Simplex}
  The \textbf{standard $n$-simplex} is the convex hull of the standard
  basis $\{\ee_k\mid 0\le k\le n\}$ of $\R^{n+1}$ denoted
  $\bDelta^n$. An \textbf{$n$-simplex} in a topological space $X$ is a
  (continuous) embedding $\kappa\colon \bDelta^n\to X$. The number $n$
  is the \textbf{dimension} of $\kappa$ and is denoted by
  $\dim(\kappa)$.  If $X=\R^m$ and $\kappa$ is linear, then $\kappa$
  is a \textbf{rectilinear} $n$-simplex.
\end{defn}


\begin{defn}\label{def:Faces}
  A \textbf{standard (proper)} $k$-\textbf{face} of the standard
  $n$-simplex is the convex hull of a (proper) subset of
  $\{\ee_0,\dots,\ee_n\}$ of size $k+1$. Given a simplex
  $\kappa\colon\bDelta^n\to X$ in a topological space $X$, its
  \textbf{$k$-face} is a $k$-simplex $\lambda\colon \bDelta^k\to X$
  such that $\kappa^{-1}\lambda=\kappa^{-1}\circ\lambda$ is an
  isometry from $\bDelta^k$ onto a standard $k$-face of $\bDelta^n$
  (in particular $\kappa^{-1}\lambda$ is rectilinear). The $0$-faces
  (both standard and otherwise) are called
  \textbf{vertices}. Technically a vertex is a function from the
  singleton $\{\ee_k\}$, for $k \in \{0, \dots, n\}$, to $X$, but we
  usually identify it with the unique point $\kappa(\ee_k)$ in its
  range.  The set of vertices of a simplex $\kappa$ is denoted by
  $V(\kappa)$.  Let $\partial\bDelta^n$ be the union of all standard
  proper faces of $\bDelta^n$. Let
  $\obDelta^n=\bDelta^n\setminus \partial\bDelta^n$.  For any simplex
  $\kappa\colon \bDelta^n\to X$, denote
  \begin{equation}
    \overset{\circ}{\kappa}:=\kappa\rest \obDelta^n\text{ and }\partial\kappa:=\kappa\rest\partial\bDelta^n.\label{eq:oversetcirc}
  \end{equation}
  Note that according to this definition, for a vertex $v$, we have
  $\overset{\circ}v=v$ and $\partial v=\es$, because the only proper face of 
  a vertex is empty.
\end{defn}



In what follows, if $f\colon Z\to X$ is a map, we denote its
range by $|f|$.

\begin{defn}\label{def:EquivalenceOfSimplexes}
  Two $n$-simplexes $\kappa,\lambda\colon \bDelta^n\to X$ are considered
  \textbf{equivalent}, if $|\kappa|=|\lambda|$ and
  $\kappa^{-1}\lambda$ (and hence also $\lambda^{-1}\kappa$)
  is an isometry of $\bDelta^n$ onto itself.  \textbf{From now on we
    consider simplexes up to this equivalence relation unless stated
    otherwise.}
\end{defn}

\begin{defn}\label{def:simplicial_complex}\cite[Chapter III]{hudson1969pl}
  Let \(X\) be a topological space. A \textbf{simplicial complex in
    \(X\)} is a (possibly finite) countable set $T$ of simplexes in
  $X$ such that the following hold up to the equivalence relation
  of Definition~\ref{def:EquivalenceOfSimplexes}:
  \begin{enumerate-(1)}
  \item\label{def:simplicial_complex-1} Every face of every simplex in
    $T$ is in $T$.
  \item\label{def:simplicial_complex-2} If $\kappa_0,\kappa_1\in T$
    and $|\kappa_0|\cap |\kappa_1|\ne\es$, then there is $\lambda\in T$
    such that $|\lambda|=|\kappa_0|\cap |\kappa_1|$ and $\lambda$
    is a face of both $\kappa_0$ and $\kappa_1$.
  \item\label{def:simplicial_complex-3} For all
    $\kappa\in T$ there is a neighbourhood $U\subset X$ of $|\kappa|$
    such that there are at most finitely many $\lambda\in T$
    with $|\lambda|\cap U\ne\es$.
  \end{enumerate-(1)}
  Sometimes we call a simplicial complex just \textbf{complex}. If
  $T_0\subseteq T$ is also a simplicial complex, then $T_0$ is called
  a \textbf{subcomplex} of~$T$.  We denote by
  \(V(T)=\Cup_{\kappa\in T} V(\kappa)\) the set of all $0$-simplexes
  (vertices) of~\(T\). The set $|T| = \Cup_{\kappa\in T} |\kappa|$ is the \textbf{realization}
  of~$T$ and is endowed with the topology induced from \(X\). If
  $|T|=X$, we say that $T$ is a \textbf{triangulation} of~$X$.  Thus,
  $T$ is always a triangulation of $|T|$. If $X\subseteq \R^m$ and
  each simplex $\kappa\in T$ is rectilinear, then we say that $T$ is a
  \textbf{rectilinear} triangulation (or complex). 
  Note that condition \ref{def:simplicial_complex-3} implies that
  $|T|$ is locally compact.
  
  If $S\subset T$,
  let $\scl(S)$ be the smallest subcomplex of $T$ which
  contains~$S$. So, for a simplex $\kappa$, by $\scl(\kappa)$ we denote the
  complex which consists of $\kappa$ and all its faces. A special case
  is $\scl(\bDelta^n)$ which consists of
  the identity map $\bDelta^n\to\bDelta^n$
  and \textit{its} faces.
\end{defn}

We recall some standard relations among simplicial complexes. 
\begin{defn}\label{def:rel_complexes}
  Let \(S\) and $T$ be simplicial complexes.
  \begin{enumerate-(i)}
  \item \label{def:Subd} $S$ is a \textbf{subdivision} of $T$, denoted
    $S\subd T$, if $|S|=|T|$, and for all $\sigma\in S$
    there is $\kappa\in T$ such that $|\sigma|\subseteq |\kappa|$ and
    $\kappa^{-1}\sigma$ is rectilinear.
  \item A map $h\colon |S|\to |T|$ is \textbf{simplicial} from $S$ to
    $T$, often written $h\colon S\to T$, if it is a homeomorphism and
    for all $\sigma\in S$ we have $h\sigma= h\circ\sigma\in T$. If a
    simplicial map from $S$ to $T$ exists, then $S$ and $T$ are
    \textbf{PL-isomorphic}.
  \item A map $h\colon |S|\to |T|$, also denoted $h\colon S\to T$,
    is a \textbf{PL-homeomorphism} from $S$ to $T$, if there are
    subdivisions $S'\subd S$ and $T'\subd T$ such that
    $h\colon S'\to T'$ is simplicial. Then $S$ and $T$ are called
    \textbf{PL-homeomorphic}. This is denoted by~$S\homeo^{PL}T$.
  \item A map $h\colon |S|\to |T|$ is a \textbf{PL-embedding} if there
    is a subcomplex $T_0$ of a subdivision of $T$ such that $h$ is a
    PL-homeomorphism from $S$ to $T_0$.
  \end{enumerate-(i)}
\end{defn}

\begin{Lemma}\label{lemma:AnySimplexRect}
  Suppose that \(S\) and \(T\) are simplicial complexes such that $S\subd T$, and $\sigma\in S$.  Suppose that $\lambda\in T$
  is \emph{any} simplex such that $|\sigma|\subset |\lambda|$. Then
  $\lambda^{-1}\sigma$ is rectilinear. 
\end{Lemma}
\begin{proof}
  Let $\kappa\in T$ be as given by \ref{def:Subd} of Definition
  \ref{def:rel_complexes}. So $|\sigma|\subset |\lambda|\cap |\kappa|$
  and $\kappa^{-1}\sigma$ is rectilinear.  Let $\tau$ be the common
  face of $\lambda$ and $\kappa$, so $|\sigma|\subset |\tau|$ and
  $\kappa^{-1}\tau$ is rectilinear by the definition of a face.  The
  range of $\kappa^{-1}\tau$ contains $|\kappa^{-1}\sigma|$, so on
  that set the inverse $\tau^{-1}\kappa$ is defined and is also
  rectilinear.  Since $\tau$ is also a face of $\lambda$,
  $\lambda^{-1}\tau$ is rectilinear too and its domain contains
  $|\tau^{-1}\kappa\rest |\kappa^{-1}\sigma||$.  Now
  $\lambda^{-1}\sigma =
  (\lambda^{-1}\tau)\circ(\tau^{-1}\kappa)\circ(\kappa^{-1}\sigma)$ is
  linear as a composition of linear maps.
\end{proof}



\begin{Def}\label{def:Star}
  Let $T$ be a simplicial complex and $A$ a set. The \textbf{star} of
  $A$ in $T$, denoted $\St_T(A)$, is the smallest subcomplex $S$ of
  $T$ with $A\cap |T|\subset |S|$. The \textbf{link} of $A$ in $T$,
  denoted $\Lk_T(A)$, is the set
  $$\Lk_{T}(A):=\{\kappa\in \St_T(A)\mid A\cap |\kappa|=\es\}.$$
  It is also a subcomplex of~$T$. For a singleton, we denote
  $\St_{T}(\{z\})=\St_{T}(z)$ and $\Lk_T(\{z\})=\Lk_T(z)$.
\end{Def}

The above definition is written with the following possibility in
mind.  Suppose $S$ is a subcomplex of $T$ and $A\subset |T|$. Then
$\St_{S}(A)$ and $\Lk_S(A)$ are well-defined. In particular, if
$A\cap |S|=\es$, then $\St_{S}(A)=\Lk_S(A)=\es$, and if
$A\subset |S|$, then $\St_S(A)=\St_T(A)$ and $\Lk_S(A)=\Lk_T(A)$.

\begin{Fact}\label{fact:LinkStarVert}
  Suppose $z\in |T|$.  If $\lambda\in\Lk_{T}(z)$, then there is
  $\kappa\in\St_T(z)$ such that $\lambda$ is a face of $\kappa$ and
  $z\in |\kappa|$.  In particular
  $V(\lambda)\cup \{z\}\subset |\kappa|$. Note also that the set of
  vectors $\{\kappa^{-1}(v)-\kappa^{-1}(z)\mid v\in V(\lambda)\}$ is
  linearly independent, because $\kappa^{-1}(z)$ is either in the
  interior of $\bDelta^{\dim(\kappa)}$ or a vertex of \(\kappa\) that
  does not belong to \(V(\lambda)\), while $\kappa^{-1}(v)$ are
  vertices of \(\kappa\) for $v\in V(\lambda)$.
\end{Fact}

\begin{fact_n_def}\label{def:DeltaInRm}
  A rectilinear simplex $\kappa\colon \bDelta^n\to \R^m$ is uniquely
  (up to the equivalence relation of Definition \ref{def:EquivalenceOfSimplexes}) determined by its
  vertices $V(\kappa)=\{\kappa(\ee_i)\mid 0\le i\le n\}$.  Conversely,
  if $W=\{\vv_0,\dots,\vv_n\}$ are points in $\R^m$, there is a unique
  linear map from $\bDelta^n$ to $\R^m$ with $\ee_i\mapsto \vv_i$ for
  all $0\le i \le n$.  This map is denoted by
  $\Delta[W]=\Delta[\vv_0,\dots,\vv_n]$.  Note that $|\Delta[W]|$ is
  the convex hull of~$W$. The map $\Delta[W]$ is a rectilinear
  $n$-simplex if and only if the vectors $\vv_i-\vv_0$ are linearly
  independent for $1\le i\le n$. In this case, clearly
  $V(\Delta[W])=W$.
\end{fact_n_def}

The following is a generalization of the operation $\Delta$ for any
simplex.

\begin{Def}\label{def:DeltaInComplex}
  Suppose that $T$ is a complex and $x_0,\dots,x_k\in |T|$. If there
  exists $\kappa$ such that $x_0,\dots,x_k\in |\kappa|$, then, abusing
  notation, let
  $$\Delta[x_0,\dots,x_k]=\kappa\circ\Delta[\kappa^{-1}(x_0),\dots,\kappa^{-1}(x_k)].$$
  Note that $\Delta[x_0,\dots,x_k]$ is independent on the choice
  of~$\kappa\in T$ as long as $x_0,\dots,x_k\in |\kappa|$. Like in
  Definition \ref{def:DeltaInRm}, $\Delta[x_0,\dots,x_k]$ is a
  $k$-simplex if and only if $\kappa^{-1}(x_m)-\kappa^{-1}(x_0)$ are
  linearly independent for $1\le m\le k$, a property which is, again,
  independent on the choice of~$\kappa$.  In this case, clearly
  \begin{equation}
    V(\Delta[x_0,\dots,x_k])=\{x_0,\dots,x_k\}. \label{eq:VerticesofDelta}
  \end{equation}
  Also note that if $\lambda$ is a simplex in $|T|$ (but not
  necessarily $\lambda\in T$) with $|\lambda|\subset |\kappa|$ for
  some $\kappa\in T$ such that $\kappa^{-1}\lambda$ is rectilinear,
  then $\lambda = \Delta[V(\lambda)]$.
\end{Def}

\begin{Def}\label{def:StellarSubdivision0}
  A \textbf{stellar subdivision} of a complex $T$ \textbf{at}
  $z\in |T|$, denoted $T*z$, is the set
  \begin{equation}
    T*z=\big(T\setminus \{\kappa\in T\mid z\in |\kappa|\}\big)\cup \{\Delta[V(\lambda)\cup \{z\}]\mid \lambda\in \Lk_T(z)\},\label{eq:DefOfStellarSubd}  
  \end{equation}  
  By Fact~\ref{fact:LinkStarVert}, $T*z$ is a well-defined.  We skip
  the proof that $T*z$ is a subdivision of~$T$ (see
  \cite{hudson1969pl} for a standard reference).
\end{Def}

\begin{Def}[\textbf{Carrier}]
  Let $T$ be a simplicial complex.  Let the \textbf{face relation}
  be defined by
  $$\tau\le^T_F\kappa\ \text{ if and only if }\ \tau\text{ is a face of }\kappa.$$ 
  Then $\le^T_F$ is a well-founded partial order on $T$. For every
  $x\in |T|$, let $\Carr_T(x)$ be the unique minimal element of the
  set $\{\kappa\in T\mid x\in |\kappa|\}$ with respect to~$\le^T_F$.
\end{Def}

\begin{Lemma}\label{lemma:CarrierRelation}
  Let $S$ be a complex and $y\in |S|$. Let $\tau=\Carr_S(y)$.
  Then for all $x\in \overset{\circ}{\tau}$, $\Carr_S(x)=\tau$.
  If $T$ is another complex and $f\colon S\to T$ is simplicial,
  then $\Carr_T(f(y))=f\tau$.
\end{Lemma}
\begin{proof}
  Clearly $\Carr_S(x)\subset\tau$. So if $\Carr_S(x)\ne\tau$,
  then $\Carr_S(x)$ must be a face of $\tau$ which contradicts
  the assumption that $x$ is in $\overset{\circ}{\tau}$.
  The second part follows from the fact that a
  simplicial map preserves the face relation.
\end{proof}

The following is a collection of some useful facts.

\begin{Lemma}\label{lemma:Collection2}
Let $S$ and $T$ be two simplicial complexes.
  \begin{enumerate-(1)}  \item\label{fact:ObviousSimplicialStellar} If
    $h\colon S\to T$ is simplicial and $z\in |S|$, then $h$ is a
    simplicial map from $S*z$ to $T*h(z)$.
  \item \label{fact:Adjust} Suppose $z,w\in |S|$ are such that
    $\Carr_S(z)=\Carr_S(w)$.  Then there is a simplicial map
    $g\colon S*z\to S*w$ such that $g(z)=w$ and $g$ is identity
    outside~$|\St_S(z)|$ and for all $x\in |S|$, $d(x,g(x))\le d(z,w)$.
  \end{enumerate-(1)}
\end{Lemma}
\begin{proof}
  \ref{fact:ObviousSimplicialStellar} Is
  obvious from the definitions.
  We now sketch a proof for \ref{fact:Adjust}.  Note that
  $\Lk_S(z)=\Lk_S(w)$, and denote it by \(L\). For each
  $\lambda\in L$, let
  $\Delta_{\lambda,z}:=\Delta[V(\lambda)\cup \{z\}]$ and
  $\Delta_{\lambda,w}:=\Delta[V(\lambda)\cup \{w\}]$, and define
  $g_{\lambda}\colon |\Delta_{\lambda,z}|\to |\Delta_{\lambda,w}|$ by
  $g_{\lambda}=\Delta_{\lambda,w}\circ
  \Delta_{\lambda,z}^{-1}$. Notice that
  \(|\St_S(z)|=|\Cup_{\lambda \in \Lk_S(z)} \Delta_{\lambda,z}|\).
  Let $g\colon S*z\to S*w$ be defined by $g(x)=x$ if
  $x\notin |\St_S(z)|$, and $g(x)=g_{\lambda}(x)$ if
  $x\in |\Delta_{\lambda,z}|$.  Restricted to $|\St_S(z)|$ we can view
  $g$ as the unique function which maps $z$ to $w$ and is linearly
  interpolated along chords connecting $z$ to $|L|$. Let
  $x\in |\St_S(z)|$.  Then $x$ is on a chord connecting $|L|$ to $z$. At
  one end of that chord $g$ is identity and at the other end it maps
  $z$ to $w$ so by linearity $d(g(x),x)\le d(z,w)$.
\end{proof}

\begin{Prop}\label{prop:ExtendSimplicial}
  Suppose $f\colon S\to T$ is a simplicial map, and that $s\in |S|$ and
  $t\in |T|$ are such that $\Carr_{T}(f(s))=\Carr_{T}(t)$. Then there
  is a simplicial $h\colon S*s\to T*t$ such that (1) $h(s)=t$, 
  (2) $h(x)=f(x)$ for all $x\in |S|\setminus |\St_S(s)|$.
  If additionally $|S|=|T|$, then (3) for all $x\in |S|$,
  $d(x,h(x))\le d(x,f(x))+d(t,f(s))$.
\end{Prop}
\begin{proof}
  Denote $y=f(s)$. By the hypothesis $\Carr_{T}(y)=\Carr_{T}(t)$, so
  by Lemma~\ref{lemma:Collection2}\ref{fact:Adjust}, there is a
  simplicial map $g\colon T*y\to T*t$ such that $g$ is the identity
  outside $|\St_{T}(y)|$ and $d(x,g(x))\le d(y,t)$ for all $x\in |T|$.
  Since $f$ is simplicial, $f^{-1}|\St_{T}(y)|=|\St_S(s)|$.  Let
  $h=g\circ f$. Using
  Lemma~\ref{lemma:Collection2}\ref{fact:ObviousSimplicialStellar}, it
  is easy to check that $h$ is as desired. Suppose now $|S|=|T|$. Then
  $d(x,h(x))\le d(x,f(x))+d(f(x),g(f(x)))\le d(x,f(x))+d(t,f(s))$.
  The latter follows by the choice of~$g$.
\end{proof}

We will now refer to our paper \cite{IW26}.
Consider iterated stellar subdivisions.  Let
$\bar z=(z_0,\dots,z_{n-1})$ be a finite sequence of elements of
\(|T|\). Define by induction $T*\bar z=T$ for $n=0$ (empty $\bar z$),
and for $n> 0$, $T*\bar z=(T*z_0*\cdots*z_{n-2})*z_{n-1}$. We call
\(T * \bar z\) a \textbf{$\lo$-stellar subdivision} of~$T$.

Let now $\bar z=(z_i)_{i\in \N}$ be an infinite sequence in a set $X$,
and let $A \subset X$. We denote by
\begin{equation}
  \bar z \cap A \label{eq:CapSeq}
\end{equation}
the subsequence $(z_i)_{i\in I}$, where $I=\{i\in\N \mid z_i \in A\}$.
If $T$ is a complex and $\bar z \subset |T|$ is a sequence such that
$\bar z \cap C$ is finite for every compact $C \subset |T|$
(equivalently, by \cite[Lemma 2.11]{IW26} $\bar z \cap |S|$ is finite for every finite subcomplex
$S \subset T$), then we say that
$\bar z$ is \textbf{locally finite}.
 
\begin{Def}\cite[Definition 2.12]{IW26}\label{def:StellarSubdivision}
  Suppose that $\bar z=(z_0,z_1,\dots)$ is an infinite locally finite
  sequence. Define
  \begin{equation}
    T*\bar z:=\Cup_{i\in\N}\Cap_{j\ge i}T*\bar z^{j}\label{eq:OmegaSubd}  
  \end{equation}
  where $\bar z^j=(z_0,\dots,z_{j-1})$ is the initial segment of $\bar z$ of length $j$.
  Equivalently, $\sigma\in T*\bar z$ if and only if there exists $i\in\N$
  such that for all $j\ge i$, $\sigma\in T*\bar z^j$.
  We say that $T*\bar z$ is an \textbf{$\o$-stellar subdivision of~$T$}.
\end{Def}
We showed in \cite[Lemma 2.13]{IW26} that an $\o$-stellar
subdivision of \(T\) is indeed a subdivision of \(T\).  Below we
collect several results on stellar subdivisions that will be used in
the sequel.

\begin{Lemma}\label{lemma:ExtedingSubdivisions}\cite[Lemma 3.1]{IW26}
  Let $S$ and $T$ be different finite triangulations of the same space
  $X$. Suppose that $S_0$ and $S_1$ are subcomplexes of $S$, and $T_0$
  and $T_1$ are subcomplexes of $T$ such that $|S_0|=|T_0|$,
  $|S_1|=|T_1|$, $S=S_0\cup S_1$, and $T=T_0\cup T_1$.  Suppose
  further that $\bar s_0\subset |S_0|$, $\bar s_1\subset |S_1|$,
  $\bar t_0\subset |T_0|$, and $\bar t_1\subset |T_1|$ are finite
  sequences such that
  \begin{align}
    S_0*\bar s_0&=T_0*\bar t_0\label{eq:LemmaESAssumption1}\\
    S_1*\bar s_0*\bar s_1&=T_1*\bar t_0*\bar t_1.\label{eq:LemmaESAssumption2}
  \end{align}
  Then $S*\bar s_0*\bar s_1=T*\bar t_0*\bar t_1$.
\end{Lemma}

\begin{thm}\label{thm:CommonStellarSubdivision}\cite[Theorem 3.2]{IW26}
  Suppose that two (possibly infinite) triangulations $S$ and $T$ of a complex $X$
  have a common subdivision. Then there are locally finite
  $\bar s,\bar t\subset X$ such that $S*\bar s=T*\bar t$.
\end{thm}

We now define a ``canonical'' dense
subset of any simplicial complex and introduce several notions that enable us to characterize the existence of a PL-homeomorphism between two complexes. In particular, we will show that a PL-homeomorphism between two complexes exists if and only if there is a specific correspondence between their canonical dense subsets, which we shall define below.

\begin{Def}\label{def:TheQTset}
Define $\Q^*=\Cup_{n=1}^\infty \Q^n$ to be the union of all finite
  dimensional rational vector spaces which, for all $n$,
  is assumed to satisfy
  $$\Q^*\cap \R^n=\Q^n.$$
  For a complex $T$, let $\Q(T)=\Cup\{\kappa[\Q^*]\mid \kappa\in T\}$.
\end{Def}
Notice that since the vertices of $\bDelta^n$ are in $\Q^*$, we have
\begin{equation}
  \label{eq:VerticesofQT}
  V(T)\subseteq \Q(T).
\end{equation}

Given a complex \(T\) in a topological space \(X\), we say that a set
  $Q\subset X$ is \textbf{$T$-dense} if for every $x\in X$, every
  simplex $\tau\in T$ with $x\in |\tau|$, and every open subset $U$ of \(X\) there is
  $q\in Q\cap |\tau|\cap U$.

\begin{Lemma}\label{lemma:FindPointsQ}
  For any triangulation $T$,
  $\Q(T)$ is $T$-dense.
\end{Lemma}
\begin{proof}
  Since each $\kappa\in T$ is a topological embedding, it is enough
  to show that $\Q^*\cap \bDelta^n$ is dense in $\bDelta^n$.
  Since the coordinate vectors $\ee_0,\dots,\ee_n$ are in $\Q^*$,
  all the linear combinations
  $$q_0\ee_0+\cdots+q_n\ee_n,$$
  for $q_0, \dots, q_n \in \Q$, are also in $\Q^*$ and these are dense in $\bDelta^n$ which is the
  convex hull of $\{\ee_k\mid 0\le k\le n\}$.
\end{proof}

\begin{Lemma}\label{lem:SimplicialQQ}
  Suppose $S$ and $T$ are simplicial complexes. 
  If $h\colon S\to T$ is simplicial, then $h\Q(S)=\Q(T)$.
\end{Lemma}
\begin{proof}
  Suppose that $q\in h\Q(S)$. Let $\sigma\in S$ and $q'\in\Q^*$ be
  such that $h(\sigma(q'))=q$. But $\kappa=h\sigma\in T$, so
  $\kappa(q')=q\in \Q(T)$.  The other direction follows by symmetry.
\end{proof}

\begin{Def}\label{def:Q-subd}
    \begin{enumerate-(i)}
    \item Given $A\subset \R^m$, a map $f\colon A\to \R^n$ is
  a \textbf{$\Q$-map}, if $f[\Q^*\cap A]= \Q^*\cap f[A]=\Q^*\cap |f|$.  When no confusion can occur, we use the notation \(f[\Q^*]\) instead of \(f[\Q^*\cap A]\). 
  \item A simplex $\kappa$ in $\R^n$ is a \textbf{$\Q$-simplex}, if it is a
  $\Q$-map as a map $\kappa\colon\bDelta^k\to \R^n$. 
  \item A subdivision
  $S\subd T$ is a \textbf{$\Q$-subdivision}, denoted $S\subd_{\Q}T$,
  if for all $\sigma\in S$ and all $\kappa\in T$ with
  $|\sigma|\subset|\kappa|$ the map $\kappa^{-1}\sigma$ is a
  rectilinear $\Q$-simplex. 
  \item A \textbf{$\Q$-PL-homeomorphism} $h$ from
  $S$ to $T$ is a PL-homeomorphism from $S$ to $T$ such that there are
  $\Q$-subdivisions $S'\subd_\Q S$ and $T'\subd_\Q T$
  such that $h$ is simplicial from $S'$ to~$T'$.  If such a map
  exists, we say that $S$ and $T$ are \textbf{$\Q$-PL-homeomorphic} and write
  $S\homeo^{\Q PL}T$.
    \end{enumerate-(i)}
\end{Def}



  

\begin{Lemma}\label{lemma:algvert}
  Let $\kappa$ be a rectilinear simplex in $\R^n$. The following are
  equivalent:
  \begin{enumerate-(i)}
      \item\label{lemma:algvert-1} \(\kappa\) is a $\Q$-simplex,
      \item\label{lemma:algvert-2} $V(\kappa)\subset\Q^*$,
      \item\label{lemma:algvert-3} $\kappa[\Q^*]\subset \Q^*$.
  \end{enumerate-(i)}
\end{Lemma}
\begin{proof}
  Let us first prove that \ref{lemma:algvert-1} implies
  \ref{lemma:algvert-2}.  Suppose $\kappa\colon \bDelta^k\to \R^n$ is
  a $\Q$-map.  Since the vertices of $\bDelta^k$ are in $\Q^*$,
  $V(\kappa)=\{\kappa(\ee_0),\dots,\kappa(\ee_k))\}\subset \Q^*$.
  
  For the implication \ref{lemma:algvert-2} $\Rightarrow$
  \ref{lemma:algvert-1}, assume that $V(\kappa)\subset\Q^*$ and let
  $q\in \Q^*\cap \bDelta^k$.  By the linearity of $\kappa$ there is a
  matrix $\bW$ such that $\kappa(q)=\bW q$. The columns of that matrix
  are given by $\bW\ee_i=\kappa(\ee_i)$ and so all entries are
  rational.  So $\kappa(q)=\bW q$ is in $\Q^*$, and we have
  $\kappa[\Q^*\cap \bDelta^k]\subseteq \Q^*\cap |\kappa|$ which by
  definition means that $\kappa$ is a $\Q$-map.

  The domain of $\kappa$ is $\bDelta^k$, so we have $\kappa[\Q^*]=\kappa[\Q^*\cap\bDelta^k]$. On the other hand $\Q^*\cap|\kappa|\subset\Q^*$. This shows that \ref{lemma:algvert-1}
  is equivalent to \ref{lemma:algvert-3}.  
\end{proof}

In the next few statements $S$ and $T$ are simplicial complexes.

\begin{Lemma}\label{lemma:Q-compatibilility}
  For all $\kappa\in T$, $\Q(T)\cap |\kappa|=\kappa[\Q^*]$. In particular, $\kappa^{-1}[\Q(T)]\subset\Q^*$.
\end{Lemma}
\begin{proof}
  The inclusion from right to left clearly holds by Definition \ref{def:TheQTset} of $\Q(T)$.  To prove the other direction,
  let $x\in \Q(T)\cap|\kappa|$ be arbitrary. Then there is 
  $\lambda \in T$, and
  $q\in \Q^*$ such that $\lambda(q)=x$. If $\lambda=\kappa$, the conclusion
  follows trivially, so we can assume that $\lambda\ne\kappa$.
  Since
  $\lambda(q)\in |\lambda|\cap |\kappa|$, there is a common face
  $\sigma\in T$ of $\lambda$ and $\kappa$ and $s\in \bDelta^{\dim(s)}$ such that
  $\sigma(s)=\lambda(q)$.  By definition of a face,
  $\lambda^{-1}\sigma$ is an isometry from
  $\bDelta^{\dim(\sigma)}$ to a face of $\bDelta^{\dim(\lambda)}$ and
  so is a $\Q$-map by Lemma \ref{lemma:algvert}. So $s\in \Q^*$. On the other hand
  $\sigma$ is a face of $\kappa$, so $\kappa^{-1}\sigma$ is also
  a $\Q$-map, so
  $q_0:=\kappa^{-1}(\sigma(s))$ is in
  $\Q^*$. Now $\kappa(q_0)=\sigma(s)=\lambda(q)=x$ and so
  $x\in \kappa[\Q^*]$.
\end{proof}


\begin{Lemma}\label{lemma:QsubdivisionIFF}
  Suppose $S\subd T$. Then the following are equivalent:
  \begin{enumerate-(1)}
  \item $S\subd_\Q T$,
  \item $V(S)\subset \Q(T)$,
  \item \(\Q(S) \subset \Q(T)\),
  \item \(\Q(S) = \Q(T)\).
  \end{enumerate-(1)}
\end{Lemma}
\begin{proof}
  For (1) $\Rightarrow$ (2)
  suppose $S\subd_{\Q} T$. Let $\lambda\in S$ be any $n$-simplex for some $n$, and let $\kappa\in T$ be
  some simplex with $|\lambda|\subset |\kappa|$.  Then
  $\kappa^{-1}\lambda$ is a $\Q$-map
  (Definition~\ref{def:Q-subd}). Since $V(\bDelta^n)\subset\Q^*$, we
  have
  $$V(\lambda)=\lambda[V(\bDelta^n)]=\kappa[\kappa^{-1}\lambda V(\bDelta^n)]\subset \kappa[\Q^*]\subset \Q(T).$$
  Then \(V(S) \subset \Q(T)\).  For (2) $\Rightarrow$ (1), let
  $\lambda\in S$ be arbitrary. By (2), $V(\lambda)\subset\Q(T)$. Let
  $\kappa\in T$ be such that $|\lambda|\subset|\kappa|$, which exists
  since $S\subd T$.  Then $\kappa^{-1}\lambda$ is rectilinear
  (Lemma~\ref{lemma:AnySimplexRect}), and thus
  $$V(\kappa^{-1}\lambda)=\kappa^{-1}[V(\lambda)]\subset \kappa^{-1}[\Q(T)]\subset \Q^*$$
  where the last inclusion is by the second assertion of Lemma~\ref{lemma:Q-compatibilility}.
  By Lemma~\ref{lemma:algvert} $\kappa^{-1}\lambda$ is a
  $\Q$-simplex. By Definition~\ref{def:Q-subd} this implies that
  $S\subd_{\Q}T$.
  
  We now prove (2)$\iff$(3). Let  $\lambda\in S$ be
  arbitrary and let $\kappa$ be some simplex in $T$ such that $|\lambda|\subset |\kappa|$.
  By Lemma \ref{lemma:AnySimplexRect} $\kappa^{-1}\lambda$ is rectilinear.
  By using this and Lemmas~\ref{lemma:algvert}
  and \ref{lemma:Q-compatibilility}, we have the following sequence of equivalences:
  \begin{align*}
    &V(\lambda)\subset \Q(T) \iff V(\lambda)\subset \Q(T)\cap |\kappa|\stackrel{\text{\ref{lemma:Q-compatibilility}}}{\iff} V(\lambda)\subset \kappa[\Q^*] \iff \kappa^{-1}V(\lambda)\subset \Q^*\\
    \iff &V(\kappa^{-1}\lambda)\subset\Q^*
    \stackrel{\ref{lemma:algvert}}{\iff} (\kappa^{-1}\lambda)[\Q^*]\subset \Q^*
    \iff \lambda[\Q^*]\subset \kappa[\Q^*]\subset \Q(T).
  \end{align*}
  Since this follows for an arbitrary \(\lambda \in S\), we have the desired equivalence.

  We finally show (3)$\iff$(4). Consider the non-obvious direction and suppose that $\Q(S)\subset \Q(T)$. Since (3) implies (1), we have \(S \subd_{\Q} T\).
  Suppose $q\in \Q(T)$. Since $S$ is a subdivision of \(T\) and \(|S|=|T|\), there are
  $\lambda\in S$ and $\kappa\in T$ with
  $q\in|\lambda|\subset |\kappa|$. By
  Lemma~\ref{lemma:Q-compatibilility} $q\in \kappa[\Q^*]$. Let
  $q'=\kappa^{-1}(q)$. Then $q'\in\Q^*$. Let $q''=\lambda^{-1}(q)$.
  Now $q''=\lambda^{-1}(\kappa(q'))$. But
  $\kappa^{-1}\lambda$ is a $\Q$-map and so is its inverse (when
  defined), so $q''\in \Q^*$.  Thus,
  $q=\lambda(q'')\in \lambda[\Q^*]\subset \Q(S)$. This completes the proof of (3) $\iff$ (4).
\end{proof}

\begin{Lemma}\label{lem:SubdQTrans}
  The relation $\subd_\Q$ is reflexive and transitive.
\end{Lemma}
\begin{proof}
  For reflexivity, let $\sigma,\kappa\in T$ be such that
  $|\sigma|\subset |\kappa|$. We need to check that
  $\kappa^{-1}\sigma$ is a rectilinear $\Q$-map. But now $\sigma$ is a
  face of $\kappa$, so it is rectilinear by definition and a $\Q$-map
  because it is realized, in fact, by a permutation matrix.
  For transitivity, suppose
  that $Q\subd_\Q S\subd_\Q T$, let $\kappa\in Q$ and $\tau\in T$ be
  such that $|\kappa|\subset |\tau|$.  We need to show that
  $\tau^{-1}\kappa$ is a rectilinear $\Q$-simplex. Let $\sigma\in S$
  be such that $|\kappa|\subset|\sigma|\subset |\tau|$ which exists by
  the definition of a subdivision. By the assumption we have that
  $\tau^{-1}\sigma$ and $\sigma^{-1}\kappa$ are rectilinear
  $\Q$-simplexes. But $\tau^{-1}\kappa$ is their composition so we are
  done.
\end{proof}

\begin{Lemma}\label{lemma:Collection3}
    Let $S$ and $T$ be complexes.
  \begin{enumerate-(1)}
  \item \label{fact:Q-stellarQ1} If \(\bar z = (z_0, \dots, z_n)\) and $S=T*\bar z$, then
    $V(S)=V(T)\cup \Cup_{i \le n} \{|z_i|\}$.
  \item \label{fact:Q-stellarQ2} If $\bar z\subset \Q(S)$, then
    $S*\bar z\subd_{\Q} S$.
\end{enumerate-(1)}
\end{Lemma}
\begin{proof}
\ref{fact:Q-stellarQ1} It follows by   Definitions~\ref{def:StellarSubdivision}, \ref{def:Star},
  and~\ref{def:DeltaInComplex}.  \ref{fact:Q-stellarQ2} By
  Lemma~\ref{lem:SubdQTrans}, $S\subd_\Q S$ so by
  Lemma~\ref{lemma:QsubdivisionIFF}((1)$\Rightarrow$(2)),
  $V(S)\subset\Q(S)$. Now by \ref{fact:Q-stellarQ1},
  $V(S*\bar z)\subset \Q(S)$, so by
  \ref{lemma:QsubdivisionIFF}((2)$\Rightarrow$(1)) we have
  $S*\bar z\subd_{\Q} S$.      
\end{proof}

\begin{Lemma}\label{lemma:Adjust}
  Suppose that $T$ is a complex in a metric space $(X,d)$,
  $\bar t \subset X$ is locally finite, and $\e\colon X\to\R_+$ is
  any continuous function. Then there is a locally finite $\bar q\subset \Q(T)$
  and a simplicial map $h\colon T*\bar t\to T*\bar q$ such that
  $h(t_i)=q_i$ for all $i$ and $d(x,h(x))\le \e(x)$ for all $x\in X$.
\end{Lemma}
\begin{proof}
  For each $\tau\in T$, $|\tau|$ is compact, so $\e\rest |\tau|$ has a
  positive minimum. Let $\e_\tau=\min\{\e(x)\mid x\in|\tau|\}$.
  Since $\bar t$ is locally finite, one can find a sequence
  $(\e_i)_{i\in\N}\subset\R_+$ such that for each $\tau$,
  \begin{equation}
    \label{eq:etaubound}
    \sum_{i\in I(\tau)} \e_i\le \e_\tau\quad\text{ where }\quad I(\tau)=\{i\in\N\mid t_i\in|\tau|\}.
  \end{equation}
  As before, denote $\bar t^j=(t_0,\dots,t_{j-1})$ with $\bar t^0$
  being empty.  For all $i\in \N$, let $T_i=T*\bar t^{i}$ (for $i=0$,
  $T_0=T$).  Let $Q_0=T_0=T$. Let $h_0\colon |T_0|\to |Q_0|$ be the
  identity.  Let $\bar q^0$ be the empty sequence.
  Then $h_0$ is clearly simplicial from $T_0$ to~$Q_0=T*\bar q^{0}$.
  As induction
  hypothesis, suppose that $h_i$ and
  $\bar q^i=(q_0,\dots,q_{i-1})\subset \Q(T)$ have been defined such
  that $h_i$ is simplicial from $T_i$ to~$Q_i:=T*\bar q^i$,
  $h_i(t_j)=q_j$ for all $j<i$, and for all $\tau\in T$,
  \begin{equation}
    \label{eq:IndHypSup}
    \sup\left\{d(x,h_i(x))\mid x\in |\tau|\right\}\le \sum_{j\in J(\tau,i)}\e_j\quad\text{ where }\quad J(\tau,i)=\{j\in\N\mid t_j\in |\tau|,\ j<i\}.
  \end{equation}
  Note that by
  Lemma~\ref{lemma:Collection3}\ref{fact:Q-stellarQ2}
  \begin{equation}
    Q_i\subd_\Q T.\label{eq:ItIsAQsubd}  
  \end{equation}  
  Let $\tau_i=\Carr_{T_i}(t_i)$. 
  Then by Lemma~\ref{lemma:FindPointsQ},
  the set
  $h_i|\overset{\circ}\tau_i|\cap \Q(Q_i)\cap B_X(h_i(t_i),\e_i)$
  is non-empty, so pick $q_i$ from that set.
  Let $p_i=h_i^{-1}(q_i)$. Then $p_i\in\overset{\circ}{\tau_i}$,
  so by the first part of Lemma~\ref{lemma:CarrierRelation}, 
  $\Carr_{T_i}(p_i)=\tau_i$. By the second part of the same Lemma,
  $\Carr_{Q_i}(q_i)=\Carr_{Q_i}(h_i(t_i))=h_i\tau_i$.
  We can now apply Proposition \ref{prop:ExtendSimplicial} with 
  $S=T_i$, $T=Q_i$, $f=h_i$, $s=t_i$, $t=q_i$ to get a simplicial 
  $h_{i+1}\colon T_i*t_i\to Q_i*q_i$
  such that $h_{i+1}(t_i)=q_i$, $h_{i+1}(x)=h_i(x)$ for 
  $x\in |T_i|\setminus |\St_{T_i}(t_i)|$ and for all $x\in |T|$,
  \begin{equation*}
    d(x,h_{i+1}(x))\le d(x,h_i(x))+d(q_i, h_i(t_i))
  \end{equation*}
  which by the fact that $q_i\in B_X(h_i(t_i),\e_i)$ implies
  \begin{equation}
    d(x,h_{i+1}(x))\le d(x,h_i(x))+\e_i.\label{eq:Ineqhip1e} 
  \end{equation}
  Suppose $\tau\in T$ and $x\in |\tau|$. We need to show that
  $$d(x,h_{i+1}(x))\le \sum_{j\in J(\tau,i+1)}\e_j.$$
  Suppose first that $\tau\notin \St_{T_i}(t_i)$.
  Then either $x\notin |\St_{T_i}(t_i)|$ or $x$ is on the boundary
  of $|\St_{T_i}(t_i)|$. The boundary case follows by the continuity
  of $h_{i+1}$ from the other case, so we can assume that $x$
  is not in $|\St_{T_i}(t_i)|$. 
  Then $h_{i+1}(x)=h_i(x)$. On the other hand, $t_i\notin |\tau|$
  (by the definition of a star), so $J(\tau,i)=J(\tau,i+1)$. So
  by \eqref{eq:IndHypSup} we have
  $$d(x,h_{i+1}(x))=d(x,h_i(x))\le \sum_{j\in J(\tau,i)}\e_j= \sum_{j\in J(\tau,i+1)}\e_j.$$
  Suppose now that $\tau\in \St_{T_i}(t_i)$. Since we already dealt
  with $x$ which is on the boundary of the star, we can w.l.o.g.
  assume that $\tau\notin\Lk_{T_i}(t_i)$. Then, $t_i\in |\tau|$ and so
  $J(\tau,i+1)=J(\tau,i)\cup \{i\}$ and so by \eqref{eq:Ineqhip1e} and
  \eqref{eq:IndHypSup}
  we have
  $$d(x,h_{i+1}(x))\le d(x,h_i(x))+\e_i
  \le \sum_{j\in J(\tau,i)}\e_j+\e_i=\sum_{j\in J(\tau,i+1)}\e_j$$
  which was to be proven.
  Let us show that from this construction it follows for all $i$ that
  $d(x,h_i(x))$ is bounded by $\e(x)$. Let
  $x\in |T|$ and let $\tau\in T_i$ be such that $x\in |\tau|$. Note
  that $J(\tau,i)\subset I(\tau)$ where the terms are defined in
  \eqref{eq:etaubound} and \eqref{eq:IndHypSup}. By the construction, 
  \begin{equation}
    \label{eq:Boundonhi}
    d(x,h_i(x))\le \sum_{j\in J(\tau,i)}\e_j\subset \sum_{j\in I(\tau)}\e_j\stackrel{\eqref{eq:etaubound}}{\le}\e_{\tau}\le \e(x).
  \end{equation}
  In the case of a finite sequence $\bar t$,
  the proof is complete by setting $h=h_i$ and $\bar q=\bar q^i$ where
  $i$ is the length of $\bar t$. 

  Suppose $\bar t$ is infinite. By the local finiteness of $\bar t$ and the choice of $h_i(x)$, the sequence
  $(h_i(x))_{i\in\N}$ is eventually constant for any fixed $x\in
  |T|$. Thus, the limit $\lim_{i\to\infty} h_i(x)$ is well-defined and
  we can set $h(x)=\lim_{i\to\infty} h_i(x)$ for all $x\in |T|$.
  By \eqref{eq:Boundonhi}, $d(x,h(x))\le \e(x)$.
    
  By the choice of $q_i$ and $h_{i+1}$, we have that
  $d(t_i,q_i)\le d(t_i,h_i(t_i))+d(h_i(t_i),q_i)\le \e(t_i)+\e_i$.  Let
  $\tau\in T$ be such that $t_i\in |\tau|$. Assume w.l.o.g. that $\e$
  is bounded by $1$. Then $\e(t_i)+\e_i\le 2$, so we have
  $d(t_i,q_i)\le 2$. By local finiteness of $\bar t$, this implies
  that $\bar q$ is locally finite too.  Thus, $T*\bar q$ is
  well-defined. It remains to show that $h$ is simplicial from
  $T*\bar t$ to $T*\bar q$. First, let us show that $h$ is bijective.
  If $x,y\in |T|$ with $x\ne y$, pick large enough $i$ such that
  $h(x)=h_j(x)$ and $h(y)=h_j(y)$ for all $j>i$. Since $h_j$ is
  one-to-one, it follows that $h(x)\ne h(y)$ and so $h$ is one-to-one
  too. Suppose $y\in |T|$. By the local finiteness of $\bar t$, there
  is $i$ such that $h_j\rest B(y,1)=h_{j'}\rest B(y,1)$ for
  all $j,j'>i$. Thus, for any $j>i$ and any $x\in B(y,1)$, we have
  $h(x)=h_j(x)$.  By \eqref{eq:Boundonhi}, $h_j^{-1}(y)\in B(y,\e(y))\subset B(y,1)$ (we are still assuming w.l.o.g. that $\e(x)\le 1$),
  so $h(h_j^{-1}(y))=y$. This proves that $h$ is onto. By similar
  arguments of restricting $h$ to bounded neighborhoods, it is easy to
  see that $h$ and $h^{-1}$ are continuous, so it is a
  homeomorphism. Suppose $\sigma\in T*\bar t$. By definition, there is
  $i_0$ such that $\sigma\in T*\bar t^j$ for all $j>i_0$. Let
  $i_1\ge i_0$ be such that $h\rest |\sigma|=h_j\rest |\sigma|$ for
  all $j>i_1$. Then, since $h_j$ is simplicial from $T*\bar t^j$ to
  $T*\bar q^j$, we have that
  $h_j\circ\sigma=h\circ\sigma\in T*\bar q^j$ for all $j>i_1$.  Thus,
  $h\circ\sigma\in T*\bar q$. This completes the proof.
\end{proof}

\begin{Prop}\label{prop:SubdivisionAdjustment}
  Suppose $S$ and $T$ are two triangulations of a metric space $(X,d)$ which have a
  common subdivision. Suppose $\e\colon X\to\R_+$ is any continuous function.
  Then there are $\bar r\in \Q(S)$ and
  $\bar q\in \Q(T)$ and a simplicial map
  $f\colon S*\bar r\to T*\bar q$ such that for all $x\in X$,
  $d(x,f(x))\le\e(x)$.
\end{Prop}
\begin{proof}
  We can assume w.l.o.g. that $\e$ is bounded from above by~$1$.
  By Theorem~\ref{thm:CommonStellarSubdivision} let
  $\bar s$ and $\bar t$ be such that
  $S*\bar s=T*\bar t$.
  If $\bar s\subset \Q(S)$ and $\bar t\subset \Q(T)$,
  then let $f$ be the identity and we are done.
  Otherwise define $\delta\colon X\to \R_{+}$ by
  $$\delta(x)=\min\Big\{\frac{1}{2}\e(y)\mid y\in \overline{B(x,1)}\Big\}.$$
  By local compactness, $\delta$ is well-defined and always positive.
  Also $\delta(x)\le\frac{1}{2}\e(x)$ for all $x\in X$.
  By Lemma~\ref{lemma:Adjust},
  find $\bar r\subset \Q(S)$ and $\bar q\subset \Q(T)$
  such that there are simplicial maps
  $g\colon S*\bar s\to S*\bar r$
  and
  $h\colon T*\bar t\to T*\bar q$
  such that for all $x\in X$, $d(x,g(x))\le\delta(x)$
  and $d(x,h(x))\le\delta(x)$.
  Let $f=h\circ g^{-1}$. Clearly $f$ is simplicial from $S*\bar r$
  to $T*\bar q$. Let $x\in X$. We will show that $d(x,f(x))\le\e(x)$.
  First note that
  \begin{align*}
    d(x,f(x))&\le d(x,g^{-1}(x))+d(g^{-1}(x),h(g^{-1}(x)))\\
             &=d\left(g(g^{-1}(x)),g^{-1}(x)\right)+d\left(g^{-1}(x),h(g^{-1}(x))\right)\\
             &\le 2\delta(g^{-1}(x))\\
             &=2\min\Big\{\frac{1}{2}\e(y)\mid y\in \overline{B\big(g^{-1}(x),1\big)}\Big\}
  \end{align*}
  Let $y_0=g^{-1}(x)$. Then
  $d(x,y_0)=d(g(y_0),y_0)\le \delta(y_0)\le \frac{1}{2}\e(y_0)\le
  \frac{1}{2}<1$, so we have $x\in B(y_0,1)=B\big(g^{-1}(x),1\big)$.
  It follows that
  $$d(x,f(x))\le 2\cdot\frac{1}{2}\e(x)=\e(x).$$
  which completes the proof.
\end{proof}

\begin{Prop}
    \label{prop:CommonQsubd}
    Suppose $S$ and $T$ are both $\Q$-subdivisions of $Z$.
    Then there is $\bar z\subset\Q(Z)$ such that
    $Z*\bar z$ is a common $\Q$-subdivision of $S$ and $T$.
\end{Prop}
\begin{proof}
  Let us first prove it for finite complexes. Theorem 1.1 of
  \cite{AP24} implies that if $S$, $T$, and $Z$ are
  finite complexes in $\R^N$ with vertices in $\Q^N$ such that
  $S\subd_\Q Z$ and $T\subd_\Q Z$, then there is a finite sequence
  $\bar z\subset \Q^N$ such that $Z*\bar z$ is a $\Q$-subdivision of
  both $S$ and $T$.  So it is enough to show that given a finite
  complex $Z$, it can be embedded into $R^N$ for large enough $N$ via
  $f\colon |Z|\to \R^N$ such that $f[\Q(Z)]=\Q^N\cap f|Z|$ and $f|Z|$
  is the realization of a complex in \(\R^N\). This is standard, but
  let us give a proof for the sake of completeness.  For instance, let
  $N$ be the cardinality of $V(Z)$, let
  $\eta\colon V(Z)\mapsto \{\ee_0,\dots,\ee_{N-1}\}$ be a bijection,
  and map $\kappa\in Z$ linearly onto $\Delta[\eta V(\kappa)]$. It is
  easy to verify that this is as needed.

  Let us now consider the case of infinite complexes $S$, $T$ and $Z$.
  \begin{claim}
    There are $\bar s,\bar t\subset \Q(Z)$ such that
    $S*\bar s=T*\bar t$.
  \end{claim}
  \begin{proof}
    This proof is similar to that of \cite[Claim
    3.2.1]{IW26}.  Let $Z_0, Z_1, \dots $ be a
    locally finite sequence of finite subcomplexes of $Z$ such that
    $\Cup_{i\in\N}Z_i=Z$ and $Z_0$ is a single vertex. Let
    $\hat Z_i=\Cup_{j\le i}Z_j$, $S_i=S\rest Z_i$, $T_i=T\rest Z_i$,
    $\hat S_i=S\rest \hat Z_i$, and $\hat T_i=T\rest \hat Z_i$.  Then
    $S_i$ and $T_i$ are $\Q$-subdivisions of~$Z_i$ and by
    Lemma~\ref{lemma:QsubdivisionIFF}, $\Q(S_i)=\Q(T_i)=\Q(Z_i)$.

    By induction on $i$ we will find sequences $\bar s^i$ and
    $\bar t^i$ such that $\hat S_i*\bar s^i=\hat T_i*\bar t^i$.
    Suppose $\bar s^0=\bar t^0$ are equal to the empty sequence. Since
    $T_0=S_0=\hat T_0=\hat T_0=Z_0=\hat Z_0$ is a single vertex, we
    have $S_0*\bar s^0=T_0*\bar t^0$.  This starts the
    induction.  Suppose $\bar s^i,\bar t^i\subset \Q(\hat Z_i)\subset\Q(Z)$ have been
    found such that $\hat S_i*\bar s^i=\hat T_i*\bar t^i$. In particular, $S_{i+1}*\bar s^i=T_{i+1}*\bar t^i$. By the
    finitary case of this proposition (proved above), we can find
    $\bar r,\bar q\subset \Q(\hat Z_{i+1})\subset \Q(Z)$ such that
    $S_{i+1}*\bar s^i*\bar r=T_{i+1}*\bar t^i*\bar q$.  We now apply
    Lemma~\ref{lemma:ExtedingSubdivisions} (with
    $S_0=\hat S_{i}$, $S_1=S_{i+1}$, $S=\hat S_{i+1}$, $T_0=\hat T_i$,
    $T_1=T_{i+1}$, $T=\hat T_{i+1}$, $\bar s_0=\bar s^{i}$,
    $\bar s_1=\bar r$, $\bar t_0=\bar t^{i}$, and $\bar t_1=\bar q$) to obtain that
    $\hat S_{i+1}*\bar s^i*\bar r=\hat T_{i+1}*\bar t^i*\bar q$, so we
    can set $\bar s^{i+1}=\bar s^i\cat\bar r$,
    $\bar t^{i+1}=\bar t^i\cat \bar q$, and set
    $\bar s=\Cup_{i\in\N} \bar s^i$ and
    $\bar t=\Cup_{i\in\N}\bar t^i$.
  \end{proof}

  Now apply the claim to $S$ and $Z$ to get $\bar s,\bar z\subset\Q(Z)$
  such that $S*\bar s=Z*\bar z$. Then apply it to $Z*\bar z$ and $T$
  to get $\bar z',\bar t\subset \Q(Z)$ such that
  $(Z*\bar z)*\bar z'=T*\bar t$. Then, by ``interlacing''
  $\bar z$ and $\bar z'$, just like in the final part
  (after \eqref{thm:CommonStellarSubdivision}) of
  the proof of Theorem~\ref{thm:CommonStellarSubdivision}, we can find
  $\bar w\subset \Q(Z)$ such that $(Z*\bar z)*\bar z'=Z*\bar w$.
  Clearly $Z*\bar w$ is a common $\Q$-subdivision of $S$ and $T$.
\end{proof}

If $T$ is a complex and $f\colon X\to |T|$ is a homeomorphism, denote
by $f_*T:=\{f^{-1}\circ\kappa\mid \kappa\in T\}$ the \textbf{pullback}
of the triangulation. Similarly, if $f\colon |T|\to X$ is a
homeomorphism, a \textbf{pushforward} is given by $f^*T:=(f^{-1})_*T$.

\begin{Fact}\label{fact:SimplicialPushSubd}
  Suppose $T$ is a simplicial complex and $f\colon X\to |T|$ a
  homeomorphism.  Then $f_*T$ is a simplicial complex in \(X\), and
  $f\colon f_*T\to T$ is a simplicial map.  Dually, for pushforward,
  if $g\colon |T|\to X$ is a homeomorphism, then $g^*T$ is a
  simplicial complex in \(X\) and $g$ is simplicial from $T$ to
  $g^*T$. In particular, by Lemma~\ref{lem:SimplicialQQ},
  $g[\Q(T)]=\Q(g^*T)$.  Suppose $T_0$ and $T_0'$ are also complexes.  If
  $h\colon T_0\to T'_0$ is simplicial, then $h^*T_0=T_0'$, and if
  $T_0\subd T$ and $h\colon |T|\to X$ is a homeomorphism, then
  $h^*T_0\subd h^*T$.  Also, $h^{*}[T*\bar r]=h^*T*h[\bar r]$.
\end{Fact}

\begin{thm}[PL-homeomorphic complexes are $\Q$-PL-homeomorphic]\label{thm:QPL-homeo}
  Let $T$ and $T'$ be triangulations of metric spaces $(X,d)$ and
  $(X',d')$ respectively.  Suppose $h\colon T\to T'$ is a
  PL-homeomorphism and $\e\colon |T|\to \R_+$ is any continuous function.  Then
  there is a $\Q$-PL-homeomorphism $g\colon T\to T'$ such that for all
  $x\in |T|$, $d'(h(x),g(x))\le \e(x)$.
\end{thm}
\begin{proof}
  Let $h\colon T\to T'$ be a PL-homeomorphism. Then there are
  subdivisions $T_0$ and $T_0'$ of $T$ and $T'$ respectively such that
  $h$ is simplicial from $T_0$ to $T_0'$.  Let $S=h^*T$.  Then
  $T'_0=h^*T_0$ is a common subdivision of $S$ and $T'$.  Let
  $\delta\colon X'\to \R_+$ be defined by
  $\delta(x)=\e(h^{-1}(x))$. By
  Proposition~\ref{prop:SubdivisionAdjustment}, let
  $\bar r\subset \Q(S)$ and $\bar q\subset \Q(T')$ be such that there
  is a simplicial $f\colon S*\bar r\to T'*\bar q$ with
  $d(x,f(x))\le \delta(x)$ for all $x\in X'$. Let $R:=h_*[S*\bar r]$ and
  $Q:=T'*\bar q$.  Then $h^{-1}$ is simplicial from $S*\bar r$ to
  $R=T*h^{-1}[\bar r]$ where
  $h^{-1}[\bar r]=(h^{-1}(r_i))_{i\in \N}$. By
  Fact \ref{fact:SimplicialPushSubd}, $h\Q(T)=\Q(S)$, so
  $h^{-1}[\bar r]\subset \Q(T)$.  Let
  $g=f\circ h$. Then $g$ is simplicial from $R$ to $Q$ which are
  $\Q$-subdivisions of $T$ and $T'$ respectively by
  Lemma~\ref{lemma:Collection3}\ref{fact:Q-stellarQ2}. Let us check
  that $d'(h(x),g(x))\le \e(x)$.  We have
  $$d'(h(x),g(x))= d'(h(x),f(h(x))\le \delta(h(x))=\e(h^{-1}(h(x)))=\e(x)$$
  which completes the proof.
\end{proof}



\vspace{10pt}

\subsubsection{The classification theorem}
\label{sssec:PLClass}

In this section we will use the results proved so far
to show that each simplicial complex can be associated with
a sorted complemented algebra as a complete invariant of PL-homeomorphism.

\begin{Def}\label{def:QPolyhedron}
  Let $T$ be a simplicial complex. A \textbf{polyhedron in} $T$ is a
  set $P\subset |T|$ such that for some subcomplex $P'$ of a
  subdivision of $T$ we have $P=|P'|$.  It is \textbf{compact}
  (resp. \textbf{co-compact}), if $P'$ can be chosen to be finite
  (resp. co-finite, i.e. $T\setminus P'$ is finite).  It is a
  $\Q$-\textbf{polyhedron} if we can choose $P'$ such that $P'$ is a
  subcomplex of a $\Q$-subdivision of~$T$.
\end{Def}

\begin{Fact}\label{fact:CanChooseFiniteStellarSubd}
  If $P$ is a $\Q$-polyhedron in $T$, then by
  Proposition~\ref{prop:CommonQsubd} there is $\bar z\subset \Q(T)$
  such that $P$ is the realization of a subcomplex of $T*\bar z$. If
  $P$ is compact or co-compact, then $\bar z$ can be chosen to be
  finite.
\end{Fact}

\begin{Lemma}\label{lemma:QpolyhIFF}
  Let $S$ and $T$ be complexes.  If $S\subd_\Q T$, then $P$ is a
  $\Q$-polyhedron in $S$ if and only if it is a $\Q$-polyhedron in~$T$. If
  $f\colon S\to T$ is simplicial, then $P$ is a $\Q$-polyhedron of $S$
  if and only if $fP$ is a $\Q$-polyhedron of~$T$.
\end{Lemma}
\begin{proof}
  Suppose $S\subd_\Q T$ and $P$ is a $\Q$-polyhedron of $S$. Then by
  definition there is $S'\subd_\Q S$ such that $P=|P'|$ for a
  subcomplex $P'$ of $S'$. By Lemma~\ref{lem:SubdQTrans}
  $S'\subd_\Q T$, so $P$ is a $\Q$-polyhedron in $T$. If $P$ is a
  $\Q$-polyhedron in $T$, let $T'\subd_\Q T$ be such that $P=|P'|$ for
  some $P'\subset T'$. By Proposition~\ref{prop:CommonQsubd} there is
  a common $\Q$-subdivision $Z$ of $T'$ and $S$.  Then
  $P=|Z\rest P'|$, so it is a $\Q$-polyhedron in~$S$.
    
  For the second part, assume that $f\colon S\to T$ is simplicial and
  $P$ is a $\Q$-polyhedron in $S$. Let $S'\subd_\Q S$ be such that $P=|P'|$
  for some $P'\subset S'$.  Then $f$ is simplicial from $S'$ to
  $f^*S'$ (Fact~\ref{fact:SimplicialPushSubd}) which is a
  $\Q$-subdivision of $T$ (to see this, use that \(S' \subd_{\Q} T\)
  and so \(\Q(S')=\Q(T)\) by Lemma \ref{lemma:QsubdivisionIFF}, but by
  Lemma \ref{lem:SimplicialQQ} $f\Q(S')=\Q(f^*S')$, and thus
  \(f^*S' \subd_{\Q} T\) by Lemma \ref{lemma:QsubdivisionIFF}), so
  $fP$ is a $\Q$-polyhedron in~$T$. For the converse, let $fP$ be a
  $\Q$-polyhedron in~$T$. Then by definition there are
  \(T' \subd_{\Q} T\) and a subcomplex \(P'\subset T'\) such that
  \(fP=|P'|\). It follows that \(f\) is simplicial from \(f_*T'\) to
  \(T'\), and \(f_*T'\) is a subdivision of \(S\). Then \(f^{-1}fP=P\)
  is a $\Q$-polyhedron in~$S$.
\end{proof}

\begin{Def}\label{def:BetaFromT}
  Suppose $T$ is a simplicial complex and $X=|T|$.  Let
  $$\beta_T=\{\int_{X}(P)\mid P\text{ is a compact or co-compact }\Q\text{-polyhedron in }T\}$$
\end{Def}

\begin{prop}\label{prop:CCLCP}
  Let $T$ be a simplicial complex and $X=|T|$.
  Then $(X,\beta_T)$ is a CCLCP basis space, and $\beta_T$ is a basis
  for the original topology on~$X$.
\end{prop}
\begin{proof}
  First let us show that $\beta_T$ is countable. Note that \(\Q(T)\)
  is countable.  For each finite $\bar z\subset\Q(T)$ denote by
  $C(\bar z)$ the set of all finite and co-finite subcomplexes of
  $T*\bar z$. Then $C(\bar z)$ is countable.  Let
  $$Z'=\Cup\{C(\bar z)\mid \bar z\in \Q(T)^{<\o}\}.$$
  Then $Z=\{|P'|\mid P'\in Z'\}$ contains all compact and co-compact
  $\Q$-polyhedra in $T$. Since $\Q(T)^{<\o}$ is countable, also
  $\beta_T$ is, because the interior operation defines a map from $Z$
  onto~$\beta_T$.
  
  Let us show that $\beta_T$ is a basis for the original topology
  on~$X$.  Each $b\in\beta_T$ is clearly open in $X$, so it is enough
  to show that for all open $U\subset X$ and all $x\in U$ there is
  $b\in\beta_T$ with $x\in b\subset U$. So let $x\in U\subset X$ be
  such.  Let $S\subd_\Q T$ be a \(\Q\)-subdivision so fine that
  $|\St_S(x)|\subset U$. One can obtain it by taking repeated
  barycentric subdivisions.  Note that barycentric subdivisions are
  $\Q$-subdivisions.  Now, $\int_X(|\St_{S}(x)|)\in\beta_T$, so we are
  done (Condition~\ref{def:simplicial_complex-3} of
  Definition~\ref{def:simplicial_complex} ensure that the interior of
  the star of any $x$ is an open neighborhood of~$x$). This also
  proves condition~\ref{def:BC5} of Definition~\ref{def:CLCP}.

  Let us prove condition \ref{def:BC9}. Suppose
  $b_1\in\beta_T$ is such that $\overline{b_1}$ is compact 
  and $b_2$ such that $\overline{b_1}\subset b_2$.
  Then $\overline{b_1}$ is a compact $\Q$-polyhedron and
  $b_2$ is an open neighborhood of each point of $\overline{b_1}$.
  Let $S$ be a $\Q$-subdivision of $T$ such that both
  $\overline{b_1}$ and $\overline{b_2}$ are 
  realization of subcomplexes of $S$.
  By compactness and Condition~\ref{def:simplicial_complex-3}
  of Definition~\ref{def:simplicial_complex} there is
  an open neighbourhood $U$ of $\overline{b_1}$ contained into $b_2$
  which intersects only finitely many simplexes of $S$.
  Let $S'$ be a sufficiently fine $\Q$-subdivision of $S$
  such that $P:=\St_{S'}(\overline{b_1})$ is contained in $U$.
  For instance, obtain $S'$ by repeated barycentric 
  subdivisions. Then $|P|$ is a compact $\Q$-polyhedron.
  Let $b:=\int(|P|)$. Then $b\in\beta_T$,  
  $\overline{b_1}\subset b$, and $\overline{b}\subset U\subset b_2$.
  
  Let us make sure that also conditions \ref{def:BC3}--\ref{def:BC4}
  are satisfied.  Condition \ref{def:BC3} follows because $|T|$ is
  Polish, see e.g. Fact~\ref{fact:locally_compact_Ksigma}.  By the
  construction of $\beta_T$, conditions \ref{def:BC1}--\ref{def:BC2}
  are also obvious.  Condition \ref{def:BC7} follows from
  Fact~\ref{fact:InteriorReg}. Condition~\ref{def:BC4} is
  automatically satisfied once we enumerate $\beta_T$ because
  $\beta_T$ was defined as a set.
\end{proof}

\begin{thm}[Characterization of $\homeo^{PL}$ by $\,\equiv$]%
    \label{thm:ClassificationOFPL-complexesByBasisSpaces}
    Let $S$ and $T$ be simplicial complexes. Then $S\homeo^{PL}T$ if and only if $(|S|,\beta_S)\equiv(|T|,\beta_T)$.
\end{thm}
\begin{proof}
    Suppose $S\homeo^{PL}T$. By Theorem \ref{thm:QPL-homeo}, 
    $S\homeo^{\Q PL}T$,
    so there are $S'\subd_\Q S $ and $T'\subd_\Q T$
    and a simplicial map $f\colon S'\to T'$. Let $b\in \beta_S$ and let $P$
    be a $\Q$-polyhedron in $S$ such that $b=\int(P)$. Applying Lemma~\ref{lemma:QpolyhIFF}
    three times,
    we have that $fP$ is a $\Q$-polyhedron in $T$. 
    Since $f$ is a homeomorphism, it preserves interior, so $f[b]\in\beta_T$.
    It follow by a symmetric argument that if $b\in \beta_T$, then $f^{-1}[b]\in\beta_S$.

    Conversely, if the basis spaces are equivalent, then the map witnessing it is
    a homeomorphism $h$ from $(|S|,\la\beta_S \ra)$ to
    $(|T|,\la\beta_T\ra)$ which preserves the basic sets and hence
    preserves all the $\Q$-polyhedra. By an application of Theorem
    3.6(c.)  of \cite{hudson1969pl}, $h$ is a PL-homeomorphism (by our
    definition; note that in \cite{hudson1969pl} the definitions, like
    that of a PL map, are different).
\end{proof}

\begin{Def}\label{def:A_T}
  For a simplicial complex $T$, let $\AA_T:=\psi(|T|,\beta_T)$, where \(\psi\) is the map of Definition~\ref{def:map_from_BBB_to_AAA}. 
  By Proposition~\ref{prop:CCLCP}
  $(|T|,\beta_T)$ is a CCLCP basis space and so by
  Lemma~\ref{lemma:CCLCP_is_AAA}, $\AA_T$ is a sorted
  complemented algebra.
\end{Def}

\begin{thm}[Characterization of $\homeo^{PL}$ by $\,\cong_\AAA$]%
  \label{thm:ClassificationOfSimplicialComplexes}
  Let $S$ and $T$ be simplicial complexes.
  Then $S\homeo^{PL}T$ if and only
  if $\AA_{S}\cong \AA_{T}$.  
\end{thm}
\begin{proof}
  Follows by applying
  Theorem~\ref{thm:ClassificationOFPL-complexesByBasisSpaces} and
  Theorem~\ref{thm:PsiReduction}.
\end{proof}

\begin{Def}\label{def:PO_Qpolyhedra}
  Let $T$ be a simplicial complex.  Let $P_T$ be the set of all
  compact $\Q$-polyhedra in $T$
  (which are regular by Fact~\ref{fact:InteriorReg}). Let $\le_T$ be a partial ordering
  on $P_T$ defined by $p\le_Tq$ iff $p$ is contained in the interior
  of~$q$ or $p=q$.
\end{Def}

\begin{thm}[Characterization of $\homeo^{PL}$ by partial orders]%
  \label{thm:PL-class-simple}
  Let $S$ and $T$ be simplicial complexes.
  Then $S\homeo^{PL}T$ if and only if
  $(P_S,\le_S)\cong (P_T,\le_T)$.
\end{thm}
\begin{proof}
  Let $T$ be a simplicial complex.  Let $(b_n)_{n\in\N}$ be an
  enumeration of $\beta_T$.  Let $\AA:=\AA_T$. Recall that $A=\N$. By
  the definition of $\beta_T$ and condition~\ref{def:Psi_item5} of
  Definition~\ref{def:map_from_BBB_to_AAA}, we have that $n\in K^\AA$
  if and only if $\overline{b_n}\in P_T$. We also have that
  $n\le_{\AA} m$ if and only if either \(b_n=b_m\) or
  $\overline{b_n}\subset b_m$ if and only if $b_n\le_T b_m$. Thus, the
  map $n\mapsto \overline{b_n}$ is an isomorphism from
  $(K^{\AA},\le_{\AA})$ to $(P_T,\le_T)$.  By
  Definition~\ref{def:map_from_BBB_to_PPP}, we have that
  $(P_T,\le_T)\cong \psi^*(|T|,\beta_T)$.  Let $S$ be another complex
  and let $\AA'=\AA_S$.  By the same argument as above,
  $(P_S,\le_S)\cong (|S|,\beta_S)$. The statement of the theorem
  follows from Proposition \ref{prop:BasisToPPP} and
  Theorem~\ref{thm:ClassificationOFPL-complexesByBasisSpaces}.
\end{proof}



\vspace{10pt}

\subsection{Classification of 2 and 3-manifolds up to homeomorphism}
\label{sssec:Classification23man}

An \textbf{$n$-manifold with boundary} is a separable metric space each of whose
points has a closed neighborhood homeomorphic to~$\bDelta^n$. 
It is \textbf{without boundary}, if every point has an open neighborhood homeomorphic to $\obDelta^n$. 

A triangulation of an $n$-manifold $M$ is a simplicial complex $T$
such that $|T|=M$.  Bing
and Moise proved that $3$-manifolds, with or without boundary, admit a
triangulation which is unique up to PL-homeomorphism (the so-called Hauptvermutung for $3$-manifolds); see \cite[Theorem 9.1]{Mo54},
\cite[Theorems 3 and 4]{Mo52}, 
\cite[Theorems 6]{Bi59}, \cite[Theorem 5 and Corollary 1]{Bing54}.
The same holds for $2$-manifolds; see
\cite{Mo77} for statements and classical references, or \cite{Ra25}.

\begin{Fact}\label{fact:MoiseBing}(\cite{Mo52,Mo77})
  Assume $n\in \{2,3\}$, and let $M$ be an $n$-manifold, with or without boundary.
  Then there is a triangulation $T$ of $M$, and any two
  triangulations of $M$ are PL-homeomorphic. Consequently, two $n$-manifolds are homeomorphic if and only if their respective triangulations are PL-homeomorphic.
\end{Fact}

For $n\in \{2,3\}$, associate to each $n$-manifold $M$ the sorted
complemented algebra $A_{T(M)}$, where $T(M)$ is a triangulation of
$M$. Recall the notation $A_T$ from Definition~\ref{def:A_T}. 

\begin{thm}[Characterization of homeomorphism on $2$ and $3$-manifolds by $\cong_\AAA$]%
   \label{thm:ClassificationOfManifolds}
   Let $n\in \{2,3\}$.
   Let $M$ and $M'$ be $n$-manifolds (with or without boundary).
   Then $M$ is homeomorphic to $M'$ if and only
   if $A_{T(M)}$ is isomorphic to~$A_{T(M')}$.  
\end{thm}
\begin{proof}
  Follows by applying 
  Fact~\ref{fact:MoiseBing} and Theorem~\ref{thm:ClassificationOfSimplicialComplexes}.
\end{proof}

Similarly as in the previous section, let us state the following version of the result.

\begin{Cor}[Characterization of homeomorphism on $2$ and $3$-manifolds by partial orders]%
  \label{cor:Manifolds-class-simple}
  Let $M$ be a $2$- or $3$-manifold (with or without boundary), and
  $T$ any triangulation of $M$. 
  Then the isomorphism type $(P_T,\le_T)$ completely determines the
  homeomorphism type of~$M$.
\end{Cor}
\begin{proof}
  Follows by Fact \ref{fact:MoiseBing} and Theorem~\ref{thm:PL-class-simple}.
\end{proof}

\vspace{10pt}

\subsection{Classification of Cantor sets in $\R^3$ up to conjugacy}
\label{sssec:Cantor1}

A Cantor set in $\R^3$ is a compact perfect totally disconnected set.
The space of such Cantor sets and the equivalence relation of
conjugation on them were studied in \cite{GKB13}.
Corollary~\ref{cor:CantorSets} below gives the first part of an answer
to Question 5.5 of \cite{GKB13}. The full answer is given in
Section~\ref{ssec:BorelClassCantor},
Theorem~\ref{thm:CantorBorelRed}. Two Cantor sets $C$ and $C'$ in
$\R^3$ are \textbf{conjugate}, if there exists a homeomorphism
$h\colon \R^3\to \R^3$ such that $h[C]=C'$.  If $U\subset\R^3$ is an
open set, let $T(U)$ be a triangulation of~$U$.

\begin{Cor}[Characterization of Cantor set equivalence]%
  \label{cor:CantorSets}
  Let $C,C'\in \R^3$ be Cantor sets. They are conjugate if and only if
  $A_{T(\R^3\setminus C)}$ and $A_{T(\R^3\setminus C')}$ are isomorphic algebras.
\end{Cor}
\begin{proof}
  Notice that the complement of a Cantor set of $\R^3$ is an open
  $3$-manifold. Then by Theorem~\ref{thm:ClassificationOfManifolds},
  it is enough to show that $C,C'$ are conjugate iff $\R^3\setminus C$
  and $\R^3\setminus C'$ are homeomorphic. This is a consequence of
  \cite[Theorem~4.1]{CM83} and the remark right after its proof. To
  see this, we have to show that a Cantor set is always locally
  non-separating in $\R^3$. But this follows directly from
  \cite[Thm~IV.4]{HurewiczWallman1941}.
\end{proof}



\vspace{10pt}

\section{Borelness of the classifications}
\label{sec:Borelification}

\vspace{10pt}

\subsection{Basic definitions of descriptive set theory}
\label{ssec:BasicDST}

\subsubsection{Polish spaces and Borel reductions}

A \textbf{Polish metric space} is a separable complete metric space.
A topological space is \textbf{Polish}, if it is separable and has a
metric which makes it into a Polish metric space.  If no confusion
ensues, we sometimes call Polish metric spaces just Polish spaces.
The collection of \textbf{Borel subsets} of a Polish space is the
smallest $\sigma$-algebra containing the basic open sets.  Closed sets
and $G_\delta$ sets (countable intersection of open sets) of a Polish
space are Polish spaces in the subspace topology~\cite{Kec95}.  The
product and sum of a countable (possibly finite) sequence of Polish
spaces is a Polish space.  An important Polish space that we often use
is given by the following.

\begin{defn}
  \label{def:K(X)}
  Let $X$ be a topological space, and consider the space $K(X)$ of all
  compact subsets of $X$ equipped with the \textbf{Vietoris topology}
  which is generated by the sets of the form
  $$\{K \in K(X) \mid K \subseteq U\}\ \text{ and }\{K \in K(X) \mid K \cap U \neq \es\},$$
  for $U$ open in $X$. 
  Let now $(X,d)$ be a metric space with $d \le 1$, i.e.\
  $d(x,y) \le 1$ for every $x,y \in X$. We define the
  \textbf{Hausdorff metric} $d_H$ on $K(X)$ as follows:
  $$d_H(K,L)=
  \begin{cases}
    0 & \text{if } K = L = \es\\
    1 & \text{if exactly one of } K,L \text{ is } \es\\
    \max\{\delta(K,L),\delta(L,K)\} & \text{if } K, L \neq \es 
  \end{cases},
  $$
  where $\delta(K,L)$ is the smallest $\delta\in \R_{\ge 0}$ such that
  $K$ is contained in the \textbf{collar} \(L_\delta\) of $L$ defined by
  $$L_\delta = \{x\in X\mid d(x,L)\le \delta\}.$$
  It is a standard result that the
  Hausdorff metric is compatible with the Vietoris topology~\cite[Exercise 4.21]{Kec95}.
  If $X$ is Polish, so is $K(X)$, and if in addition $X$ is compact,
  so is $K(X)$ (see \cite[Theorems 4.25 and 4.26]{Kec95}).
\end{defn}

Below are some known properties of basic relations defined on $K(X)$.

\begin{Fact}\cite[Exercise 4.29]{Kec95}\label{fact:K(X)}
  Let $X$ be a metrizable.
  \begin{enumerate-(i)}
  \item The relation ``$x \in K$'' is closed, i.e.,
    $\{(x,K) \mid x \in K\}$ is closed in $X \times K(X)$.
  \item\label{prop:K(X)-2} The relation ``$K \subseteq L$'' is closed, i.e.,
    $\{(K,L)\mid K \subseteq L\}$ is closed in $K(X)^2$.
  \item The relation ``$K \cap L \neq \es$'' is closed in $K(X)^2$. \label{prop:K(X)-3}
  \item\label{prop:UnionContinuous} The map $K(X)\times K(X)\to K(X)\colon(K_0,K_1)\mapsto K_0\cup K_1$ is continuous.
  \item\label{prop:K(X)-4} If $Y$ is metrizable, then the map
    $(K, L) \mapsto K \times L$ from $K(X) \times K(Y)$ into
    $K(X \times Y)$ is continuous.
  \end{enumerate-(i)}
\end{Fact}

A \textbf{standard Borel space} is a pair $(X, \BB)$ where $X$ is a
set, $\BB$ is a $\sigma$-algebra on $X$, and there is a Polish
topology on $X$ for which $\BB$ is precisely the collection of Borel
sets. The elements of \( \BB \) are called Borel sets of~\((X,\BB)\).
In particular, every Polish space is standard Borel when equipped with
its \(\sigma\)-algebra of Borel sets. The product and sum of a
countable (possibly finite) sequence of standard Borel spaces are
standard Borel spaces. Moreover, if $(X, \BB)$ is a standard Borel
space and $Y \subseteq X$ is in $\BB$, then $(Y, \BB\rest Y)$ is also
a standard Borel space where $\BB\rest Y=\{B\cap Y\mid B\in\BB\}$ is
the \textbf{induced Borel structure} on $Y$. One can consider the
induced Borel structure $\BB\rest Y$ even if $Y$ is not in $\BB$ (as
we do below in Definition~\ref{def:BorelReducibility}), although then
$(Y,\BB\rest Y)$ might not be a standard Borel space.  A particularly
important construction of a standard Borel
space 
is given in the following example.

\begin{exmp}
  Given a topological space $X$, the collection $F(X)$ of all its
  closed subsets can be equipped with the $\sigma$-algebra
  $\BB_{F(X)}$ generated by the sets of the form
  $$\{F \in F(X) \mid F \cap U \neq \es\},$$
  for $U \subseteq X$ non-empty open. It turns out that if $X$ is
  Polish, then $(F(X), \BB_{F(X)})$ is a standard Borel space, called
  the \textbf{Effros Borel space} (see \cite[Theorem 12.6]{Kec95}).
\end{exmp}

The following results are basic facts about the Effros Borel space.

\begin{Fact}\cite[Exercise 12.11]{Kec95}\label{fact:F(X)}
Let $X$ be Polish.
\begin{enumerate-(i)-r}
\item\label{prop:F(X)-1} $K(X)$ is a Borel set in $F(X)$.
\item\label{prop:F(X)-5} The relation ``$x \in F$'' in $X\times F(X)$ is Borel. 
\item\label{prop:F(X)-2} The relation ``$F_0 \subseteq F_1$'' in $F(X)^2$ is Borel. 
\item\label{prop:F(X)-3} The set of regular closed sets in $X$ is
  Borel in $F(X)$ (cf. Fact \ref{fact:InteriorReg})
\item\label{prop:F(X)-4} The relation ``$F \cap K = \es$'' is Borel in $F(X) \times K(X)$.
\end{enumerate-(i)-r} 
\end{Fact}

If we are only interested in the Borel structure, we treat Polish
spaces as standard Borel spaces or vice versa, if there is no loss of
generality, or danger of confusion.  Let $X$ and $Y$ be Polish or
standard Borel spaces. A function $\varphi \colon X \to Y$ is
\textbf{Borel} if the preimage of any Borel subset of $Y$ is Borel in
$X$.

\begin{Fact}\cite[Theorem 12.13]{Kec95}\label{prop:selection_thm}
  Let $X$ be Polish. There is a sequence of Borel functions
  $d_n\colon F(X) \to X$, such that for non-empty $F \in F(X)$,
  $\{d_n(F)\mid n\in\N\}$ is dense in $F$. \qed
\end{Fact}

We denote by $\DD(F)$ the set $\{d_n(F)\mid n\in\N\}$ as defined above.

\begin{Lemma}\label{lemma:SeqClosureBorel}
  Let $X$ be a Polish space.
  The map $f\colon X^\N\to F(X)$ given by
  $$f((x_n)_{n\in\N})=\overline{\{x_n\mid n\in\N\}}$$
  is continuous and hence Borel.
\end{Lemma}
\begin{proof}
  Let $U\subseteq X$ be open and let
  $\tilde U=\{F\in F(X)\mid F\cap U\ne\es\}$. Then
  $f^{-1}[\tilde U]=\Cup_{m\in\N}\{(x_n)_{n\in\N}\mid x_m\in U\}$
  which is a union of open sets in the product topology of~$X^\N$.
\end{proof}

\begin{prop}\label{prop:StrongComplementBorel}
  Let $X$ be a metric space. The map $c\colon F(X)\times F(X)\to F(X)$ given by
  $c(A,B)=\overline{A\setminus B}$ is Borel.
\end{prop}
\begin{proof}
  Let $U\subseteq X$ be open and let
  $\tilde U=\{F\in F(X)\mid F\cap U\ne\es\}$.
  Now
  \begin{align*}    
    c^{-1}[\tilde U]&=\{(A,B)\mid \overline{A\setminus B}\cap U\ne \es\}\\
   &=\{(A,B)\mid (\DD(A)\setminus B)\cap U\ne\es\}\\
   &=\{(A,B)\mid \exists n\in \N (d_n(A)\in U\land d_n(A)\notin B)\}\\
   &=\Cup_{n\in\N}\{(A,B)\mid d_n(A)\notin X\setminus U\}\cap \{(A,B)\mid d_n(A)\notin B\}
  \end{align*}
  which is Borel by Fact \ref{fact:F(X)}\ref{prop:F(X)-5}.
\end{proof}

We often make use of the \textbf{universal Urysohn space} $\U$,
referring the reader to \cite[Section 1.2]{Gao09} for the relevant
definitions and proofs. Given any Polish metric space $X$, using the
Kat\v{e}tov construction one can canonically construct a Polish metric
space $\U_X$. This construction satisfies the following for all Polish
metric spaces $X$ and $Y$:
\begin{itemize}
\item $\U_X$ contains (a canonical isometric copy of) $X$, and every
  isometry $\iota\colon X \to Y$ can be extended to an isometry
  $\iota^*\colon \U_X \to \U_Y$;
\item $\U_X$ is isometric to $\U_Y$ (the so-called Urysohn property).
\end{itemize}
Let now $\U$ be the space $\U_\R$: by the Urysohn property, a metric
space is Polish if and only if it is isometric to a closed subspace of
$\U$. It is thus natural to regard the Effros Borel space $F(\U)$ of
closed subspaces of $\U$ as the standard Borel space of all Polish
metric spaces. We assume that $\es\in F(\U)$ and that it is an
isolated point in~$F(\U)$.

We say that a Polish metric space $X$ is \textbf{Heine-Borel} if any
closed bounded subset of $X$ is compact. One can equivalently express
the property of being Heine-Borel in $\U$ in the following form:

\begin{Fact}\label{fact:HB_for_subsets_of_U}
  A set $X\in F(\U)$ is Heine-Borel if and only if $\overline{B(x,n)}\cap X$
  is compact for all $x\in \U$, $n\in\N$ if and only if for some
  $x\in\U$ the set $\overline{B(x,n)}\cap X$ is compact for all $n$.  \qed
\end{Fact}

As usually, we say that a subset $A$ of a topological space $X$ is $K_\sigma$ if
$A=\Cup_{n\in\N} K_n$ for some $K_n \in K(X)$.

\begin{Fact}\cite[Theorem 5.3]{Kec95}%
  \label{fact:locally_compact_Ksigma}
  Let $X$ be Hausdorff and locally compact. Then $X$ is metrizable and
  $K_\sigma$ if and only if $X$ is Polish. \qed
\end{Fact}

The following is a relevant result involving Heine-Borel metric
spaces:

\begin{Fact}(\cite{WJ87, Vau37})%
  \label{fact:locally_compact_HB}
  If $X$ is a $K_\sigma$, locally compact, metrizable space, then
  there is a compatible metric on $X$ which is Heine-Borel. \qed
\end{Fact}

We now discuss Borel reducibility, a tool of descriptive set theory
which is useful to gain insight into some natural and interesting
classification problems.

\begin{defn}\label{def:BorelReducibility}
  Suppose $X$ and $Y$ are standard Borel spaces and 
  $E$ and $F$ are equivalence relations on $X$ and $Y$
  respectively. We say that $f\colon X\to Y$ is a \textbf{reduction}
  of $\mathrel{E}$ to $\mathrel{F}$, if for all $x_0,x_1\in X$ we have
  $$x_0\mathrel{E}x_1\iff f(x_0)\mathrel{F}f(x_1).$$
  If $f$ is Borel, then it is a \textbf{Borel reduction}.
  In this case we say that $E$ is \textbf{Borel reducible} to $F$, 
  and write $E \le_B F$. We finally write $E\sim_B F$ when
  $E\le_B F$ and $F\le_B E$.
\end{defn}

Let $L$ be any countable (first-order) vocabulary and let $\Mod(L)$ be
the set of all countable $L$-structures with domain $\N$. This space
can be viewed as a Polish space, in fact homeomorphic to the Cantor
space $2^\N$ (see \cite{Hjo00}).  Denote by $\cong_{\Mod(L)}$ the
isomorphism relation on $\Mod(L)$ and more generally, if
$X \subseteq \Mod(L)$ is closed under isomorphism, denote
$ {\cong_{\Mod(L)} \rest X}$ by ${\cong_X}$.

\begin{Fact}\cite[Theorem 16.8]{Kec95}%
  \label{fact:borel_subsets_of_Mod(L)}
  The Borel subsets of $\Mod(L)$ which are closed under isomorphism
  are exactly those whose elements are codes for models satisfying
  some $L_{\omega_1\omega}$-sentence $\varphi$, where
  $L_{\omega_1\omega}$ is the infinitary logic obtained from the usual
  first-order logic by further allowing the use of (infinite)
  countable conjunctions and disjunctions.
\end{Fact}

It follows that most common classes of countable structures, such as
those of countable trees, countable linear orders, countable groups,
countable sorted complemented algebras (see Section \ref{ssec:SCA}),
form Borel subsets of~$\Mod(L)$ for some $L$, and hence are standard
Borel spaces.

An equivalence relation $E$ on a Borel space $X$ is called
\textbf{classifiable by countable structures} if
$E \le_B {\cong_{\Mod(L)}}$ for some~$L$.  

\vspace{10pt}

\subsubsection{Spaces of embeddings}

In this section we define some standard Borel spaces of functions
which are useful in the sequel. A map is \emph{an embedding} if it is
a homeomorphism onto its image. A \emph{partial embedding} from $X$
to $Y$ is an embedding from some $Z\subset X$ to~$Y$.

\begin{Def}\label{def:PartEmb}
  Let $X=(X,d)$ and $Y=(Y,d')$ be metric spaces.  Let $\PartEmb(X,Y)$
  be the set of all graphs of partial embeddings from $X$ to $Y$ with compact
  domain. If $F\in \PartEmb(X,Y)$, then $F=g[C]$ where $C\subset X$ is compact
  and $g\colon C\to X\times Y$ is given by $g(x)=(x,f(x))$ for some
  continuous $f\colon C\to Y$, so $F$ is compact. Thus,
  $\PartEmb(X,Y)\subset K(X\times Y)$ and we equip it with the
  subspace topology.  If $f$ is an embedding from a compact
  $C\subset X$ to $Y$, we denote by $\graph(f)$ the corresponding
  element of $\PartEmb(X,Y)$.
\end{Def}

Given the product $\prod_{i \in I} X_i$ of a family of topological
spaces $(X_i)_{i \in I}$, we denote by $\prj_{j}$ or
$\prj_{X_j}$ the projection function $\prod_{i \in I} X_i \to X_j$
onto the $j$-th coordinate, for each $j \in I$.

\begin{Lemma}\label{lemma:PartEmbContainsAllEmbeddings}
  Suppose $(X,d),(Y,d')$ are metric spaces. Then $F\in \PartEmb(X,Y)$
  if and only if
  \begin{align}
    &\forall \e\in\R_+\exists\d\in\R_+ \forall ((x,y),(x',y'))\in (F\times F)\big(d(x,x')<\delta \rightarrow d'(y,y')\le \e\big)\label{def:PartEmb1}\\
    \text{and }\ &\forall \e\in \R_+\exists\d\in\R_+ \forall ((x,y),(x',y'))\in (F\times F)\big(d'(y,y')<\delta \rightarrow d(x,x')\le \e\big).\label{def:PartEmb2}
  \end{align}
\end{Lemma}
\begin{proof}
  Suppose that $F \in \PartEmb(X,Y)$ and $f\colon C\to Y$ is an
  embedding such that $C \in K(X)$ and $F=\graph(f)$.  Then the
  conditions \eqref{def:PartEmb1} and \eqref{def:PartEmb2} are simply
  saying that both $f$ and its inverse are uniformly continuous, so
  they are satisfied by the compactness of~$C$ (from which in turn the
  compactness of $f[C]$ also follows).

  Suppose $F \in K(X\times Y)$ satisfies conditions
  \eqref{def:PartEmb1} and \eqref{def:PartEmb2}. Let $C=\prj_X(F)$,
  $D=\prj_{Y}(F)$. Then $C$ and $D$ are compact by the compactness
  of~$F$. By the choice of $C$ and $D$, for all $x\in C$ there is
  $y\in D$ and for all $y$ there is $x$ such that $(x,y)\in F$.
  Conditions \eqref{def:PartEmb1} and \eqref{def:PartEmb2} imply that
  for all $x\in C$ there is a unique $y\in D$ with $(x,y)\in F$ and
  vice versa, so $F$ is indeed a graph of some function
  $f\colon C\to D$ and that map is a bijection. But then conditions
  \eqref{def:PartEmb1} and \eqref{def:PartEmb2} imply that both $f$
  and its inverse are continuous, so we are done.
\end{proof}

\begin{Lemma}\label{lemma:PartEmbBorel}
  If $X$ and $Y$ are Polish, then $\PartEmb(X,Y)$ is a Borel subset of $K(X \times Y)$.
\end{Lemma}
\begin{proof}
  For any fixed $\e,\d\in\Q_+$ the sets
  $$Z^0_{\e\d}=\{((x,y),(x',y'))\in (X\times Y)^2\mid d(x,x')<\delta\rightarrow d(y,y')\le\e\}$$
  $$Z^1_{\e\d}=\{((x,y),(x',y'))\in (X\times Y)^2\mid d(y,y')<\delta\rightarrow d(x,x')\le\e\}$$
  are closed in $(X\times Y)^2$, so the sets $K(Z^0_{\e\d})$ and
  $K(Z^1_{\e\d})$ are closed in $K((X\times Y)^2)$.  Let
  $$\xi_0\colon K(X\times Y) \to K(X\times Y) \times K(X\times Y)\ \text{ be defined by }\ K \mapsto (K,K),$$ 
  and 
  \[\xi_1\colon K(X\times Y) \times K(X\times Y) \to K((X\times Y)^{2})\ \text{ by }\ (K,L) \mapsto K \times L.\] 
  By Fact~\ref{fact:K(X)}\ref{prop:K(X)-4} the composition
  $\xi= \xi_1 \circ \xi_0$ given by $K\mapsto K\times K$ is
  continuous. Hence the set
  $\xi^{-1}[K(Z^0_{\e\d})\cap K(Z^1_{\e\d})]$ is closed and so Borel
  for each $\e$ and~$\delta$.  The lemma will follow from the
  following equality:
  $$\PartEmb(X,Y)=\Cap_{\e\in\Q_+} \Cup_{\d\in\Q_+} \xi^{-1}[K(Z^0_{\e\d} \cap Z^1_{\e\d})].$$ 
  Denote the right-hand side of this equality by \(W\). Suppose
  \(F \in \PartEmb(X,Y)\).  Let $\e\in\Q_+$ be arbitrary. Then by
  Lemma \ref{lemma:PartEmbContainsAllEmbeddings}, there are
  \(\delta_0, \delta_1 \in \Q_+\) such that
  \(F \times F \subset Z^0_{\e\d_0}\cap Z^1_{\e\d_1}\). Let
  $\delta=\min\{\delta_0,\delta_1\}$. Then
  $F\times F\subset Z^0_{\e\d}\cap Z^1_{\e\d}$.  Thus,
  \(\xi(F)=F \times F \in K(Z^0_{\e\d}) \cap K(Z^1_{\e\d})\), and so
  $$F=\xi^{-1}[F\times F]\in \xi^{-1}[K(Z^0_{\e\d})\cap K(Z^1_{\e\d})].$$
  This shows that \(F \in W\). On the other hand, if \(F \in W\), then
  for each \(\e \in \Q_+\) there is \(\d \in \Q_+\) such that
  \(F \in \xi^{-1}[K(Z^0_{\e\d})\cap K(Z^1_{\e\d})]\), i.e.,
  \(F\times F=\xi(F)\in K(Z^0_{\e\d}\cap Z^1_{\e\d})=K(Z^0_{\e\d})\cap
  K(Z^1_{\e\d})\), which by the other direction of Lemma
  \ref{lemma:PartEmbContainsAllEmbeddings} implies that
  \(F \in \PartEmb(X,Y)\).
\end{proof}

Notice that by Lemma~\ref{lemma:PartEmbContainsAllEmbeddings}, if
\(F \in \PartEmb(X,Y)\), then \(F=\graph(f)\) for some
\(f\colon \prj_X(F) \to \prj_Y(F)\). So, the domain of such an \(f\)
is always $\prj_X(F)$. We will use this fact in the sequel.

\begin{Def}\label{def:EmbSpace}
  For a compact metric space $X$ and a metric space $Y$, let
  $\Emb(X,Y)=\{F\in \PartEmb(X,Y)\mid \prj_X(F)=X\}$ be the set of
  (the graphs of) embeddings of $X$ into~$Y$.
\end{Def}

\begin{Lemma}\label{lemma:EmbIsClosed}
  If $X$ and $Y$ are Polish and $X$ is compact, then $\Emb(X,Y)$ is a closed subset
  of $\PartEmb(X,Y)$, and hence a standard Borel space.
\end{Lemma}
\begin{proof}
  Suppose $F \in \PartEmb(X,Y)\setminus \Emb(X,Y)$. Let $(x,y)\in (X \times Y)\setminus F$ and
  let $\e=\frac{1}{2}d((x,y),F)$. Notice that $\e>0$ since $F$ is
  compact.  Let $U_\e$ be the $\e$-neighbourhood of $F$ in
  $K(X\times Y)$. Then $U_\e\cap \PartEmb(X,Y)$ is contained in
  $\PartEmb(X,Y)\setminus \Emb(X,Y)$ since $x$ is not in $\pr_X(F')$ for every $F' \in U_\e$.  Hence,
  $\PartEmb(X,Y)\setminus \Emb(X,Y)$ is an open subset of
  $\PartEmb(X,Y)$, and thus $\Emb(X,Y)$ is closed.  The second
  assertion follows by applying Lemma~\ref{lemma:PartEmbBorel}.
\end{proof}

\begin{Lemma}\label{lemma:SupmetricHomeoEmb}
  Let $K_1$ and $K_2$ be compact Polish metric spaces and let
  $C(K_1,K_2)$ be the set of continuous maps $K_1\to K_2$ equipped
  with the sup-metric.  Let $E\subset C(K_1,K_2)$ be the set of
  embeddings. Then $E$ is homeomorphic to $\Emb(K_1,K_2)$.
\end{Lemma}
\begin{proof}
  We can assume that the metric on $K_1\times K_2$ is the
  $\max$-metric given by
  $$d_{K_1\times K_2}((x,y),(x',y'))=\max\{d_{K_1}(x,x'),d_{K_2}(y,y')\}.$$
  Let $h\colon E\to \Emb(K_1,K_2)$ be
  defined by $h(f)=\graph(f)$.  It is clearly one-to-one and by
  Lemma~\ref{lemma:PartEmbContainsAllEmbeddings} it is onto. Let
  $f\in E$ and suppose $\e>0$.  We will show that
  $h[B_{E}(f,\e)]\subset B_{\Emb(K_1\times K_2)}(\graph(f),\e)$ where
  the ball on the left-hand-side is taken in the sup-metric on
  \(C(K_1,K_2)\) and the ball on the right-hand-side in the Hausdorff
  metric $d_H$ on $K(K_1\times K_2)$.  Suppose $g\in B_{E}(f,\e)$.  To
  show that $h(g)\in B_{\Emb(K_1\times K_2)}(\graph(f),\e)$ it is
  enough to show for all $x\in K_1$ that there exists $y$ such that
  $d_{K_1\times K_2}((x,g(x)),(y,f(y)))< \e$.  We can select $x=y$ in
  which case
  $d_{K_1\times K_2}((x,g(x)),(x,f(x)))=d_{K_2}(f(x),g(x))<\e$ by the
  choice of $g$. This proves that $h$ is open. Let us show now
  that $h$ is continuous. Let $f\in E$ and $\e>0$.  Since $f$
  is uniformly continuous by the compactness of $K_1$, there is
  $\delta<\e/2$ such that for all $x,y\in K_1$, if
  $d_{K_1}(x,y)<\delta$, then $d_{K_2}(f(x),f(y))<\e/2$. We will show
  that
  $h^{-1}[B_{\Emb(K_1\times K_2)}(h(f),\delta)]\subset B_{E}(f,\e)$.
  Let $g$ be such that $d_H(h(g),h(f))<\delta$.  This means that for
  every $x\in K_1$ there is $y\in K_1$ such that
  $$d_{K_1\times K_2}((x,g(x)),(y,f(y)))<\delta.$$
  By definition of $d_{K_1\times K_2}$ this is equivalent to
  \begin{equation}
    d_{K_1}(x,y)<\delta\text{ and }d_{K_2}(g(x),f(y))<\delta\label{eq:Inequalities_H}.
  \end{equation}
  From the first of the inequalities \eqref{eq:Inequalities_H}, by the
  choice of $\delta$, we have $d_{K_2}(f(x),f(y))<\e/2$. Combining
  with the second inequality of \eqref{eq:Inequalities_H} and the
  triangle inequality, we get
  $$d_{K_2}(f(x),g(x))<\e/2+\delta<\e.$$
  The last inequality follows from the assumption that
  $\delta<\e/2$. Since $x$ was arbitrary, this means that
  $g\in B_{E}(f,\e)$.
\end{proof}

From now on we treat $\PartEmb(X,Y)$
as a set of functions or a set of their graphs interchangeably hoping
that no confusion ensues.

\begin{Lemma}\label{lemma:r_C}
  Suppose $X$ and $Y$ are Polish metric spaces. Then the following hold:
  \begin{enumerate-(a)}
  \item\label{lemma:r_C-1} The projection maps
    \(\pr_X \colon K(X \times Y) \to K(X)\) and
    \(\pr_Y \colon K(X \times Y) \to K(Y)\) defined by
    \(\pr_X(K)=\{x \in X \mid \exists y \in Y ((x,y) \in K)\}\) and
    \(\pr_Y(K)=\{y \in Y \mid \exists x \in X ((x,y) \in K)\}\) are
    continuous.
  \item\label{lemma:r_C-3} If \(X\) is compact and $C\subset X$ is
    compact, the map $r_C\colon \Emb(X,Y)\to K(Y)$ defined by
    $r_C(F)=\{y\in Y\mid\exists x\in C((x,y)\in F)\}$ is Borel.
  \item\label{lemma:DomainOfEmb} The function
    $\dom\colon \PartEmb(X,Y)\to K(X)$ is Borel.
  \end{enumerate-(a)}
\end{Lemma}
\begin{proof}
  For \ref{lemma:r_C-1} it is enough to show this for $\pr_X$. We will
  show that the preimages of the sets generating the Vietoris topology
  on \(K(X)\) are open in \(K(X \times Y)\). Let \(U \subset X\) be
  open, and consider
\begin{align*}
    \pr_X^{-1}[\{K \in K(X)\mid K \cap U \neq \es\}]= &\ \{K \in K(X\times Y) \mid \pr_X(K) \cap U \neq \es\}\\
    = &\ \{K \in K(X\times Y) \mid (U \times Y) \cap K \neq \es\},
\end{align*}
which is a basic open set in \(K(X \times Y)\) because $U\times Y$ is
open in $X\times Y$. Similarly, one can prove that, given an open
subset \(U \subset X\), \(\pr_X^{-1}[\{K \in K(X)\mid K \subset U\}]\)
is open in \(K(X \times Y)\), and thus \(\pr_X\) is Borel.

Now let us prove \ref{lemma:r_C-3}.
  Let $D\in K(Y)$ and $\e>0$. We will show that $r_C^{-1}B_{K(Y)}(D,\e)$
  is Borel where $B_{K(Y)}(D,\e)$ is the ball of radius $\e$
  in the Hausdorff metric on $K(Y)$.
  Let \(f\colon C \to Y\) be such that \(F=\graph(f)\). Since $f$ is continuous, the image $f[C]=r_C(F)$ is compact.
  Now, 
  \begin{align*}
    F\in r^{-1}_C[B_{K(Y)}(D,\e)]&\iff r_C(F)\in B_{K(Y)}(D,\e)\\
    &\iff d_{K(Y)}(r_C(F),D)<\e\\
    &\iff r_C(F)\subset D_\e\land D\subset (r_C(F))_\e,
  \end{align*}
  where $D_\e$ and $(r_C(F))_\e$ are respectively the collar of $D$ and $r_C(F)$ (see Definition~\ref{def:K(X)}). The latter is a Borel condition by Fact~\ref{fact:K(X)}\ref{prop:K(X)-2} because
  $D_\e$ and $(r_C(F))_\e$ are compact by the compactness of $X$.
  Finally, we notice that the map \(\dom\) in \ref{lemma:DomainOfEmb} is just the restriction of \(\pr_X\) to \(\PartEmb(X,Y)\) given by
  $\dom(F)=\prj_X(F)$. Thus, by \ref{lemma:r_C-1} it is Borel.
\end{proof}

\begin{Lemma}\label{lemma:SubsetOfDomain}
  The set
  $B=\{(C,F)\in K(X)\times\PartEmb(X,Y)\mid C\subseteq \prj_X(F)\}$
  is closed.
\end{Lemma}
\begin{proof}
    If $(C,F)\notin B$, pick $x\in C\setminus \prj_X(F)$ and
    $\e>0$ such that $d(x,\prj_X(F))>2\e$, which exists by the compactness
    of $\prj_X(F)$. Then the $\e$-neighborhood of $(C,F)$ is an open
    neighborhood of $(C,F)$ outside~$B$. Thus, the complement of $B$ is open, equivalently $B$ is closed. 
\end{proof}
 
The following result appears to be new.

\begin{thm}[$\Emb_{PL}(|K_1|,|K_2|)$ is $K_\sigma$]%
  \label{thm:PLK_sigma}
  Let $K_1$ and $K_2$ be finite simplicial complexes in a Polish
  metric space $(X,d)$.  Let $Z=\Emb_{PL}(|K_1|,|K_2|)$ be the
  subspace of $\Emb(|K_1|,|K_2|)$ which consists of graphs of
  PL-embeddings from $K_1$ into $K_2$.  Then $Z$ is $K_\sigma$ in
  $\Emb(|K_1|,|K_2|)$.
\end{thm}
\begin{proof}
  Let $C^k$ be the set of triples
  $(\bar z_1,\bar z_2,\graph(f)) \in |K_1|^{\le k}\times |K_2|^{\le
    k}\times \Emb(|K_1|,|K_2|)$ such that $f$ is simplicial from
  $K_1*\bar z_1$ to a subcomplex of $K_2*\bar z_2$, the lengths of
  $\bar z_1$ and $\bar z_2$ are at most $k$ and $f$ is
  $k$-bilipschitz.
  \begin{claim}
    $C^k$ is compact.    
  \end{claim}
  \begin{proof}
    The sets $|K_1|^{\le k}$ and $|K_2|^{\le k}$ of sequences of
    length at most $k$ in the compact sets $|K_1|$ and $|K_2|$,
    respectively, are compact. Recall that $C(|K_1|,|K_2|)$ denotes
    the space of all continuous functions equipped with the
    sup-metric. The set $B^k$ of $k$-bilipschitz embeddings is closed
    in $C(|K_1|,|K_2|)$.  By the Arzelà-Ascoli Theorem (see
    e.g. \cite[Ch.~7]{Kelley1955}), $B^k$ is in fact compact. Let
    $B^k_0$ be the set of graphs of elements of $B^k$.  Then $B_0^k$
    is compact by Lemma~\ref{lemma:SupmetricHomeoEmb} in the topology
    of $\Emb(|K_1|,|K_2|)$.  Now,
    $C^k\subset |K_1|^{\le k}\times |K_2|^{\le k}\times B^k_0$ and so
    it is compact if it is closed. Let us show that $C^k$ is closed in
    $|K_1|^{\le k}\times |K_2|^{\le k}\times \Emb(|K_1|,|K_2|)$.
        
    Suppose $c=(\bar z_1,\bar z_2,\graph(f))\notin C^k$.  We will find
    an open neighborhood of $c$ outside~$C^k$ Denote $S:=K_1*\bar z_1$
    and $T:=K_2*\bar z_2$.
    \begin{remark}
      A map \(f\) is simplicial from \(S\) to \(T\) if and only if
      for all $\sigma \in S$, there is $\tau\in T$ such that
      $f|\sigma| = |\tau|$ and the map
      $\tau^{-1}\circ f\circ \sigma$ is linear (see Definitions
      \ref{def:EquivalenceOfSimplexes} and
      \ref{def:rel_complexes}).
    \end{remark}
                
    Suppose first that for some \(n\)-simplex $\sigma\in S$,
    $f|\sigma|\neq |\tau|$ for all $\tau \in T$.  Let
    \(\tau_0, \dots, \tau_k\) be the simplexes of \(T\) such that
    \(f|\sigma| \cap |\tau_i| \neq \es\). Denote by
    $A\triangle B=(A\setminus B)\cup (B\setminus A)$ the symmetric
    difference. Now $f|\sigma|\triangle |\tau_i|$ is non-empty.  Pick
    a point $w_i \in f|\sigma| \triangle |\tau_i|$ for each
    \(i \le k\). Let \(\e = \min\{d(w_i,A_i) \mid i \le k\}\), where
    \[ A_i=\begin{cases}
             |\tau_i| & \text{ if } w_i \in f|\sigma| \setminus |\tau_i|\\
             f|\sigma| & \text{ otherwise.}
           \end{cases}
    \]
    By the compactness of $f|\sigma|$ and \(|\tau_i|\) for each
    \(i \le k\), \(\e\) is greater than \(0\).  It is now possible to
    find $\e_0<\e$ such that for all $(\bar z'_1,\bar z'_2,\graph(f'))$
    in the $\e_0$-neighborhood $U$ of $c$, $f'$ maps some simplex
    \(\sigma' \in K_1*\bar z'_1\) not onto any simplex in
    $K_2*\bar z'_2$ witnessed by the same points $w_0,\dots,w_k$.
    Thus, $U$ is a neighborhood of $c$ outside $C^k$.
        
    Now assume that there are an \(n\)-simplex $\sigma\in S$ and
    \(\tau \in T\) such that \(f|\sigma|=\tau\) and
    $g = \tau^{-1}\circ f\circ \sigma$ is not linear. This is
    witnessed by two points $x, y \in \bDelta^n$ and $t \in (0,1)$,
    i.e. there is $\e>0$ such that
    $d(g(tx+(1-t)y),tg(x)+(1-t)g(y))>\e$. Since \(g\) is continuous,
    we can assume that $x$ and $y$ are in the interior of
    $\bDelta^n$. Let $\{v_0,\dots,v_n\}=V(\sigma)$ and let $\delta>0$
    be such that for all $v'_0,\dots,v'_n$ with $d(v'_i,v_i)<\delta$
    we have that $x$ and $y$ are in the image via \(\sigma^{-1}\) of
    the interior of $\Delta[v'_0,\dots,v'_n]$.  Let
    $\e'=\min\{\e,\delta\}$.  Then for all
    $(\bar z'_1,\bar z'_2,\graph(f'))$ in the $\e'$-neighborhood $U$
    of $c$, $\tau' \circ f' \circ \sigma'$ is not linear for some
    $\sigma'\in K_1*\bar z'_1$ and \(\tau' \in K_2*\bar z'_2\).  Thus,
    $U$ is a neighborhood of $c$ outside~$C^k$.

    We now consider the last case in which \(f\) is simplicial from
    \(S\) to \(T\), but not \(k\)-bilipschitz. Let \(x,y \in |S|\)
    witness this, i.e., either \(d(f(x),f(y))-kd(x,y)>0\) or
    \(\frac{1}{k}d(x,y)-d(f(x),f(y))>0\). Let \(\e\) be such a
    number. Then it is easy to see that any \(\e'\)-neighborhood of
    \(c\) with \(0<\e'<\frac{\e}{2}\) is an open neighborhood of \(c\)
    outside \(C^k\).
  \end{proof}
  Since $C^k$ is compact, its projection $P^k$ to the third coordinate
  is compact and so $P:=\Cup_{k\in\N}P^k$ is $K_\sigma$.  Let us show
  that $P=Z$. Clearly $P\subset Z$. Suppose $f$ is a PL-embedding from
  $K_1$ to $K_2$. Then there are triangulations $T_1$ and $T_2$ of
  $K_1$ and $K_2$ respectively such that $f$ is simplicial from $T_1$
  to a subcomplex of $T_2$. By \cite{AP24} there are
  $\bar z_1\subset K_1$ and $\bar z_2\subset K_2$ such that $f$ is
  simplicial from $K_1*\bar z_1$ to a subcomplex of $K_2*\bar z_2$. On
  each simplex of $K_1*\bar z_1$, $f$ is a linear embedding and
  therefore $L$-bilipschitz for some $L$. Since there are only
  finitely many simplexes in $K_1*\bar z_1$, there is an upper bound
  on all of these $L$ which gives a global bilipschitz constant for
  $f$.  Let $k$ be larger than this upper bound and larger than
  $\max\{l(z_1),l(z_2)\}$. Then $f\in P^k$.  Hence $f\in P$.
\end{proof}

\begin{Cor}\label{cor:FromCompactToNonCompactKsigmaPL}
  Let $K$ be a finite simplicial complex, $T$ an infinite simplicial
  complex, and let $Z=\Emb_{PL}(|K|,|T|)$ be the space of piecewise
  linear embeddings from $|K|$ into $|T|$.  Then $Z$ is $K_\sigma$ in
  $\Emb(|K|,|T|)$.
\end{Cor}
\begin{proof}
  Let $\PP$ be the set of finite subcomplexes of $T$. Note that $\PP$
  is countable.  Clearly $Z=\Cup_{P\in\PP} \Emb_{PL}(|K|,|P|)$.  By
  the previous theorem $\Emb_{PL}(|K|,|P|)$ is $K_\sigma$, so $Z$ is
  too.
\end{proof}

\vspace{10pt}

\subsubsection{Notation summary}

There will be many different spaces and equivalence relations defined
and used along the next few sections.  For the reader's convenience we
gather the most important ones here:

\vspace{10pt}

\noindent Spaces:

\renewcommand{\arraystretch}{1.3}
\noindent%
\begin{longtable}{p{1.5cm}@{ }p{13cm}}
  $\AAA$ & space of sorted complemented algebras (Definition~\ref{def:SpaceAAA}) \\
  $\BBB$ & space of basis spaces (Definition~\ref{def:BBB})\\
  $\BBB^C$ & space of CCLCP basis spaces (Definition~\ref{def:BBB}) \\
  $\U$ & the universal Urysohn space\\
  $\SSS(Z)$ &space of countable sequences of simplexes (of any dimension) in some space $Z$ (Definition~\ref{def:SpaceOfTTT})\\
  $\SSS$ & $=\SSS(\U)$ \\
  $\TTT(Z)$ & space of Heine-Borel simplicial complexes in $Z$ (Definition~\ref{def:SpaceOfTTT})\\
  $\TTT$ & $=\TTT(\U)$\\
  $\SSS_n$& Space of countable sequences of $n$-simplexes in $\U$ (Definition~\ref{def:SpaceOfManifolds})\\
  $\MMM_n$ & space of $n$-manifolds presented as atlases (Definition~\ref{def:SpaceOfManifolds})\\
  $\MMM_n^{PL}$ & space of PL $n$-manifolds presented as PL-compatible atlases (Definition~\ref{def:MMMPL})\\
  $\MMM_n^{PL,\le i}$ & space of $n$-manifolds whose first $i$ charts are PL-compatible. (Definition~\ref{def:MMMPL})\\
  $O(\R^n)$& Space of open subsets of $\R^n$ (Definition~\ref{def:SpaceOfOpenSets})\\
  $\TT^{i,j}_n$ & the set of pairs
  $(\bar\f,\bar\tau)\in \MMM^{PL,\le j}\times \TTT$ where $\tau$ is a
  compatible triangulation of the part of
  $M(\bar\f)$ which is covered by the first $i$ charts. (Definition~\ref{def:Tiin})\\
\end{longtable}

\vspace{10pt}

\noindent Equivalence relations

\noindent%
\begin{longtable}{p{1.5cm}@{ }p{13cm}}
  $\cong_\AAA$ & Isomorphism on $\AAA$ (Definition~\ref{def:SpaceAAA})\\
  $\equiv_\BBB$ & Equivalence of basis spaces on $\BBB$ (Definition~\ref{def:BBB})\\
  $\homeo_n$  & Homeomorphism relation on $\MMM_n$ (Definition~\ref{def:SpaceOfManifolds})\\
  $\homeo^{PL}$ & PL-homeomorphism relation on $\TTT$ (Definition~\ref{def:SpaceOfTTT})\\
  $\homeo^{PL}_n$ & PL-homeomorphism relation on $\MMM^{PL}_n$ (Definition~\ref{def:homeoPLnrel})\\
  $\homeo^O_n$ & Homeomorphism relation on $O(\R^n)$
  (Definition~\ref{def:SpaceOfOpenSets})
\end{longtable}
\renewcommand{\arraystretch}{1}

\vspace{10pt}

\noindent We have the following helpful facts:
$$\MMM^{PL}_n\subset \MMM^{PL,\le i}_n\subset \MMM_n\subset \SSS_n,
\qquad\TTT(Z)\subset \SSS(Z),\qquad
\TTT\subset\SSS,\qquad\TT^{i,j}_n\subset \MMM^{PL,\le j}_n\times
\TTT$$

\vspace{10pt}

\subsection{Borel classification of Heine-Borel simplicial complexes up to PL-homeomorphism}
\label{ssec:ManToBS}


In this section we define the standard Borel spaces $\AAA$
and $\TTT$ encoding sorted complemented algebras and
simplicial complexes, respectively. We then prove that the
PL-homeomorphism relation on simplicial complexes is Borel reducible
to the isomorphism on countable structures, specifically that on
countable sorted complemented algebras. We postpone the study of the exact complexity of the latter with respect to Borel reducibility to Section \ref{sec:Final}, where we prove that it is Borel bireducible with the isomorphism on countable graphs.

\begin{Def}[Space of sorted complemented algebras]\label{def:SpaceAAA}
  Let $X=2^{\N^2}\times \N\times\N\times 2^{\N^2}\times 2^\N$.  For
  any $x=(\eta,n,m,\xi,\zeta)\in X$, define an $L^+$-structure
  $\AA(x)=(A,\le,\bzero,\bone,c,K)$ such that $A=\N$, for all
  $a,b\in A$, $a\le b\iff \eta(a,b)=1$, $\bzero=n$, $\bone=m$, for all
  $a,b\in A$, $c(a)=b$ iff $\xi(a,b)=1$, and for all $a\in A$,
  $a\in K$ iff $\zeta(a)=1$.  Let $\AAA$ be the subset of $X$
  consisting of those $x$ for which $c$ as above is well-defined
  (i.e. for all $a$ there is exactly one $b$ such that $\xi(a,b)=1$)
  and such that $\AA(x)$ is a sorted complemented algebra.  Let
  $\cong_\AAA$ be the equivalence relation on $\AAA$ where two
  elements are equivalent if and only if the corresponding sorted
  complemented algebras are isomorphic.
\end{Def}

\begin{prop}
  $\AAA$ is a Borel subset of $X=2^{\N^2}\times \N\times\N\times 2^{\N^2}\times 2^\N$.
\end{prop}
\begin{proof}
  The axioms listed in Definition~\ref{def:algebra} are all
  first-order axioms, so the set is Borel by
  Fact~\ref{fact:borel_subsets_of_Mod(L)}.
\end{proof}

\begin{Def}[Space of basis spaces]\label{def:BBB}
  Let $\BBB$ be the set of all pairs $(X,\bar O)=(X,(O_i)_{i\in\N})$
  where $X\in F(\U)$ and $O_i\subseteq X$ is open in~$X$.  We will
  endow $\BBB$ with a Borel structure.  Let
  $\BBB^*= F(\U)\times F(\U)^\N$.  Let $\xi\colon \BBB^*\to \BBB$ be
  defined by $\xi(F,(F_i)_{i\in\N})=(F,(F\setminus F_i)_{i\in\N})$.
  This is a bijection and pushes the Borel structure to $\BBB$
  from~$\BBB^*$.  Then each $(X,(O_i)_{i\in\N})\in \BBB$ is a basis
  space.  Denote by $\BBB^C\subset \BBB$ those pairs $(X,\bar O)$
  where $\bar O$ is a basis for the induced topology on $X$ and
  $(X,\bar O)$ is a CCLCP basis space, that is, satisfies the
  conditions of Definition~\ref{def:CLCP}.  Let $\equiv_\BBB$ be the
  equivalence relation on $\BBB$ where two elements are equivalent if
  and only if the corresponding basis spaces are equivalent (see
  Definition~\ref{def:LCBS}).
\end{Def}

Our arguments work irrespectively of whether or not \(\BBB\) and
\(\BBB^C\) are standard Borel spaces, so we leave this question
untouched.

\begin{Def}[Space of simplicial complexes]%
  \label{def:SpaceOfTTT}
  For technical reasons, define $\dummy$ to be a ``dummy simplex''.
  It is just a place-holder, and we define by convention that
  $|\dummy|=\es$.  Given a Polish space $Z$, let
  $$\SSS(Z)=\left(\Cup_{m\in \N}\Emb(\bDelta^m,Z)\cup\{\dummy\}\right)^{\N}$$
  be the set of all infinite sequences of 
  simplexes. 
  Given $\bar\tau\in \SSS(Z)$, denote by 
  $\bar \tau_*$
  the corresponding set $\{\tau_i\mid i\in\N\}\setminus\{\dummy\}$
  of non-dummy simplexes.  
  Let $\TTT(Z)\subset \SSS(Z)$ be
  the space of all \textbf{Heine-Borel simplicial complexes} in \(Z\), i.e.,  
  $\TTT(Z)$
  consists of those sequences $\bar\tau=(\tau_i)_{i\in\N}$ such that
  $$T(\bar\tau):=\{[\tau]\mid \tau\in \bar\tau_*\}$$
  satisfies Definition~\ref{def:simplicial_complex} where $[\tau]$ is
  the equivalence class of $\tau$ according to
  Definition~\ref{def:EquivalenceOfSimplexes}, and for all $n$ there
  is $m$ such that for all $k>m$ we have $B_Z(z_0,n)\cap |\tau_k|=\es$
  for some fixed point $z_0\in Z$.  The latter is the Heine-Borel
  property.  For $\bar\tau\in \TTT(Z)$, conventionally define
  $|\bar\tau|:=|T(\bar\tau)|.$ By having only finitely many non-dummy
  simplexes in the sequence, we represent finite complexes, so they
  are also thought of as members of~$\TTT(Z)$. We denote
  $\TTT:=\TTT(\U)$ for $Z=\U$ the Urysohn space.
  Let $\homeo^{PL}$ be the equivalence relation of
  PL-homeomorphism on $\TTT$ (defined as in Definition~\ref{def:rel_complexes}).
\end{Def}

As pointed out in Lemma \ref{lemma:EmbIsClosed}, if $Z$ is a Polish
metric space, $\Emb(\bDelta^m,Z)$ is a standard Borel space.  The
Borel structure on $\Cup_{m\in \N}\Emb(\bDelta^m,Z)\cup\{\dummy\}$ is
that of a disjoint union of the sets $\Emb(\bDelta^m,Z)$ and an
isolated point $\{\dummy\}$.  Then $\SSS(Z)$ is a standard Borel space
as a product of standard Borel spaces.

In the next lemma we show that $\TTT(Z)$ is a standard Borel
space. Notice that the Heine-Borel property of the elements of
$\TTT(Z)$ implies condition \ref{def:simplicial_complex-3} of
Definition \ref{def:simplicial_complex}. The reason for using this
stronger property will become clear in the proof of Theorem
\ref{thm:PLtoBasis}.

\begin{Lemma}
  For a Polish space $Z$, $\TTT(Z)$ is a Borel subset
  of~$\SSS(Z)$.
\end{Lemma}
\begin{proof}
  The Heine-Borel property is clearly Borel.  Let $U_{ij}$ be the set
  of those $\bar\tau\in \SSS(Z)$ for which $|\tau_i|$ and $|\tau_j|$
  are disjoint. Then $U_{ij}$ is an open set by Fact \ref{fact:F(X)}. Let $F_{ij}$ be the set
  of those $\bar\tau$ such that $\tau_i$ is a face of $\tau_j$, let
  $\FF(\bDelta^n)$ be the set of all faces of $\bDelta^n$ and let
  $I(A,B)$ be the set of isometries from $A$ onto $B$.  Then
  $$F_{ij}=\big\{\bar\tau\mid |\tau_i|\subset |\tau_j|\big\}\cap  \Big\{\bar\tau\mid \tau_j^{-1}\circ\tau_i\in I\big(\bDelta^{\dim(\tau_i)},D\big)\text{ for some }D\in \FF\big(\bDelta^{\dim(\tau_j)}\big)\Big\}.$$
  By Fact \ref{fact:F(X)} and \cite[Exercise 9.4]{Kec95} $F_{ij}$ is a
  Borel set.  Let $G_{ij}$ be the set of those $\bar\tau$ for which
  there is $k\in \N$ such that $|\tau_i|\cap|\tau_j|=|\tau_k|$ and
  $\tau_k$ is a face of both $\tau_i$ and $\tau_j$. Then
  $$G_{ij}=\Cup_{k\in\N}\big\{\bar\tau\mid |\tau_i|\cap |\tau_j|=|\tau_k|\big\}\cap F_{ki}\cap F_{kj}.$$
  By Fact \ref{fact:F(X)} $G_{ij}$ is a Borel set. Let $H$ be the set
  of those $\bar \tau$ such that every face of every element of
  $\bar\tau$ is in $\bar\tau$.  Then
  $$H=\Cap_{i\in\N}\Cap_{\lambda\in \FF(\tau_i)}\Cup_{j\in \N} F_{ji}\cap \{\bar\tau\mid \tau_j=\lambda\}$$
  where $\FF(\tau)$ is the set of faces of~$\tau$.  Finally, let
  $$L=\{\bar\tau\mid \bar \tau \text{ has the Heine-Borel property}\}.$$
  Now
  $$\TTT(Z)=H\cap L\cap \Cap_{i\in\N}\Cap_{j\in\N}(U_{ij}\cup G_{ij}).$$
  and so it is Borel.
\end{proof}

Now, recall that $\homeo^{PL}$ denote
the PL-homeomorphism relation on $\TTT$ and $\equiv_\BBB$ is 
the equivalence of basis spaces (Definition~\ref{def:LCBS})
on the space~$\BBB$, and denote by $\equiv_{\BBB^C}$ the restriction
of $\equiv_\BBB$ to $\BBB^C$.
Also, by $\cong_\AAA$ we denote the isomorphism 
relation on~$\AAA$.

\begin{thm}%
  \label{thm:PLtoBasis}
  There is a Borel function $\eta\colon\TTT\to \BBB$ 
  reducing
  $\homeo^{PL}$ to $\ \equiv_{\BBB}$.
  Moreover, $\ran(\eta)\subset \BBB^C$, 
  and for all $\bar\tau\in\TTT$, if $(X,\beta)=\eta(\bar\tau)$, then
  $(X,\la\beta\ra)$ is homeomorphic to~$|\bar\tau|$.
\end{thm}
\begin{proof}
  Recall the construction of $(|T|,\beta_T)$ from a complex $T$ (see
  Definition~\ref{def:BetaFromT}), where \(\beta_T\) is defined as the
  collection of the interiors of all compact and co-compact
  \(\Q\)-polyhedra in \(T\). First, we show that the \(T\)-dense set
  $\Q(T)$ can be enumerated in a Borel way for $T=T(\bar\tau)$ and
  $\bar\tau\in \TTT$.  Let $(q_i)$ be a fixed enumeration of $\Q^*$.
  Given $\bar\tau\in T$, define $p_{ij}(\bar\tau)=\tau_i(q_j)$. We set
  \(p_{ij}(\bar \tau)\) equal to the empty sequence if
  \(q_j \notin \dom(\tau_i)\). Let $\pi\colon\N\to\N\times\N$ be a
  bijection and let $p_i(\bar\tau):=p_{\pi(i)}(\bar\tau)$.  Then the
  function $\bar\tau\mapsto (p_{\pi(i)}(\bar\tau))_{i\in\N}$ from
  $\TTT$ to $\U^{\N}$ is Borel and enumerates $\Q(T)$.
  
  Given $\bar z\in \N^{<\N}$, let
  $r(\bar z,\bar\tau)=(p_{z}(\bar\tau))_{z\in\bar z}$.  For every
  $\bar z\in \N^{<\N}$, let $\bar\tau*r(\bar z,\bar\tau)$ be the
  corresponding $\lo$-stellar subdivision enumerated canonically
  according to the order of appearance of new simplexes. In this way
  $$\subd\colon (\bar\tau,\bar z)\mapsto \bar\tau*r(\bar z,\bar\tau)$$
  is a Borel function from $\TTT\times\N^{<\N}$ to~$\TTT$.  Note that
  by Fact~\ref{fact:CanChooseFiniteStellarSubd} every finite
  polyhedron in $T$ is represented by a subcomplex of
  $T*r(\bar z,\bar\tau)$ for some finite $\bar z$.
  
  For every
  $(\bar z,\bar s,\bar\tau)\in \N^{<\N}\times \N^{<\N}\times\TTT$,
  define
  $$b(\bar z,\bar s,\bar \tau)=\int_{|\bar \tau|}|\scl(\{\kappa_{s(i)}\mid i\in \{0,\dots,l(s)-1\}\})|$$
  where $\bar\kappa=\ \subd\!(\bar\tau,\bar z)$, and $l(s)$ is the
  length of $\bar s$.  Using this function $b$ we are planning to
  enumerate all elements of $\beta_T$.  Since the basic sets in the
  basis spaces in $\BBB$ are ultimately represented through their
  complements (see Definition \ref{def:BBB}), it is enough to show
  that the following function is Borel:
  \begin{equation}\label{eq:enumeration_of_basis}
    \N^{<\N}\times \N^{<\N}\times\TTT \to F(\U)\qquad\qquad(\bar z,\bar s,\bar \tau)\mapsto |\bar\tau|\setminus b(\bar z,\bar s,\bar \tau).
  \end{equation}
  We first check that \(|\bar \tau| \in F(\U)\), equivalently,
  \(\U \setminus |\bar \tau|\) is open.  Fix \(z_0 \in |\bar \tau|\),
  and let \(x \in \U \setminus |\bar \tau|\) be arbitrary.  Let
  \(n \in \N\) be sufficiently large so that \(x\in B(z_0,n)\).  By
  the Heine-Borel property of \(\bar \tau\) we have that
  \(K=|\bar \tau| \cap \overline{B(z_0,n)}\) is compact, and thus
  \(d(x,K)= \e >0\).  Then \(B(x, \frac{\e}{2})\) is an open
  neighborhood of \(x\) in \(\U \setminus |\bar \tau|\), and thus
  \(\U \setminus |\bar \tau|\) is open.
  
  To show that \eqref{eq:enumeration_of_basis} is Borel, it is enough to know that the following functions are Borel:
  \begin{align}
    &\TTT \to F(\U) &&\bar\tau\mapsto |\bar\tau|\label{item:RealizationIsBorel}\\
    &\Emb(\bDelta^n,\U)\to K(\U) &&\kappa\mapsto |\kappa|,\label{item:RangeIsBorel}\\
    &K(\U)^\N\times \N^{<\N}\to K(\U)&& (\bar a,\bar s)\mapsto a_{s(0)}\cup\cdots\cup a_{s(l(s)-1)}\\
    &K(\U)\times F(\U)\to F(\U)&&(K,F)\mapsto \overline{F\setminus K}.\label{item:RangeIsBorel2}\
  \end{align}  
  By Lemma \ref{lemma:r_C}\ref{lemma:r_C-3}, Fact
  \ref{fact:K(X)}\ref{prop:UnionContinuous} and Proposition
  \ref{prop:StrongComplementBorel}, the functions
  \eqref{item:RangeIsBorel}--\eqref{item:RangeIsBorel2} are Borel. The
  only thing we should still check is that
  $\bar\tau\mapsto |\bar\tau|$ is Borel.  To see this, together with
  \eqref{item:RangeIsBorel} it is enough to show that the infinite
  union
  \begin{align*}
      &X\to F(\U)&& \bar a \mapsto \Cup_{i\in\N}a_i
  \end{align*}
  is a Borel function where $X$ is the subset of $K(\U)^\N$ which
  consisting of those $(K_i)_{i\in\N}$ for which the Heine-Borel
  property in the definition of~$\TTT$ holds (see
  Definition~\ref{def:SpaceOfTTT}), i.e. for all $n$ there is $i$ such
  that for all $j>i$ we have $K_j\cap B(z_0,n)=\es$, with $z_0$ a
  fixed point in $\U$.  This easily follows by applying an argument
  similar to Lemma \ref{lemma:r_C}\ref{lemma:r_C-1}, by showing that
  the preimage of the basic open sets in \(F(\U)\) are Borel subsets
  of \(X\).  Then let $\pi\colon \N\to \N^{<\N}\times \N^{<\N}$ be a
  bijection with coordinates $\pi_1$ and $\pi_2$, and let
  $\beta'(\bar\tau):=(b(\pi_1(i),\pi_2(i),\bar\tau))_{i\in\N}$.  This
  will enumerate only the compact polyhedra. Let $\beta(\bar\tau)$ be
  a sequence such that the odd indices enumerate $\beta'(\bar\tau)$
  and the even indices enumerate the complements of elements of
  $\beta'(\bar\tau)$ in the same order.  Then define the final Borel
  function \(\eta\) by $\bar\tau\mapsto (|\bar\tau|,\beta(\bar\tau))$,
  which by Proposition \ref{prop:CCLCP} and Theorem
  \ref{thm:ClassificationOFPL-complexesByBasisSpaces} is a reduction
  of \(\homeo^{PL}\) to \(\equiv_\BBB\) whose range is contained
  in~\(\BBB^C\).
\end{proof}

\begin{prop}\label{prop:RedBasisToAlg}
  There is a Borel $\xi\colon \BBB\to \AAA$ such that for all
  $(X,\beta),(X',\beta')\in \BBB^C$, we have
  $(X,\beta)\equiv_\BBB (X',\beta')\iff \xi(X,\beta)\cong_\AAA
  \xi(X',\beta')$, that is, $\xi\rest \BBB^C$ is a reduction from
  $\equiv_{\BBB^C}$ to~$\cong_\AAA$.
\end{prop}
\begin{proof}
  Recall the operation $\psi$ from
  Definition~\ref{def:map_from_BBB_to_AAA}.  It gives a function
  $\xi\colon \BBB\to \Mod_{L^+}(\N)$,
  $\xi\colon (X,\beta)\mapsto \psi(X,\beta)$.  From the definition of
  $\psi$ it is easy to check using Fact \ref{fact:F(X)} and
  Proposition \ref{prop:StrongComplementBorel} that this function is,
  in fact, Borel.  By Lemma~\ref{lemma:CCLCP_is_AAA}, we have
  $\psi[\BBB^C]\subset \AAA$.  The rest follows from
  Theorem~\ref{thm:PsiReduction}.
\end{proof}

\begin{thm}[$\homeo^{PL}\,\le_B\ \cong_\AAA$]%
  \label{thm:BorelClassificationOfPLtoISO}
  There is a Borel reduction $\zeta\colon \TTT\to \AAA$ reducing
  $\homeo^{PL}$ to $\cong_\AAA$.
\end{thm}
\begin{proof}
  Let $\eta$ be the reduction $\homeo^{PL}$ to $\equiv_{\BBB^C}$ given
  by Theorem~\ref{thm:PLtoBasis} and $\xi$ the reduction from
  $\equiv_{\BBB^C}$ to $\cong_\AAA$ given by
  Proposition~\ref{prop:RedBasisToAlg}.  Then from their properties it
  follows that $\zeta=\xi\circ\eta$ is the desired reduction. Note
  that we did not need to prove that $\BBB^C$ is a Borel set. We also
  only showed that $\xi$ has the property of a reduction restricted to
  $\BBB^C$ even though $\dom(\xi)=\BBB$.  Since
  $\ran(\eta)\subset \BBB^C\subset \dom(\xi)$, and both functions
  $\eta$ and $\xi$ are Borel, the composition is also Borel.
\end{proof}

\vspace{10pt}

\subsection{Borel classification of manifolds}\label{sec:manifolds}

We have shown that PL-homeomorphism on Heine-Borel simplicial
complexes can be Borel reduced to the isomorphism on sorted
complemented algebras
(Theorem~\ref{thm:BorelClassificationOfPLtoISO}).  If we show that a
triangulation of $2$- and $3$-manifolds can be obtained in a Borel
way, then we have a Borel classification of these manifolds
strengthening the result of
Theorem~\ref{thm:ClassificationOfManifolds}.  We will do it for
$3$-manifolds with or without boundary and for open subsets of $\R^2$
and \(\R^3\), while we refer to \cite[Theorem 4.16]{BS25} for a proof
of the existence of a triangulation for 2-manifolds without boundary
via a Borel map.  We leave
the case of $2$-manifolds with boundary for future work. 
To begin with, we need to define a Borel space of $3$-manifolds (with
or without boundary). This part (defining the space of manifolds) we
will do for general $n$.  Since every $n$-manifold is characterized by
an atlas, we will define a Borel space of atlases whose realizations
are $n$-manifolds. Another natural option would be, for example, to
let $M\subset \U$ be the set of closed subsets of $\U$ which are
manifolds. Since $\U$ contains a copy of every metric space, $M$ would
indeed contain a copy of every manifold, at least up to homeomorphism
(some metric manifolds such as $\R^2\setminus\{0\}$ might not be
represented isometrically as a closed subset of~$\U$).  But as far as
the authors are aware, there is no reason to assume that such $M$
would be a Borel set. This is why we resort to atlases as
parametrizations of manifolds.

\vspace{10pt} 

\subsubsection{Spaces of manifolds without boundary}
\label{ssec:SpaceOfManifolds}

In the beginning of Section \ref{sssec:Classification23man} we defined
$n$-manifolds without boundary as separable metric spaces locally
homeomorphic to~$\obDelta^n$. In this section we resort
to a different definition.
From this point on, an \textbf{$n$-manifold without boundary} is
a metric space $M$ for which there exists a countable sequence of maps
$\bar\f=(\f_i)_{i\in\N}$ called an \textbf{atlas} such that each
$\f_i$ is an embedding $\f_i\colon\bDelta^n\to M$ with the
following conditions:
\begin{enumerate}[label={\upshape (M\arabic*)}, leftmargin=2.5pc]
\item \label{item:m1} The space $M$ is covered by
  the interiors:
  $M=\Cup_{i\in \N}|\overset{\circ}\f_i|$.
\item \label{item:m2} For all $i\in \N$, the image
  $|\overset{\circ}\f_i|$ is open in~$M$.
\item \label{item:m3} For all compact $C\subset M$ there are only 
  finitely many $i$ with $|\f_i|\cap C\ne\es$.
\end{enumerate}
We skip the somewhat technical, but standard proof of the fact
that a metric space $M$ is an $n$-manifold without boundary
according to the definition in Section \ref{sssec:Classification23man}
if and only if it satisfies \ref{item:m1}--\ref{item:m3}. From now on, we refer to \(n\)-manifolds without boundary simply as \(n\)-manifolds, while we will specify if they have a boundary. 

\begin{Def}\label{def:SpaceOfManifolds}
  Let 
  $\SSS_n:=\Emb(\bDelta^n,\U)^{\N}$ be the set of
  infinite sequences \(\bar\f=(\f_i)_{i \in \N}\) of 
  $n$-simplexes in $\U$. Given such a sequence, let
  $M(\bar\f)=|\bar\f|=\Cup_{i\in\N}|\f_i|$.  Let $\MMM_n\subseteq \SSS_n$ consist
  of those $\bar\f$ for which $(M,\bar\f)=(M(\bar\f),\bar \f)$
  satisfies conditions \ref{item:m1}--\ref{item:m3}
  above, and additionally such that $M(\bar\f)$
  is Heine-Borel. We require the Heine-Borel property
  because otherwise $M(\bar\f)$ is not necessarily closed in~$\U$. Let $\homeo_n$ be the equivalence relation
  of homeomorphism on~$\MMM_n$ defined 
  for all $\bar\f,\bar\f'\in\MMM_n$ by
  $\bar\f\homeo_n \bar\f'\iff M(\bar\f)\homeo M(\bar\f')$.
\end{Def}

Let us now show that $\MMM_n$ indeed contains a copy of every
$n$-manifold in the following sense:

\begin{Lemma}\label{lemma:FindEveryManifold}
  Let $M$ be an $n$-manifold. Then there is $\bar\f\in\MMM_n$ such
  that $M$ is homeomorphic to $M(\bar\f)$. Conversely, for every
  $\bar\f\in\MMM_n$, $(M(\bar\f),\bar\f)$ is an $n$-manifold.
\end{Lemma}
\begin{proof}
  Let $M$ be an $n$-manifold. As discussed in the beginning of this
  section, there is an atlas $\bar\psi$ on $M$ satisfying conditions
  \ref{item:m1}--\ref{item:m3}.  By Fact \ref{fact:locally_compact_HB} there is a
  homeomorphism $h\colon M\to M'$ onto some Heine-Borel manifold~$M'$.
  Let $g\colon M'\to\U$ be an isometric embedding. Let
  $$\f_i=g\circ h\circ\psi_i.$$
  Since all the Conditions \ref{item:m1}--\ref{item:m3} are satisfied
  for $(M,\bar\psi)$, they transfer to $(M(\bar\f),\bar\f)$ in virtue
  of $h$ and $g$ being homeomorphisms. Since $M'$ is Heine-Borel and
  $g$ is an isometry, also $M(\bar\f)$ is Heine-Borel.  Thus,
  $\bar\f\in\MMM_n$.  The reverse direction trivially follows from the
  definition of~$\MMM_n$.
\end{proof}

\begin{prop}[$\MMM_n$ is Borel]\label{prop:ManifoldsBorel}
  $\MMM_n$ is a standard Borel space.
\end{prop}
\begin{proof}
  We will show that it is a Borel subset of the standard Borel space
  $\SSS_n=\Emb(\bDelta^n,\U)^{\N}$. We will give an equivalent definition of
  $\MMM_n$ and then show the set thus defined is Borel. Let $u_0\in\U$
  be a fixed point.  Let $\MMM'_n\subseteq \SSS_n$ be the set of those
  $\bar\f$ for which the following conditions hold:
  \begin{enumerate}[label={\upshape (M\arabic*')}, leftmargin=2.5pc]
  \item \label{item:m1p}
   For every $i\in\N$ there is
   $m\in\N$ such that $|\partial \f_i|\subseteq \Cup_{j=0}^m|\overset{\circ}\f_{j}|$
   (recall Definition~\ref{def:Faces}\eqref{eq:oversetcirc}).
  \item  \label{item:m2p} For all $i\in\N$,
    the image $|\overset{\circ}\f_i|$ is open in~$M(\bar\f)$.
  \item \label{item:m3p} For all $r\in\N$, there are only finitely many
    $i\in\N$ such that $\overline{B_\U(u_0,r)}\cap|\f_i|\ne\es$.
  \end{enumerate}
  First, let us show that $\MMM'_n=\MMM_n$.  Suppose
  $\bar\f\in\MMM'_n$.  We want to show that $(M,\bar\f)$ where
  $M=M(\bar\f)$ satisfies conditions \ref{item:m1}--\ref{item:m3} and
  $M$ is Heine-Borel.  Clearly we have
  $$M(\bar\f)=\Cup_{i\in\N}|\overset{\circ}\f_i|\cup |\partial \f_i|.$$
  Condition \ref{item:m1} then follows from Condition \ref{item:m1p}.
  Condition \ref{item:m2p} directly implies \ref{item:m2}.  Suppose
  $C\subset M(\bar\f)$ is compact. Then it is bounded in $\U$ and so
  there is $r$ such that $C\subset B_\U(u_0,r)$.  Now
  Condition~\ref{item:m3} follows from Condition~\ref{item:m3p}.
  Finally, let us prove that $M(\bar\f)$ is Heine-Borel. Let
  $C\subset M(\bar\f)$ be a closed and bounded subset. Then it is
  bounded also in $\U$, so there is $r$ such that
  $C\subset B_\U(u_0,r)$.  By Condition \ref{item:m3p}, $C$ can
  intersect only finitely many $|\f_i|$ and since
  $C\subset M(\bar\f)$, there is some $m$ such that
  $C\subset \Cup_{i=0}^m|\f_i|$ and is therefore compact as a
  closed subset of a compact set. Thus, we have proved that
  $\MMM_n'\subseteq \MMM_n$.
  
  For the reverse inclusion, suppose $\bar\f\in\MMM_n$. Condition
  \ref{item:m1} together with \ref{item:m3} and the compactness of
  $\partial\bDelta^n$ implies \ref{item:m1p}. Condition \ref{item:m2},
  again, directly implies \ref{item:m2p}.  Let $r\in\N$. Then, since
  $M(\bar\f)$ is Heine-Borel, $M(\bar\f)\cap \overline{B_\U(u_0,r)}$ is
  compact in $M(\bar\f)$. Now Condition \ref{item:m3} implies \ref{item:m3p}.  This
  completes the proof that $\MMM'_n=\MMM_n$.

  We now show that $\MMM'_n$, and hence $\MMM_n$, is Borel.  Re-write:
  \begin{align*}
    &|\partial \f_i|\subseteq
    \Cup_{j=0}^m|\overset{\circ}\f_{j}|\\
    \iff&
          |\partial \f_i|\cap
          \Cap_{j=0}^m M\setminus|\overset{\circ}\f_{j}|=\es \\
    \iff&\partial\bDelta^n\cap \Cap_{j=0}^m \underbrace{\bDelta^n\setminus \f_i^{-1}(|\overset{\circ}\f_j|)}_{\text{compact}}=\es.
  \end{align*}
  By Fact \ref{fact:F(X)} this condition is Borel for each $m$ and taking the union
  over all $m\in\N$ the condition \ref{item:m1p} becomes Borel.  By Fact
  \ref{fact:K(X)} and the compactness of $|\f_i|$, Condition
  \ref{item:m3p} is also seen to be Borel.

  Finally, it remains to show that condition \ref{item:m2p} is Borel.
  Let $\bDelta^n_\e\subset \obDelta^n$ be a smaller simplex inside
  $\obDelta^n$ which is obtained by shrinking $\bDelta^n$ a little
  bit:
  $$\bDelta^n_\e:=\{x\in \obDelta^n\mid d(x,\partial\bDelta^n)\ge \e\}$$
  and define the interior of it by
  $$\obDelta^n_\e:=\{x\in \obDelta^n\mid d(x,\partial\bDelta^n)> \e\}.$$
  We claim that \ref{item:m2p} is equivalent to the following
  condition~$(*)$:
  \begin{itemize}
  \item[$(*)$] for all $\e\in\Q_+$ there is 
    $\delta\in\Q_+$ such that
    for all $x\in |\f_i\rest (\obDelta^n_\e \cap \Q^*)|$, all $j\in \N\setminus \{i\}$, and all
    $y\in |\overset{\circ}\f_j\rest \Q^*|\setminus |\f_i|$ we have
    $d(x,y)>\delta$.
  \end{itemize}
  This condition is clearly Borel, because all quantifiers range over
  countable sets.  Let us prove that \ref{item:m2p}~$\Rightarrow (*)$.
  Suppose \ref{item:m2p} holds. Let $\e\in\Q_+$.  Since
  $|\f_i\rest \bDelta^n_\e|$ is compact in $|\overset{\circ}\f_i|$ which in turn
  is open in $M(\bar\f)$, we have that
  $M(\bar\f)\setminus |\overset{\circ}\f_i|$ is closed and
  \begin{equation}
    \label{eq:deltai}
    \delta:=d\Big(M(\bar\f)\setminus |\overset{\circ}\f_i|,|\f_i\rest \bDelta^n_\e|\Big)>0.
  \end{equation}
  Let
  $x\in |\f_i\rest(\obDelta^n_\e\cap\Q^*)|$ and
  $y\in |\f_j\rest (\obDelta^n\cap \Q^*)|\setminus |\f_i|$ 
  for some $j\in \N\setminus \{i\}$.
  Then by \eqref{eq:deltai}, clearly $d(x,y)>\delta$.

  Now let us prove $(*)\Rightarrow$~\ref{item:m2p}. 
  Assume $(*)$ and suppose $x'\in |\overset{\circ}\f_i|$. 
  Let $\e\in\Q_+$ be such
  that 
  $$\e<d((\f_i^{-1}(x'),\partial\bDelta^n).$$
  Let $\delta\in \Q_+$ be as
  given by $(*)$. Let $j\in \N\setminus \{i\}$ 
  be arbitrary. We claim that
  \begin{equation}\label{eq:capsubset1}
    B(x',\delta/3)\cap |\overset{\circ}\f_j|\subseteq |\f_i|.
  \end{equation}
  This is sufficient, because by the arbitrariness of $j$, it implies that
  $B(x',\delta/3)\cap |\overset{\circ}\f_i|=B(x',\delta/3)\cap M(\bar\f)$ 
  is an open neighbourhood of $x'$
  in~$|\overset{\circ}\f_i|$. We prove the following statement, which is equivalent
  to~\eqref{eq:capsubset1}: 
  \begin{equation}\label{eq:capsubset2}
    (|\overset{\circ}\f_j|\setminus |\f_i|)\cap B(x',\delta/3)=\es.
  \end{equation}
  If $|\overset{\circ}\f_j|\setminus |\f_i|$
  is empty, we are done. Otherwise pick an arbitrary $y'\in |\overset{\circ}\f_j|\setminus |\f_i|$.
  We will show that $d(x',y')>\delta/3$.
  By the choice of $\e$, we have 
  $\f_i^{-1}(x')\in\obDelta^n_\e$, so there
  is $q_x\in \obDelta^n_\e\cap \Q^*$
  with $d(\f_i(q_x),x')<\delta/3$.
  Let $x=\f_i(q_x)$.
  Since $|\f_i|$ is compact, the 
  set
  $$\f_j^{-1}\big[|\overset{\circ}\f_j|\setminus |\f_i|\big]=\obDelta^n\setminus \f_j^{-1}\big[|\f_i|\big]$$ 
  is open, so
  there is $q_y\in (\obDelta^n\cap \Q^*)\setminus \f_j^{-1}[|\f_i|]$
  such that $d(\f_j(q_y),y')<\delta/3$. Let $y=\f_j(q_y)$, so we have
  $y\in |\overset{\circ}\f_j\rest\Q^*|\setminus |\f_i|$.  
  Now the conditions of $(*)$ are satisfied for $\e,\delta,x,j$ and $y$,
  so we have $d(x,y)>\delta$. 
  Thus, $d(x',y')\ge d(x,y)-d(x,x')-d(y,y')> \delta -\delta/3-\delta/3=\delta/3$.
\end{proof}

Let us now define the space of locally PL-compatible atlases.  Let
$U\subset\R^m$. A simplicial complex $T_0$ in $U$ is a
\textbf{polyhedron} if it is finite and for some rectilinear
triangulation \(T\) of $\R^m$, $T_0$ is a polyhedron in $T$ according
to Definition~\ref{def:QPolyhedron}. Since any two rectilinear
triangulations of $\R^m$ have a common rectilinear subdivision,
``some'' can be replaced by ``any'' in this definition.  A function
$f\colon U\to \R^m$, $U\subset \obDelta^n$, is \textbf{locally PL}, if
for all polyhedra $T_0$ in $U$, $f\rest |T_0|$ is~PL.

A \textbf{$\Q$-polyhedron} in $\bDelta^n$ is a subcomplex of a
$\Q$-subdivision of~$\scl(\bDelta^n)$.  

\begin{lem}\label{lem:ExtensionTo_bDelta}
  If $T_0\subset \obDelta^n$ is a polyhedron, then there exists a
  subdivision $T$ of $\bDelta^n$ such that some subdivision of $T_0$
  is a subcomplex of~$T$.
\end{lem}
\begin{proof}
  By \cite[Lemma 6]{Bi59} there is a rectilinear triangulation
  $P$ of $\R^n$ such that $T_0$ is a subcomplex of $P$.
  \cite[Lemma 6]{Bi59}
  is stated for $n=3$ but the
  same proof works for
  general $n$ as remarked in~\cite{AP24}.
  By the same argument, there is a rectilinear triangulation
  $P'$ of $\R^n$ such that $\bDelta^n$ is a simplex in $P'$.
  Let $T'$ be a common subdivision of $P$ and $P'$ and let
  $T=T'\rest \bDelta^n$.
\end{proof}

\begin{lem}\label{fact:locallyPL-polyhedron}
  Suppose $U\subset \obDelta^n$ is open. Then $f\colon U\to \R^m$
  is locally PL if and only if for all $\Q$-polyhedra $T_0$
  in $U$, $f\rest |T_0|$ is PL.
\end{lem}
\begin{proof}
  The direction from left to right is obvious.  Let
  $f\colon U\to \R^m$, $U\subset\obDelta^n$ open, and suppose that for
  all $\Q$-polyhedra $Q_0$ in $U\subset\bDelta^n$, $f\rest |Q_0|$
  is~PL.  Let $T_0$ be a polyhedron in~$U$. Let $Q$ be a fine enough
  $\Q$-subdivision of $\bDelta^n$ such that
  $Q_0:=\St_Q(|T_0|)\subset U$. This is possible because $U$ is open
  and $T_0$ is finite (so $|T_0|$ is compact).  Then $f$ is simplicial
  on a subdivision $P_0$ of $Q_0$.  Note that $P_0$ is not necessarily
  a $\Q$-polyhedron.  Let $T$ be a subdivision of $\bDelta^n$ such
  that some subdivision $T_0'$ of $T_0$ is a subcomplex of $T$ (exists
  by Lemma~\ref{lem:ExtensionTo_bDelta}).  Let $S$ be a common
  subdivision of $Q$ and $T$. Then there are subcomplexes $S_0$ and $S_1$ of $S$  
  such that $S_0\subset S_1$, $|S_0|=|T_0|$, and $|S_1|=|Q_0|$. Let
  $S_1'$ be a subdivision of $P_0$ such that $S_0'=S_1\rest |T_0|$
  is a subdivision of $T_0$. Then $f$ is simplicial on $S_1'$ and therefore
  on~$S_0'$.
\end{proof}

Similarly to the case of $n$-manifolds, we now use locally PL-compatible atlases to parametrize PL $n$-manifolds.

\begin{Def}\label{def:MMMPL}
  Two functions $\f,\psi\in \Emb(\bDelta^n,\U)$ are
  \textbf{locally PL-compatible}, if the functions
  $$\big(\overset{\circ}{\psi}\big)^{-1}\circ \overset{\circ}\f\colon \big(\overset{\circ}\f\big)^{-1}|\overset{\circ}{\psi}|
  \to \big(\overset{\circ}\psi\big)^{-1}|\overset{\circ}\f| \quad\text{ and }\quad
  \big(\overset{\circ}\f\big)^{-1}\circ \overset{\circ}\psi\colon
  \big(\overset{\circ}\psi\big)^{-1}|\overset{\circ}\f|\to
  \big(\overset{\circ}\f\big)^{-1}|\overset{\circ}\psi|$$
  are locally PL. Let $\MMM_n^{PL}\subseteq \MMM_n$ be the set of those
  $\bar\f$ such that for all $i,j$, $\f_i$ and $\f_j$ are locally PL-compatible.
  Let $\MMM_n^{\PL,\le i}$ be the set of those $\bar\f\in \MMM_n$ for which
  $\f_0,\dots,\f_i$ are pairwise locally PL-compatible.
\end{Def}

\begin{Lemma}\label{lemma:AlgPLManifoldsBorel}
  The spaces $\MMM^{PL,\le i}_n$ and $\MMM_n^{PL}$ are Borel subsets of $\MMM_n$.
\end{Lemma}
\begin{proof}
  Clearly $\MMM_n^{PL}=\Cap_{i\in\N}\MMM^{PL,\le i}_n$, so it is
  enough to show that $\MMM^{PL,\le i}_{n}$ is Borel for all~$i\in\N$.
  Let $\PPP^{\le i}\subset \SSS_n$ be the set of sequences $\bar\f$
  such that $\f_0,\dots,\f_i$ are pairwise locally PL-compatible. Then
  $\MMM_n^{PL,\le i}=\MMM_n\cap \PPP^{\le i}$. By Proposition
  \ref{prop:ManifoldsBorel} it suffices to show that $\PPP^{\le i}$ is
  a Borel subset of~$\SSS_n$.

  For a fixed $\Q$-polyhedron $T$ in \(\bDelta^n\), let $A_T$ be the
  set of pairs
  $(\f_1,\f_2)\in \Emb(\bDelta^n,\U)\times\Emb(\bDelta^n,\U)$ such
  that $|T|\subset
  (\overset{\circ}{\f}_1)^{-1}|\overset{\circ}\f_2|$. We show that
  \(A_T\) is Borel.  Let
  $$F_T\colon A_T\to \Emb(|T|,\bDelta^n)$$
  be given by $F_T(\f_1,\f_2)=(\f_2^{-1}\circ \f_1)\rest |T|$.  By
  Lemma \ref{lemma:r_C} the maps $\f_2\mapsto |\f_2|$,
  $\f_2\mapsto |\partial \f_2|$ and $(\f_1,C)\mapsto \f_1^{-1}[C]$,
  where \(C \subset \U\) is compact, are Borel, so also
  $(\f_1,\f_2)\mapsto \f_1^{-1}|\partial \f_2|$ is Borel.  By
  compactness the set
  $$\{(\f_1,\f_2)\mid \f_1^{-1}|\partial \f_2|\cap |T|=\es\}$$
  is Borel as well as
  $$\{(\f_1,\f_2)\mid |T|\subset \f_1^{-1}|\f_2|\}.$$
  The intersection of those is~$A_T$.
  
  Let us show that the function $F_T$ Borel. It is easy to see that if
  $\Emb(\bDelta^n,\U)$ and $\Emb(|T|,\bDelta^n)$ are equipped with the
  sup-metric, then $F_T$ is continuous. But the sup-metric generates
  the same Borel sets of those given by the topology on both (see
  Lemma \ref{lemma:SupmetricHomeoEmb}), so $F_T$ is Borel.

  By Lemma~\ref{fact:locallyPL-polyhedron}, $\bar\f\in \PPP^{\le i}$
  if and only if for all $j<k\le i$ there is a $\Q$-complex $T_0$ in
  $U=(\overset{\circ}{\f_j})^{-1}|\overset{\circ}{\f_k}|$ such that
  $F_{T_0}(\f_j,\f_k)$ is PL. All quantifiers in this statement are
  countable, so it remains to show that for any given $\Q$-polyhedron
  $T$, the set
  $$\Emb_{PL}(|T|,\bDelta^n):=\{f\in \Emb(|T|,\bDelta^n)\mid f\text{ is }PL\}$$
  is a Borel subset of $\Emb(|T|,\bDelta^n)$. But this follows
  from Theorem~\ref{thm:PLK_sigma}.
\end{proof}

\begin{Def}\label{def:Compatible}
  Suppose $\bar\tau\in \TTT$ and $\bar\f\in\MMM_n$ are such that
  $|\bar\tau|\subset M(\bar\f)$. We say that $\bar\tau$ is
  \textbf{PL-compatible with} $\f_i$, if for all $j$ with
  $|\tau_j|\subset |\overset{\circ}\f_i|$, the function
  $\f_i^{-1}\circ\tau_j$ is~PL. If $\bar\f\in\MMM_n^{PL}$ and
  $\bar\tau$ is a triangulation of $M(\bar\f)$ which is PL-compatible
  with each $\f_i\in\bar\f$ and such that for all $l$ there is $k$
  such that $|\tau_l|\subset |\overset{\circ}{\f}_k|$, we say that
  $\bar\tau$ is a \textbf{compatible triangulation} of $M(\bar\f)$.
\end{Def}

\begin{Def}\label{def:Tiin}
  For $i\le j$, let $\TT^{i,j}_n$ be the set of pairs
  $(\bar \f,\bar\tau)\in \MMM^{PL,\le j}_n\times \TTT$ such that
  $\bar\tau$ is a triangulation of
  $\Cup_{k\le i}|\overset{\circ}\f_k|$ such that for all $l\in\N$
  there is $k\le i$ such that
  $|\tau_l|\subset |\overset{\circ}{\f}_k|$ and $\f_k^{-1}\circ\tau_l$
  is~PL.
\end{Def}

By the \textbf{standard triangulation of} $\obDelta^n$ we mean a fixed
infinite rectilinear triangulation of the open set~$\obDelta^n$.

\begin{Lemma}\label{lemma:TTijIsBorel}
  For all $i\le j$, the set $\TT^{i,j}_n$ is Borel.
\end{Lemma}
\begin{proof}
  Let $S_1$ be the set of those pairs
  $(\bar\f,\bar\tau)\in \MMM^{PL,\le j}_n\times \TTT$ 
  such that for all $l\in\N$ there is $k\le i$ such
  that $|\tau_l|\subset |\overset{\circ}{\f}_k|$ and $\f_k^{-1}\circ\tau_l$ is PL. 
  Let $S_2$ be the set of those
  pairs $(\bar\f,\bar\tau)\in \MMM^{PL,\le j}_n\times \TTT$ 
  such that for all $k\le i$, $|\overset{\circ}{\f}_k|\subset |\bar \tau|$.
  Clearly $\TT^{i,j}_n=S_1\cap S_2$. We will show that $S_1$
  and $S_2$ are Borel. Let $\bar\delta=(\delta_i)_{i \in \N}$ be the standard
  triangulation of $\obDelta^n$.  Then
  $(\bar\f,\bar\tau)\in S_2$ if and only if for all $k\le i$
  for all $k_1\in\N$ there is $k_2\in\N$ such that
  $\f_k[|\delta_{k_1}|]\subset \Cup_{l\le k_2}|\tau_l|$.
  This is Borel by Facts \ref{fact:F(X)} and \ref{fact:K(X)}\ref{prop:UnionContinuous}.
  For $S_1$ by Fact \ref{fact:F(X)}\ref{prop:F(X)-4} it is enough to show
  that the set of PL embeddings $\bDelta^n\to \obDelta^n$
  is Borel. This follows directly from
  Corollary~\ref{cor:FromCompactToNonCompactKsigmaPL}.
\end{proof}

\vspace{10pt}

\subsubsection{PL \(n\)-manifolds without boundary up to PL-homeomorphism}

In the following we show that triangulations can be extended in a
Borel way.

\begin{prop}\label{prop:TriangulationStep}
  Fix $i\in\N$.  There is a Borel map
  $$f_i\colon \TT^{i,i+1}_n\to \TT^{i+1,i+1}_n$$
  such that if $(\bar\f',\bar\tau')=f_i(\bar\f,\bar\tau)$, then
  $\bar\f'=\bar\f$ and
  \begin{equation}
    \text{if $j<i$ is such that $|\tau_j|\cap |\f_{i+1}|=\es$, then $\tau_j'=\tau_j$.}\label{eq:tauistaup} 
  \end{equation}
\end{prop}
\begin{Remark}
  The technical condition~\eqref{eq:tauistaup} is needed later in the
  proof of Theorem~\ref{thm:GetTrianulation} to prove that a specially
  defined limit of the functions $f_i$ is Borel.
\end{Remark}
\begin{proof}
  Let $U_1=|\overset{\circ}{\f}_{i+1}|$ and
  $U_2=\Cup_{k\le i}|\overset{\circ}{\f}_k|$.  Let $\bar\delta$ be the
  standard triangulation of~$\obDelta^n$, and let
  $T_1:=\bar\delta'=(\f_{i+1}\circ\delta_j)_{j\in\N}$.  Note that
  \(T_1\) is a triangulation of \(U_1\) and the map
  $\bar\f\mapsto\bar\delta'$ is Borel since it is given by the
  countable sequence of the Borel functions
  $\bar\f\mapsto \f_{i+1}\circ \delta_j$.  We now check that the
  triangulation $T_2:=\bar\tau$ is compatible with
  \(\f_{i+1}\circ \delta_j\) for all \(j\). Suppose that for some
  \(l\) and \(j\) we have
  \(|\tau_l|\subset |\interior(\f_{i+1}\circ \d_j)|\). We want to show
  that \((\f_{i+1}\circ \d_j)^{-1} \circ \tau_l\) is PL. By the
  assumption that \((\bar \f,\bar \tau) \in \TT_n^{i,i+1}\), it
  follows that there is \(k\) such that
  \(|\tau_l| \subset |\overset{\circ}{\f_k}|\) and
  \(\f_k^{-1} \circ \tau_l\) is PL, and
  \((\overset{\circ}{\f_{i+1}})^{-1}\circ \overset{\circ}{\f_k}\) is
  PL. So, the map
  \[
    (\f_{i+1}\circ \d_j)^{-1} \circ \tau_l = ((\overset{\circ}{\f_{i+1}}\restriction |\d_j|)^{-1}\circ  \overset{\circ}{\f_k}) \circ ((\overset{\circ}{\f_k})^{-1}\circ \tau_l)
  \]
  is PL as well, as desired.
  
  We will now refer to Bing's proof of \cite[Theorem 7]{Bi59}. First,
  notice that from the fact that $T_2$ is compatible with
  \(\f_{i+1}\circ \delta_j\) for all \(j\), we have that if a simplex
  \(\sigma \in \bar \d'\) is such that \(|\sigma| \cap U_2 \neq \es\)
  and \(|\sigma| \cap (U_1 \setminus U_2)\neq \es\), then the faces of
  \(\sigma\) in \(U_1 \cap U_2\) are polyhedra in \(T_2\). We can thus
  assume that the map $h$ in Bing's proof, which is a homeomorphism
  from \(U_1\) to itself such that each face of a \(3\)-simplex of a
  subdivision of \(T_1\) goes onto a polyhedron of \(T_2\), is the
  identity map. Let $W_1\subset T_1$ be the set of those simplexes
  whose realization is disjoint from $U_2$, and let $W_2$ be the set
  of those simplexes whose realization intersects $U_1\setminus U_2$,
  but are not in $W_1$.  Let $W'_2=hW_2=W_2$. Let $V_2$ be the set of
  simplexes in $T_2$ which intersect simplexes in $W'_2=W_2$ and $V_1$
  the rest of $T_2$'s simplexes. On pages 59--61 of \cite{Bi59} Bing
  builds a triangulation $T_3$ of $U_1\cup U_2$. The construction is
  very detailed and explicit. Even though he is dealing with the case
  $n=3$, this part of the proof works for general $n$. Only the
  existence of $h$ in Bing's proof requires the assumption $n=3$, but
  we have assumed for the present proof that $h$ is the identity.
  From it, it is evident that the existence of a simplex in $T_3$
  depends only on the configuration of a finite number of simplexes in
  $T_1\cup T_2$. Therefore the map $(T_1,T_2)\mapsto T_3$ can be made
  continuous under a proper enumeration of $T_3$ which we will now
  do. In Bing's construction, $V_1\cup W_1\subset T_3$, and other
  simplexes in $T_3$ are obtained by subdividing the simplexes in
  $W_2\cup V_2$ one by one. To achieve a suitable enumeration, proceed
  as follows. Let $\bar\sigma^0$ be the sequence of simplexes of
  length $i$ such that for all $k<i$ we have $\sigma^0_k=\dummy$, if
  $\tau_k\notin V_1$ and otherwise $\sigma^0_k=\tau_k$.  For all $j$,
  suppose that a finite sequence of simplexes $\bar\sigma^j$ has been
  defined.  If $\delta'_{j}\notin W_1$, let $s=(\dummy)$ be the
  sequence of length~1 containing the dummy simplex. Otherwise, let
  $s:=(\delta'_{j})$.  If $\tau_{i+j}\notin V_1$, let
  $s'=(\dummy)$. Otherwise let $s':=(\tau_{i+j})$. Let
  $\bar \sigma^{j+1}:=\bar \sigma^j\cat s\cat s'$.  Finally, let
  $\bar\sigma:=\Cup_{j\in\N}\bar \sigma^j$. Note that since
  $\{\tau_j\mid j\in\N\land |\tau_j|\cap |\f_{i+1}|=\es\}\subset V_1$,
  we have for all $j<i$ that if $|\tau_j|\cap|\f_{i+1}|=\es$, then
  $\sigma_j=\tau_j$. Also note that $\bar\sigma$ enumerates exactly
  all simplexes in $W_1\cup V_1$ with possibly the dummy simplex.

  Next we will deal with simplexes in $W_2$ and~$V_2$.  For
  bookkeeping purposes, let $K_0$ be the set of those $k$ for which
  $\delta_k'\in W_2$ and let $L_0$ be the set of those $l$ for which
  $\tau_l$ is in~$V_2$. Let $\bar\pi^0$ be an empty sequence of
  simplexes. Suppose that index sets $K_j$, $L_j$, and a finite
  sequence of simplexes $\bar\pi^j$ have been defined. If $K_j$ is
  empty, let $\bar\kappa$ be the sequence of length one containing the
  dummy simplex, and $K_{j+1}:=K_j$.  Otherwise, let $k_0$ be the
  smallest element of $K_j$ and let
  $\bar\kappa=(\kappa_0,\dots,\kappa_{n(k_0)})$ be the subdivision of
  $\delta'_{k_0}$ given by Bing in his proof. As we have said before,
  it only depends on a finite initial segment of
  $(\bar\delta',\bar\tau)$. Let $K_{j+1}=K_{j}\setminus
  \{k_0\}$. Similarly, if $L_j$ is empty, let $\bar\lambda=(\dummy)$,
  and $L_{j+1}:=L_j$.  Otherwise let $l_0$ be the smallest element of
  $L_{j}$ and let $\bar\lambda=(\lambda_0,\dots,\lambda_{m(l_{0})})$
  be the subdivision of $\tau_{l_0}$ given by Bing. Let
  $L_{j+1}=L_{j}\setminus \{l_0\}$. Finally let
  $\bar\pi^{j+1}=\bar\pi^{j}\cat \bar\kappa\cat \bar\lambda$.  After
  this process let $\bar\pi:=\Cup_{j\in\N} \bar\pi^j$. Let
  $p_i\colon \N\to \N\times \{0,1\}$ be a fixed bijection with the
  property that for all $j<i$ we have $p_i(j)=(j,1)$.  For $z=(x,y)$
  denote $z_1=x$ and $z_2=y$.  Now for all $j$ let
  $$\tau'_j:=
  \begin{cases}
    \sigma_{p_i(j)_1},\text{ if }p_i(j)_2=1\\
    \pi_{p_i(j)_1},\text{ if }p_i(j)_2=0.
  \end{cases}
  $$
  By the construction, we see that the map
  $$(\bar\delta',\bar\tau)\mapsto \bar\tau'$$
  is continuous. Since the function
  $(\bar\f,\bar\tau)\mapsto(\bar\delta',\bar\tau)$ is Borel, also the
  desired function $(\bar\f,\bar\tau)\mapsto (\bar\f,\bar\tau')$ is
  Borel. By Bing's proof, $\bar\tau'$ is a triangulation of
  $U_1\cup U_2$, and, by the properties of $\bar\sigma$ and $p$,
  also condition \eqref{eq:tauistaup} holds.
\end{proof}

Recall that $\MMM^{PL}_n$ is the space of locally PL-compatible
atlases and $\MMM^{PL,\le i}_n$ is the set of atlases where the first
$i$ charts are locally PL-compatible (see Definition~\ref{def:MMMPL}).

\begin{Lemma}\label{lemma:inductiveTriangulation}
  For all $j\in\N$ there is a Borel
  $\mu_j\colon \MMM^{PL,\le j}_n\to \TTT$ such that
  $\bar\tau^j=(\tau^j_k)_{k \in \N}=\mu_j(\bar\f)$ is a triangulation of
  $\Cup_{k\le j}|\overset{\circ}{\f}_k|$.  Moreover, if there exists
  \(j_1\) such that $k<j_1$ and $|\f_k|\cap |\f_j|=\es$ for all
  $j\ge j_1$, then $\tau^{j_1}_k=\tau^{j}_k$ for all $j>j_1$.
\end{Lemma}
\begin{proof}
  Let $f_k$ be the maps given by Proposition
  \ref{prop:TriangulationStep}, $k\le j$.  As before, let $\bar\delta$
  be the standard triangulation of $\obDelta^n$. Let
  $\bar\tau^0=(\f_0\circ \delta_j)_{j\in\N}$.  Then
  $(\bar\f,\bar \tau^0)\in \TT^{0,j}_n\subset \TT^{0,1}_n$. Suppose
  that $\bar \tau^i$ has been defined for $i<j$ such that
  $(\bar\f,\bar \tau^i)\in\TT^{i,j}_n$. Then define $\bar \tau^{i+1}$
  be such that $(\bar\f,\bar \tau^{i+1})=f_i(\bar\f,\bar \tau^i)$.
  Note that $\TT^{i,j}_n\subset\TT^{i,i+1}_n$ and because the first
  coordinate stays unchanged in the application of $f_i$, we also have
  that
  $f_i(\bar\f,\bar \tau^i)\in \TT^{i+1,j}_n \subset \TT^{i+1,i+1}_n$.
  Now we have
  $\bar \tau^j=(\pr_2\circ f_j\circ\cdots\circ f_0)(\bar\f,\bar
  \tau^0)$ where $\pr_2$ is the projection to the second
  coordinate. All the functions in the composition are Borel, so
  $(\bar\f,\bar\tau^0)\mapsto \bar\tau^j$ is Borel.  But also
  $\bar\f\mapsto\bar\tau^0$ is Borel.  Let $\mu_j$ be the function
  $\bar\f\mapsto\bar\tau^j$.  The ``moreover''-part follows from
  condition~\eqref{eq:tauistaup}.
\end{proof}

\begin{thm}[Borel triangulation of PL-manifolds without boundary]%
  \label{thm:GetTrianulation}
  There is a Borel $\xi\colon \MMM^{PL}_n\to \TTT$ such that for all
  $\bar\f\in \MMM^{\PL}_n$, $\xi(\bar\f)$ is a compatible
  triangulation of~$M(\bar\f)$.
\end{thm}
\begin{Remark}
    We defined the term ``compatible''
    at the end of Definition~\ref{def:Compatible}.
\end{Remark}
\begin{proof}
  Let $\bar\f\in \MMM^{PL}_n$.  For all $i$, let $\bar\tau^i=(\tau^i_k)_{k \in \N}$ be as
  given by Lemma~\ref{lemma:inductiveTriangulation}.  Then
  $\bar\tau^i$ is obtained in a Borel way from $\bar\f$. This is
  possible, because $\MMM^{PL}_n\subset\MMM^{PL,\le i}_n$ for all
  $i\in\N$.  Let $\bar\tau$ be defined so that for all $k\in\N$,
  $\tau_k$ equals to $\tau^i_k$ for all large enough $i$.  We will
  show that $\xi\colon \bar\f\mapsto \bar\tau$ is well-defined and
  Borel, and $\bar\tau$ is a triangulation of $M(\bar\f)$.  Since
  $\bar\f$ satisfies condition \ref{item:m3}, for all $k$ there is
  $j_1$ such that for all $j>j_1$ we have $|\f_k|\cap |\f_j|=\es$.
  Therefore $\tau^j_k=\tau^{j_1}_k$ for all $j>j_1$. So $\tau_k$ is
  well-defined. Also, for large enough $j$ we have
  $(\tau_k)_{k\le i}=(\tau^j_k)_{k\le i}$, so it is easy to check that
  the conditions of triangulation are satisfied for $\bar\tau$ as they
  get locally inherited from $\bar\tau^j$.  Denote
  $Z=\Cup_{m\in\N}\Emb(\bDelta^m,\U)$.  Let $N\subset Z^\N$ be a basic
  open set of the form $N=N_1\times\cdots\times N_i\times Z^\N$.
  Recall that $\TTT\subset Z^\N$, so it is enough to show that
  $\xi^{-1}N$ is Borel. Clearly $\bar\f\in\xi^{-1}N$ if and only if
  $(\xi(\bar \f))_k=\tau_k\in N_k$ for $k\le i$ if and only if for all
  $k\le i$ there is $j_1$ such that for all $j>j_1$,
  $\tau^j_k\in N_k$, where \(\tau^j_k\) is defined as before.  Thus,
  $$\xi^{-1}N=\Cap_{k=1}^i\Cup_{j_1\in\N}\Cap_{j\ge j_1}\zeta_{jk}^{-1}N_k$$
  where $\zeta_{jk}$ is the function $\bar\f\mapsto \tau^j_k$ which
  is Borel by the previous lemma. Thus, $\xi^{-1}N$ is also Borel.
\end{proof}

\begin{Def}\label{def:homeoPLnrel}
  We have previously defined $\homeo^{PL}$
  to be the PL-homeomorphism relation on $\TTT$.
  Let $\homeo^{PL}_n$ be the following relation
  on $\MMM^{PL}_n$. For $\bar\f,\bar\f'\in\MMM^{PL}_n$, 
  the relation
  $\bar\f\homeo^{PL}_n\bar\f'$ holds, if 
  there is a homeomorphism $h\colon M(\bar\f)\to M(\bar\f')$
  such that for all $\tau\colon\bDelta^n\to M(\bar\f)$,
  if $i$ and $j$ are such that $|\tau|\subset |\overset{\circ}{\f}_i|$,
  $|h\circ\tau|\subset |\overset{\circ}{\f}{}'_j|$
  and $\f_i^{-1}\circ\tau$ is PL, then
  $(\f'_j)^{-1}\circ h\circ\tau$ is also PL.
\end{Def}

\begin{Prop} 
  \label{prop:PLiffPL}
  For all $\bar\f,\bar\f'\in\MMM^{PL}_n$ the relation
  $\bar\f\homeo^{PL}_n\bar\f'$ holds if and only if
  $\xi(\bar\f)\homeo^{PL}\xi(\bar\f')$ where $\xi$ is from
  Theorem~\ref{thm:GetTrianulation}.
\end{Prop}
\begin{proof} Suppose first that $\bar\f\homeo^{PL}_n\bar\f'$, and denote by $h$ the homeomorphism witnessing this. Let $\xi(\bar \f)=(\tau_l)_{l \in \N}$ and $\xi(\bar \f')=(\tau'_j)_{j \in \N}$. By the fact that $\xi(\bar \f)$ is a compatible triangulation of $M(\bar \f)$, we have that for all $l$ there is $k$ such that $|\tau_l| \subset |\overset{\circ}{\f}_k|$. Now, define a subdivision $(\sigma_i)_{i \in \N}$ of $\xi(\bar \f)$ fine enough so that for each $i$ there is $j$ such that $|h\sigma_i|\subset |\tau'_j|$. Using that $\xi(\bar \f')$ is a compatible triangulation of $M(\bar \f')$, it then follows that for each $i$ there are $l$, $k$, $j$, and $k'$ such that $|\sigma_i| \subset |\tau_l| \subset |\overset{\circ}{\f}_k|$ and $|h\sigma_i|\subset |\tau'_j| \subset |\overset{\circ}{\f}{}'_{k'}|$. Since $\f_k^{-1} \circ \sigma_i$ is PL, we have that $(\f'_{k'})^{-1} \circ h \circ \sigma_i$ is PL as well. Then $h$ witnesses that $\xi(\bar\f)\homeo^{PL}\xi(\bar\f')$. 

Conversely, let $h:M(\bar \f) \to M(\bar \f')$ witness $\xi(\bar\f)\homeo^{PL}\xi(\bar\f')$, and let $\tau\colon\bDelta^n\to M(\bar\f)$,
 $i$ and $j$ be such that $|\tau|\subset |\overset{\circ}{\f}_i|$,
  $|h\circ\tau|\subset |\overset{\circ}{\f}{}'_j|$
  and $\f_i^{-1}\circ\tau$ is PL. We want to show that
  $(\f'_j)^{-1}\circ h\circ\tau$ is PL. Notice that $|\tau|$ is a compact polyhedron in $\xi(\bar\f)$, and thus $|h \circ \tau|$ is a compact polyhedron in $\xi(\bar\f')$. Using the fact $\xi(\bar\f')$ is a compatible triangulation of $M(\f')$,  it is then easy to see that $(\f'_j)^{-1}\circ h\circ\tau$ is PL.
\end{proof}

\begin{Cor}[$\homeo^{PL}_n\,\le_B\ \cong_\AAA$]
\label{cor:HomeoPL_on_Man_red_cong}
  There is a Borel map $\eta\colon \MMM^{PL}_n\to\AAA$ which reduces
  $\homeo^{PL}_n$ to~$\cong_\AAA$.
\end{Cor}
\begin{proof}
    Let $\zeta\colon \TTT\to \AAA$ be the 
    reduction given by Theorem~\ref{thm:BorelClassificationOfPLtoISO}
    and $\xi\colon \MMM^{PL}_n\to \TTT$
    the function given by 
    Theorem~\ref{thm:GetTrianulation}.
    Let $\eta=\zeta\circ\xi$. It is a composition of
    Borel function and hence Borel. By Proposition~\ref{prop:PLiffPL} 
    it reduces $\homeo^{PL}_n$ to $\cong_{\AAA}$.
\end{proof}

\vspace{10pt}

\subsubsection{$3$-manifolds without boundary up to homeomorphism}
\label{ssec:BorelTriangulation}

In this section we show that $3$-manifolds without boundary are classifiable by countable structures. 
Suppose $A$ and $B$ are metric spaces, $f,g\colon A\to B$ are some
functions, and $\e\in\Q_+$ is a number.  We say that $f$ is an
\textbf{$\e$-approximation of $g$}, if $d(f(x),g(x))<\e$ for all
$x\in A$.  Let $S$ and $T$ be simplicial complexes. Recall that a map
\(f\colon S \to T\) is a \textbf{PL-embedding} if $f$ is a
PL-homeomorphism from $S$ onto a polyhedron in $T$.  It is a
$\Q$-PL-embedding, if $f$ is a $\Q$-PL-homeomorphism from $S$ to some
$\Q$-polyhedron in~$T$.

\begin{Lemma}[$\Q$-PL-approximation]\label{lemma:QPLApprox}
  Let $S$ and $T$ be finite complexes in metric spaces $(X,d)$ and
  $(Y,d')$ respectively, $S_1$ a subcomplex of $S$ and
  $f\colon S\to T$ a PL-embedding such that $f\rest |S_1|$ is a
  $\Q$-PL-embedding of $S_1$ into~$T$.  Suppose that $\e>0$. Then
  there is $g\colon S\to T$ such that $g$ is an $\e$-approximation of
  $f$, $g$ is a $\Q$-PL-embedding from $S$ to $T$, and
  $g\rest S_1=f\rest S_1$.
\end{Lemma}
\begin{proof}
  Let $S,T,f,S_1$, and $\e$ be as in the assumptions.  Let $S'\subd S$
  and $T'\subd T$ be subdivisions such that $f$ is simplicial from
  $S'$ to a subcomplex of $T'$.  Let $S'_1=S'\rest S_1$.  By Theorem
  \ref{thm:QPL-homeo} we can assume that $S'_1$ is a $\Q$-subdivision
  of $S_1$ and $fS'_1$ a $\Q$-polyhedron in~$T$.  By the methods of
  Section~\ref{ssec:PointAdjustment} there are $\bar s\subset S$ and
  $\bar t\subset T$ such that $S*\bar s$ is a subdivision of $S'$,
  $S_1*\bar s$ is a $\Q$-subdivision of $S_1$ (i.e.
  $\bar s\cap S_1\subset \Q(S_1)$), $T*\bar t$ is a subdivision of
  $T'$ and $\bar t\cap |fS_1|\subset \Q(T)$. Moreover $f$ is
  simplicial from $S*\bar s$ to a subcomplex of $T*\bar t$. Let $L$ be
  a Lipschitz-constant for $f$ which exists by the finiteness of $S$.
  Apply Lemma~\ref{lemma:Adjust} to find $\bar q\subset S$ and
  $\bar q'\subset T$ satisfying the following:
    \begin{enumerate-(1)}
    \item $l(\bar q)=l(\bar s)$, $l(\bar q')=l(\bar t)$,
    \item $d(q_i,s_i)<\e/(2L)$, $d(q'_i,t_i)<\e/2$,
    \item there are simplicial $\chi\colon S*\bar s\to S*\bar q$
    and $\chi'\colon T*\bar t\to T*\bar q'$,
    \item for all $i$ such that $s_i\in |S_1|$, $q_i=s_i$
    and for all $i$ such that $t_i\in f|S_1|$, $q'_i=t_i$.
    \label{lastcondition_def_of_qq}
    \end{enumerate-(1)}
    Let $g=\chi'\circ f\circ \chi^{-1}$.
    Then $g$ is simplicial from $S*\bar q$ to a subcomplex
    of $T*\bar q'$, so it is a $\Q$-PL-embedding. By the choice
    of $L$ it is easy to see that $g$ is an $\e$-approximation
    of $f$ and by condition \ref{lastcondition_def_of_qq},
    $\chi\rest |S_1|$ and
    $\chi'\rest f|S_1|$ are identity maps, so
    $f\rest S_1=g\rest S_1$.
\end{proof}

The following is a Borel version of \cite[Theorem~1']{Bing54}, a
result which is essentially due to Moise \cite{Mo52} (where it is
stated in a less convenient form).

A \textbf{triangulated $n$-manifold with boundary} is a simplicial complex
$T$ such that $|T|$ is a $n$-manifold with boundary. It is \textbf{compact}
if $T$ is finite (equivalently, if $|T|$ is compact).

\begin{Lemma}[A Borel version of a Bing-Moise lemma]%
  \label{lemma:MoiseLemma4}
  Let $X$ be the set of tuples
  $$(K_1,K_2,f,\e)\in \TTT(\U)^2\times \PartEmb(\U,\bDelta^3)\times \Q_+$$
  such that
  \begin{enumerate-(1)}
  \item $K_1$, $K_2$, and $K_1\cup K_2$ are triangulated compact
    $3$-manifolds with boundary \label{item:mm1}
  \item $L=K_1\cap K_2$ is a triangulated $2$-manifold with boundary \label{item:mm2}
  \item $f$ is an embedding $f\colon |K_1\cup K_2|\to \bDelta^3$ \label{item:mm3}
  \item $\e\in\Q_+$. \label{item:mm4}
  \end{enumerate-(1)}
  Then $X$ is Borel and there is a Borel function
  $\xi\colon X\to \Q_+$ such that if $\gamma=\xi(K_1,K_2,f,\e)$, then
  each $\Q$-PL-embedding $g\colon K_1\to \obDelta^3$ which is a
  $\gamma$-approximation of $f\rest K_1$ can be extended to a
  $\Q$-PL-embedding $g'\colon K_1\cup K_2\to \bDelta^3$ which is an
  $\e$-approximation of $f$ on $K_1\cup K_2$.
\end{Lemma}
\begin{proof}
  According to Moise~\cite{Mo52}, a complex is a triangulated
  $3$-manifold with boundary if and only if the
  star of every vertex is PL-homeomorphic to $\bDelta^3$. By
  Theorem~\ref{thm:QPL-homeo} we can w.l.o.g. change PL-homeomorphic
  to $\Q$-PL-homeomorphic.  Thus, to say that $K\in \TTT(\U)$ is a
  triangulated $3$-manifold with boundary is to say that for all
  $v\in V(K)$, there are $\Q$-subdivisions of $\St_K(v)$ and of
  $\bDelta^3$, and a simplicial map between them. All quantifiers are
  countable and the set of simplicial maps between given finite
  triangulations is in fact a finite set. To say that $K$ is compact
  is to say that all but finite number of elements in the sequence are
  dummy simplexes. This verifies that condition \ref{item:mm1} is Borel. Analogously, noticing that $|K_1 \cap K_2|$ is compact, one can show that condition  \ref{item:mm2} is Borel as well. The condition \ref{item:mm3} just says
  that the domain of $f$ must be equal to $K_1\cup K_2$ which is Borel
  by Lemma~\ref{lemma:r_C}\ref{lemma:DomainOfEmb}.
  
  Let $A$ be the set of tuples $(K_1,K_2,f,\e,\gamma)\in X\times\Q_+$
  where $(K_1,K_2,f,\e)\in X$ and $\gamma$ satisfies the conclusion.
  The set of pairs $(f,g)$ such that $g$ is a $\gamma$-approximation
  of $f$ is an open set (this is an application of
  Lemma~\ref{lemma:SupmetricHomeoEmb}). All quantifiers in the
  condition for $\gamma$ are countable (note that $K_1$ and
  $\bDelta^3$ are compact, so the number of $\Q$-PL-embeddings is
  countable). Thus, $A$ is Borel.  For each $(K_1,K_2,f,\e)\in X$ the
  section $A(K_1,K_2,f,\e):=\{\gamma\mid (K_1,K_2,f,\e,\gamma)\in A\}$
  is obviously countable. The section is non-empty by
  \cite[Theorem~1']{Bing54} and an application of Lemma
  \ref{lemma:QPLApprox}. Therefore, by \cite[Theorem 18.10]{Kec95}
  there is a Borel function as desired.
\end{proof}

The following is, again, a Borel version of 
results of Moise and Bing, see
\cite[Theorem~1]{Mo52a} and \cite[Theorem~2']{Bing54}.

\begin{Lemma}[Borel $\e$-extension lemma]%
  \label{lemma:GetTheApprox}
  Let $X$ be as in Lemma~\ref{lemma:MoiseLemma4}. Let
  $Z\subset X\times \PartEmb(\U,\bDelta^3)$ be the set of those
  $(K_1,K_2,f,\e,f')$ such that $f'\colon K_1\to \obDelta^3$ is a
  $\Q$-PL-embedding which is a $\gamma$-approximation of $f$ where
  $\gamma=\xi(K_1,K_2,f,\e)$ and $\xi$ is from
  Lemma~\ref{lemma:MoiseLemma4}. Then $Z$ is Borel and there is a
  Borel map
  $$\eta_1\colon Z\to \PartEmb(\U,\bDelta^3)$$
  such that if $g=\eta_1(K_1,K_2,f,\e,f')$, then $g$ extends $f'$ to
  $K_1\cup K_2$ and $g$ is a $\Q$-PL-embedding
  and an $\e$-approximation of~$f$.
\end{Lemma}
\begin{proof}
    By Lemma~\ref{lemma:MoiseLemma4}, $X$ is a Borel set and $\xi$ is
    a Borel function. Let
    $$j\colon X\times\PartEmb(\U,\bDelta^3)\to \R\times\Q_+$$
    be the function given by
    $j(K_1,K_2,f,\e,f')=(d(f',f),\xi(K_1,K_2,f,\e))$.  By the above,
    $j$ is Borel, and $Z=j^{-1}\{(x,y)\in \R\times\Q_+\mid x<y\}$, so
    $Z$ is Borel.  Let $A$ be the set of tuples $(K_1,K_2,f,\e,f',g)$
    such that $(K_1,K_2,f,\e,f')\in Z$ and $g$ satisfies the
    conclusion, i.e.  $g$ extends $f'$ to $K_1\cup K_2$ and $g$ is an
    $\e$-approximation of $f$.  Let us show that $A$ is Borel.  Let
    $S_1$ be the set of those tuples $(K_1,K_2,f,\e,f',g)$ for which
    $g$ is an extension of $f'$. Let $S_2$ be the set of those tuples
    $(K_1,K_2,f,\e,f',g)$ for which $g$ is an $\e$-approximation of
    $f$.  Clearly $S_1$ and $S_2$ are both Borel, and $A=S_1\cap
    S_2$. For a fixed $z\in Z$, let
    $A(z)=\{g\in \PartEmb(\U,\bDelta^3): (z,g) \in A\}$.  We will show that the
    section $A(z)$ is countable and non-empty. The result will then
    follow from \cite[Theorem 18.10]{Kec95}.  It is non-empty by the
    choice of $f'$ and Lemma~\ref{lemma:MoiseLemma4}.  It is
    countable, because the number of $\Q$-PL-embeddings from the
    compact $K_1\cup K_2$ to the compact $\bDelta^3$ is countable.
\end{proof}

\begin{Lemma}[Borel $\e$-approximation on compact polyhedra]%
  \label{lemma:Approx1}
  Let $Z$ be the set of tuples
  $$(K,f,\e)\in \TTT\times \PartEmb(\U,\bDelta^3)\times\Q_+$$
  where $K$ is a triangulated compact $3$-manifold with
  boundary, $f\colon |K|\to \obDelta^3$ is an embedding, and $\e\in \Q_+$.  Then
  $Z$ is Borel and there is a Borel map
  $$\eta\colon Z\to \PartEmb(\U,\bDelta^3)$$
  such that if $f'=\eta(K,f,\e)$, then $f'$ is a $\Q$-PL-embedding
  $|K|\to \obDelta^3$ which is an $\e$-approximation of~$f$.
\end{Lemma}
\begin{proof}
  By similar arguments as in the proof of
  Lemma~\ref{lemma:MoiseLemma4} one can show that $Z$ is Borel.  Let
  $A$ be the set of those tuples
  $(K,f,\e,f')\in Z\times\PartEmb(\U,\bDelta^3)$ such that
  $(K,f,\e)\in Z$ and $f'$ satisfies the conclusion of the lemma.
  Again, by similar arguments as above, $A$ is Borel. For each
  $(K,f,\e)\in Z$ the sections
  $$A(K,f,\e):=\{f'\in\PartEmb(\U,\bDelta^3)\mid (K,f,\e,f')\in A\}$$
  are countable. They are non-empty by \cite[Theorem 1]{Mo52a}
  and an application of Lemma~\ref{lemma:QPLApprox}.
  Thus by \cite[Theorem 18.10]{Kec95} there is a Borel function as desired.
\end{proof}

\begin{Lemma}[Borel exhaustion lemma]%
  \label{lemma:Exh}
  For every $i$, there is a Borel $\xi_i\colon \MMM^{PL,\le i}_3\to \TTT^{\N}$
  such that if $(C_j)_{j\in\N}=\xi_i(\bar\f)$, then the following conditions hold:
  \begin{enumerate-(1)}
  \item $\Cup_{j\in\N} |C_j|=U:=|\overset{\circ}\f_{i+1}|\cap \Cup_{j\le i}|\overset{\circ}\f_j|$,
  \item for all $j$, $C_j$ is PL-compatible with $\f_k$ for $k\le i$,
  \item for all $j$, $C_j$ is a triangulated compact $3$-manifold with boundary,
  \item for all $j$, $C_j\cap C_{j+1}$ is a triangulated $2$-manifold with boundary,
  \item for all $j$, $C_j\cap C_{k}=\es$ for all $j,k$ with $|j-k|>1$ \label{def:Exhaustion5}.
  \end{enumerate-(1)}
\end{Lemma}
\begin{proof}
  Use the same technique as in the previous proofs which relies on
  \cite[Theorem 18.10]{Kec95} to show the existence of the Borel
  function.  The non-emptiness of the sections in this case is
  standard, see e.g. the proof of \cite[Theorem 2]{Mo52}. 
\end{proof}

Now we are ready to prove the following.

\begin{prop}[Borel extension of triangulation]%
  \label{thm:PL-compatible_step}
  There is a Borel
  $f_i\colon \MMM^{\PL,\le i}_3\to \MMM^{\PL,\le i+1}_{3}$ such that
  \begin{equation}
    \text{if 
      $\bar\f'=f_i(\bar\f)$, then
      for all $j\ne i+1$ we have $\f_j=\f'_j$, and
      $|\f_{i+1}|=|\f'_{i+1}|$.}\label{eq:conditionofthmPLcomp}
  \end{equation}
\end{prop}
\begin{proof}
  Let $\bar\f\in\MMM^{\PL,\le i}_3$. Using the same inductive
  argument as in the first part of the proof of
  Theorem~\ref{thm:GetTrianulation}, obtain, in a Borel way, a
  triangulation $T_i$ of $\Cup_{k\le i}|\overset{\circ}\f_i|$.  Let
  $$U=|\overset{\circ}\f_{i+1}|\cap \Cup_{k\le i}|\overset{\circ}\f_k|.$$
  Let $f\colon U\to \obDelta^3$ be given by $f:=\f_{i+1}^{-1}\rest
  U$. Let $V=fU$. In a Borel way obtain an exhaustive sequence of compact polyhedra
  $(C_i)_{i\in\N}$ as in Lemma~\ref{lemma:Exh}.
  Let $\e_i=\min_{j\le i}\frac{1}{2}d(fC_j,\partial V)$. It is easy to see that
  $(\e_i)_{i\in\N}$ is obtained in a Borel way from $\bar\f$.
  Denote $D_i=C_i\cup C_{i+1}\cup C_{i+2}$.
  By
  Lemma~\ref{lemma:MoiseLemma4} obtain
  $$\gamma_i=\xi(C_i\cup C_{i+2},C_{i+1},f\rest D_i,\e_{i+1}).$$
  Then $\gamma_i$ is such that every $\gamma_i$-approximation of
  $f\rest (C_i\cup C_{i+2})$ can be extended to an
  $\e_{i+1}$-approximation of $f\rest D_i$.  By
  Lemma~\ref{lemma:Approx1} obtain, in a Borel way, for even indices:
  $$f'_{2i}:=\eta(C_{2i}, f\rest C_{2i}, \min_{j\le 2i}\gamma_j).$$
  Then $f'_{2i}$ is a $\gamma_j$-approximation of $f\rest C_{2i}$
  for all $j\le 2i$.
  Now $f'_{2i}\cup f'_{2i+2}$ is a $\gamma_{2i}$-approximation of $f\rest (C_{2i}\cup C_{2i+2})$.
  By Lemma~\ref{lemma:GetTheApprox} get for all odd indices:
  $$g_{2i+1}:=\eta_1(C_{2i}\cup C_{2i+2},C_{2i+1},f\rest D_{2i},\e_{2i+1},f'_{2i}\cup f'_{2i+2}).$$
  Then $g_{2i+1}$ extends $f'_{2i}\cup f'_{2i+2}$ to $D_{2i}$ and is
  an $\e_{2i+1}$-approximation of $f\rest D_{2i}$.
  For all $i$ we have $\dom(g_{2i+1})\cap\dom(g_{2(i+1)+1})=C_{2i+2}$
  and $g_{2i+1}\rest C_{2i+2}=f'_{2i+2}=g_{2(i+1)+1}\rest C_{2i+2}$. If $|i-j|>1$,
  then $\dom(g_{2i+1})\cap \dom(g_{2j+1})=D_{2i}\cap D_{2j}=\es$
  (to see this, apply definition of $D_i$ and condition \ref{def:Exhaustion5} of Lemma~\ref{lemma:Exh}).
  Let
  $$g:=\Cup_{i\in \N}g_{2i+1}.$$
  Then $g$ is a well-defined $\Q$-PL-embedding from $U$ to
  $\obDelta^3$. Let us show that for all $x\in U$,
  \begin{equation}
    d(f(x),g(x))<\frac{1}{2}d(f(x),\partial V).\label{eq:eApprox}
  \end{equation}
  Let $x\in U$ and let $j$ be such that $x\in C_j$.  If $j$ is odd,
  let $i$ be such that $j=2i+1$. Then $g(x)=g_{2i+1}(x)$. Since
  $g_{2i+1}$ is a $\e_{2i+1}$-approximation of $f$, we have using the
  definition of $\e_{2i+1}$:
  $$d(f(x),g(x))=d(f(x),g_{2i+1}(x))<\e_{2i+1}\le \frac{1}{2}d(fC_{2i+1},\partial V)< d(f(x),\partial V).$$
  The last inequality follows, because
  $d(fC_{2i+1},\partial V)=\inf_{y\in C_{2i+1}}d(f(y),\partial V)$,
  and $x\in C_{2i+1}$.  If $j$ is even, let $i$ be such that
  $j=2i$. We again have $g(x)=g_{2i+1}(x)$, because
  $C_{2i}\subset D_{2i}=\dom(g_{2i+1})$. Similarly as above, we have
  $$d(f(x),g(x))=d(f(x),g_{2i+1}(x))<\e_{2i+1}\le \frac{1}{2}d(fC_{2i},\partial V)< d(f(x),\partial V).$$
  The second to last inequality follows from
  definition of $\e_{2i+1}$, because $2i<2i+1$.
  By \eqref{eq:eApprox}, $gU=V$. We made sure that $g$ is obtained from
  $\bar\f$ in a Borel way. 
  Now define $\f_{i+1}'$ by
  $$\f'_{i+1}(x)=
    \begin{cases}
      g^{-1}(x),&\text{ if }x\in V\\
      \f_{i+1}(x),&\text{ otherwise.}
    \end{cases}
  $$
  By \eqref{eq:eApprox} this function is continuous and, in fact, an
  embedding.  For $j\ne i+1$ let $\f_{j}':=\f_j$.  Then
  $f_i(\bar\f)=\bar\f'$ is as desired. Indeed,
  $(\f_{i+1}')^{-1}\circ(\f_j')=g\circ \f_j$.  But $\f_j$ is a
  $\Q$-PL-embedding into $T_i$ (defined at the beginning of this
  proof) and $g$ is a $\Q$-PL-embedding from $T_i$ into $\obDelta^3$.
  Therefore $g\circ \f_j$ is a $\Q$-PL-embedding.  Symmetrically for
  $(\f_j')^{-1}\circ \f'_{i+1}$.
\end{proof}

\begin{thm}[Get PL-compatible atlas]\label{thm:GetPLcomp}
  There is a Borel map $\zeta\colon \MMM_3\to\MMM_3^{PL}$ such that
  $M(\bar\f)=M(\zeta(\bar \f))$ for all $\bar\f\in\MMM_3$.
\end{thm}
\begin{proof}
  Consider the maps
  $f_i\colon \MMM_3^{PL,\le i}\to \MMM_{3}^{PL,\le i+1}$ given by
  Proposition~\ref{thm:PL-compatible_step}.  Note that
  $\MMM_3^{PL,\le 0}=\MMM_3$. Let $\zeta\colon \MMM_3\to\MMM^{PL}_3$
  be defined as the limit of the functions $f_i$ given by
  $\zeta(\bar\f)=\bar\f'$ where for all $i$,
  $\f'_i=(f_i\circ\cdots\circ f_0)(\bar\f)_{i}$.  By
  \eqref{eq:conditionofthmPLcomp}, $\zeta$ is well-defined and
  satisfies $M(\bar\f)=M(\zeta(\bar\f))$. Let us check that it is
  Borel.  Recall that $\MMM_3^{PL}\subset \SSS_3(\U)$ and
  $\SSS_3(\U)=\Emb(\bDelta^3,\U)^\N$.  It is enough to show that
  $\zeta^{-1}N$ is a Borel set for a basic open
  $N\subset \Emb(\bDelta^3,\U)^\N$.  So let $N$ be of the form
  $$N=N_1\times \cdots\times N_i\times \Emb(\bDelta^3,\U)^\N$$
  where $N_j$ is open in $\Emb(\bDelta^3,\U)$ for $j\le i$.  Then
  $\bar \f\in \zeta^{-1}(N)$ if and only if $\zeta(\bar\f)\in N$ if
  and only if $\zeta(\bar\f)_j\in N_j$ for all $j\le i$ if and only if
  $(f_j\circ\cdots\circ f_0)(\bar\f)_j\in N_j$ for all $j\le i$ if and
  only if $\bar\f\in \xi_j^{-1}(N_j)$ for all $j\le i$ where
  $\xi_j=\pr_j\circ f_j\circ\cdots\circ f_0$. Thus, we have
  $\zeta^{-1}N=\Cap_{j\le i}\xi_j^{-1}N_j$.  Clearly $\xi_j$ is a Borel
  function, so it follows that $\zeta^{-1}N$ is Borel.
\end{proof}

Finally the main theorem of this section:

\begin{thm}[Borel triangulation of 3-manifolds]%
  \label{thm:BorelTriangulation}
  There is a Borel
  $\eta\colon \MMM_3\to \TTT$ such that for all $\bar\f\in\MMM_3$,
  $\eta(\bar\f)$ is a triangulation of $M(\bar\f)$.
\end{thm}
\begin{proof}
  Let $\zeta\colon \MMM_3\to\MMM^{PL}_3$ be as given by
  Theorem~\ref{thm:GetPLcomp}.  Let $\xi\colon \MMM^{PL}_3\to \TTT$ be
  the function given by Theorem~\ref{thm:GetTrianulation}.  Let
  $\eta=\xi\circ\zeta$. Then $\eta$ is as desired.
\end{proof}

\begin{thm}[$\homeo_3\ \le_B\ \cong_\AAA$]
   \label{thm:MainBorelReduction3}
  There is a Borel $\xi\colon \MMM_3\to \AAA$ 
  which reduces $\homeo_3$ to $\cong_\AAA$.
\end{thm}
\begin{proof}
  Let $\zeta\colon \TTT\to \AAA$ be the reduction
  $\homeo^{PL}\ \le_B\ \cong_\AAA$ given by
  Theorem~\ref{thm:BorelClassificationOfPLtoISO}.  Let
  $\eta\colon \MMM_3\to \TTT$ be the function given by Theorem
  \ref{thm:BorelTriangulation} which gives a triangulation for a given
  $3$-manifold.  Let $\xi=\zeta\circ\eta$. Suppose
  $\bar\f,\bar\f'\in\MMM_3$. Then they are homeomorphic if and only if
  $\eta(\bar\f)$ is PL-homeomorphic to $\eta(\bar\f')$ (follows from
  Fact~\ref{fact:MoiseBing}) if and only if $\zeta(\eta(\bar\f))$ is
  isomorphic to $\zeta(\eta(\bar\f'))$ (follows from the choice
  of~$\zeta$).
\end{proof}

\begin{remark}
  With some extra work we could prove that $\homeo_3$ is Borel
  reducible to the isomorphism relation on partial orders using
  Proposition \ref{prop:BasisToPPP}, but this also follows from the
  fact that partial orders are classifiable by countable structures by
  classical results~\cite{Friedman1989}.
\end{remark}

\vspace{10pt}

\subsubsection{\(2\)-manifolds without boundary up to homeomorphism}\label{sec:2manifolds}


In this section, we give a proof for the classification of
\(2\)-manifolds without boundary by countable structures. This classification was already
established in \cite{BS25} using methods that rely essentially on the
two-dimensional setting and on specific properties of surfaces. Notice
that from Corollary \ref{cor:HomeoPL_on_Man_red_cong} we have in
particular that PL \(2\)-manifolds are classifiable by countable
structures. The step we were missing to prove that
\(2\)-manifolds without boundary are classifiable by countable
structures is showing that one can get a triangulation of
\(2\)-manifolds without boundary in a Borel way. This has been done in
\cite{BS25}. Adapting our parametrization of \(2\)-manifolds to the
one given in \cite{BS25}, the arguments developed in this paper for PL
\(2\)-manifolds yield the following result.

\begin{thm}[$\homeo_2\ \le_B\ \cong_\AAA$]
   \label{thm:MainBorelReduction2}
  There is a Borel $\xi\colon \MMM_2\to \AAA$ 
  which reduces $\homeo_2$ to $\cong_\AAA$.
\end{thm}

Before proving the theorem, we state some definitions from \cite{BS25}. 

\begin{defn}\cite[Definitions 3.3, 3.5, and 3.8, Example 3.4, Section 4.2]{BS25}\label{def:bs25}
    Let Top be the collection of homeomorphisms between open subsets of \(\R^2\). The standard Borel space $\MMM_2^*$ of \(2\)-manifolds without boundary is given by all pairs $(\UU,c)$, where
    \begin{itemize}
        \item $\UU = \langle U_{i,j} \mid (i, j) \in \N^2\rangle$ is a family of open subsets of \(\R^2\),
    \item $c = \langle \phi_{i,j} : U_{i,j} \to U_{j,i} \mid (i, j)\in \N^2\rangle \subset \text{Top}$,
    \end{itemize}
    such that the following conditions hold:
    \begin{itemize}
    \item $U_{i,j} \subset U_i := U_{i,i}$ for all $(i, j) \in \N^2$,
    \item $\phi_{i,i} = \id_{U_i} : U_i \to U_i$ for all $i \in \N$,
    \item $\phi_{i,j} = \phi^{-1}_{j,i}$ for all $(i, j) \in \N^2$,
    \item $\phi^{-1}_{i,j} [U_{j,i} \cap U_{j,k}] \subset U_{i,k}$ and $\phi_{j,k} \circ \phi_{i,j}
    \restriction (\phi^{-1}_{i,j}[U_{j,i}\cap U_{j,k} ]) = \phi_{i,k}\restriction (\phi^{-1}_{i,j}[U_{j,i} \cap U_{j,k}])$ for
    all $i, j, k \in \N$, and
  \item $\coprod_{i \in \N} U_i/\sim$ is Hausdorff, where \(\sim\) is
    the equivalence relation defined on the disjoint union
    $\coprod_{i \in \N} U_i$ of the $U_i$'s by $x \sim y$ for
    $x \in U_i$ and $y \in U_j$ if and only if \(\phi_{i,j}(x)=y\).
    \end{itemize}

    The standard Borel space \(\mathcal{S}\) of all \(2\)-simplicial complexes is the set of all triples $S = (S_0, S_1, S_2) \in \{0, 1\}^\N\times \{0, 1\}^{\N^2}\times \{0, 1\}^{\N^3}$
    such that:
    \begin{itemize}
        \item $S_1(i, j) = S_1(j, i)$ for all $i, j \in \N$;
        \item $S_2(i, j, k) = S_2(i, k, j) = S_2(j, i, k) = S_2(j, k, i) = S_2(k, i, j) = S_2(k, j, i)$ for all $i, j, k \in \N$;
        \item $S_2(i, j, k) = 1$ implies $S_1(i, j) = S_1(i, k) = S_1(j, k) = 1$ for all $i, j, k \in \N$;
        \item $S_1(i, j) = 1$ implies $S_0(i) = S_0(j) = 1$ for all $i, j \in \N$; and
        \item for all $i \in \N$, there are at most finitely many $j$ such that $S_2(i, j) = 1$.
    \end{itemize}
    So an element $S =
    (S_0, S_1, S_2)$ can be identified with the abstract simplicial complex given by
    $K = \{n \mid S_0(n) = 1\}$, 
    where $\{i\}$ is a $0$-simplex in $K$ if and only if
    $S_0(i)=1$, $\{i, j\}$ is a $1$-simplex in $K$ if and only if
    $S_1(i, j) = 1$, and $\{i, j, k\}$ is a $2$-simplex in $K$ if and only if $S_2(i, j, k) = 1$.
    The Borel subset \(\mathcal{S}_{\text{poly}}\) of all elements of \(\mathcal{S}\) for which their realization are surfaces are those such that:
\begin{itemize}
    \item every $1$-simplex is in exactly two $2$-simplices,
    \item the $1$-simplices and $2$-simplices which contain a given vertex $x$ can be written as $a_1, \dots, a_m$ and $A_1,\dots, A_m$, respectively, for $m \geq 3$, so that $a_1 = A_1 \cap A_m$ and $a_i = A_i \cap A_{i-1}$ for $2 \leq i \leq m$, and
    \item \(S\) is connected (i.e., it cannot be written as a union of two disjoint subcomplexes).
\end{itemize}
\end{defn}
\begin{proof}[Proof of Theorem \ref{thm:MainBorelReduction2}]
  We first show that \(\homeo_2\ \le_B\ \homeo_2^*\), where
  \(\homeo_2^*\) is the homeomorphism relation on the space of
  \(2\)-manifolds \(\MMM_2^*\) given in Definition \ref{def:bs25}.
  Fix an isometry \(\iota\colon \R^2 \to \U\). For each
  $\bar \f \in \MMM_2$ and each \(i \in \N\), define
  \(U_i=\ran (\iota^{-1}\circ \overset{\circ}{\f_i})\). Then consider
  the map $\xi_0\colon \MMM_2 \to \MMM_2^*$ given by
  $\xi_0(\bar \f)=(\langle U_{i,j} \mid (i, j) \in \N^2\rangle,
  \langle \phi_{i,j} \mid (i, j)\in \N^2\rangle)$, where
\begin{itemize}
    \item \(U_{i,j}=U_i \cap U_j\) for all \(i,j \in \N\), and
    \item \(\phi_{i,j}=(\overset{\circ}{\f_j})^{-1} \circ \overset{\circ}{\f_i}\) for all \(i,j \in \N\).
\end{itemize}   
It is easy to see that \(\xi_0\) is a Borel reduction of \(\homeo_2\) to \(\homeo_2^*\).

    By \cite{BS25}, there is a Borel reduction \(\homeo_2^*\ \le_B\ \cong_{\mathcal{S}_{\text{poly}}}\). We will 
    show that \(\cong_{\mathcal{S}_{\text{poly}}}\ \le_B\ \homeo^{PL}\). 
    Then it is enough to apply Theorem \ref{thm:BorelClassificationOfPLtoISO} to obtain the desired result. 
   
    To prove that \(\cong_{\mathcal{S}_{\text{poly}}}\ \le_B\ \homeo^{PL}\), recall that each $S \in \mathcal{S}_{\text{poly}}$ can be identified with the abstract simplicial complex $K= \{n \mid S_0(n) = 1\}$. We define its geometric realization $\RR(K)$ by
    \[
\RR(K)=\Big\{(x_n) \in [0,1]^{\N} \mid x_n \geq 0,\ \sum_{i \in \N} x_i=1,\ \supp((x_n)) \text{ is finite},\ \supp((x_n)) \in K\Big\},
    \]
    where \(\supp((x_n))=\{n \in \N \mid x_n \neq 0\}\). 
    So for each \(s=\{i_0, \dots, i_k\} \in K\), we define \[
    \RR(s)=\{(x_n) \in [0,1]^{\N} \mid x_j=0 \text{ if } j \notin s, \sum_{j \in s} x_j=1\}.
    \] 
    Clearly, $\RR(K)=\Cup_{s \in K} \RR(s)$. 
    Let $\delta$ be the metric on the subset \(X\) of $[0,1]^{\N}$ of all sequences with finite support
    defined by
    $$\delta((x_n),(y_n))=\sum_{n\in\N}n |x_n-y_n|.$$
    Note that by the finiteness of the support of the sequences, $\delta$
    is always well-defined. Then each complex \(K\) is Heine-Borel
    by the local finiteness of complexes, i.e., if \((s_i)_{i \in \N}\) is an enumeration of the elements of \(K\), then for all \(n \in \N\), there is \(m \in \N\) such that for all $k > m$ we have $B(x, n)\cap \RR(s_k) = \es$ for some fixed point \(x \in X\).

    Now, if \(s=\{i_0, \dots, i_k\} \in K\), define \(f_s\) as the unique affine map which sends every vertex \(\ee_j \in \bDelta^k\) to the element \(e_{i_j}=(0, \dots,0,1,0,\dots) \in X\) whose coordinates are all equal to \(0\) except for the \(j\)-th one which is \(1\).
    Fix an isometry \(\iota\colon X \to \U\). Then for each \(s \in K\), let \(\tau_s= \iota \circ f_s \in \Emb(\bDelta^k,\U)\). We show that \((\tau_s)_{s \in K} \in \TTT\). It is easy to prove that $(\tau_s)_{s \in K}$ satisfies properties \ref{def:simplicial_complex-1}--\ref{def:simplicial_complex-3} of Definition \ref{def:simplicial_complex}. It remains to show that \((\tau_s)_{s \in K}\) has the Heine-Borel property. This easily follows from the fact that \(K\) is Heine-Borel in the sense of the previous paragraph and \(\iota\) is an isometry. 
    
    We finally define \(\xi_1\colon \mathcal{S}_{\text{poly}} \to \TTT\) by $\xi_1(S)=(\tau_s)_{s \in K}$, which gives the desired Borel reduction. 
\end{proof}

\vspace{10pt}

\subsubsection{Some classes of \(n\)-manifolds with boundary}\label{sec:with_boundary}

In this section we deal with the Borel classification of some classes of $n$-manifolds with boundary. First, we need to find a suitable parametrization of all \(n\)-manifolds with boundary. This is done analogously to the space of \(n\)-manifolds without boundary introduced in Section~\ref{ssec:SpaceOfManifolds}. 
Denote by $\ubDelta^n$ the \(n\)-simplex \(\bDelta^n\) regarded as a subset of the closed half-space \(\R_+^{n+1}\) of \(\R^{n+1}\), with all topological operations on \(\ubDelta^n\) taken relative to \(\R_+^{n+1}\). For example, its interior $\uobDelta^n$ is the union of $\obDelta^n$ with the interior (in the sense of Definition \ref{def:Faces}) of the \((n-1)\)-face of \(\bDelta^n\) given by the convex hull of \(\{\ee_0, \dots, \ee_{n-1}\}\). 

\begin{defn}
  Let \(\SSS_n^\partial\) be the subset of all sequences \(\bar \f=(\f_i)_{i \in \N}\) in \((\Emb(\bDelta^n,\U) \cup \Emb(\ubDelta^n,\U) \cup \{\dummy\} )^{\N}\) with at least one component in \(\Emb(\ubDelta^n,\U)\). 
  Set $M(\bar\f)=|\bar\f|=\Cup_{i\in\N}|\f_i|$.   
    We say that \(\bar \f \in \SSS_n^\partial\) is an \(n\)-manifold with boundary if and only if $(M(\bar\f),\bar \f)$ satisfies conditions \ref{item:m1}--\ref{item:m3} and  $M(\bar\f)$
  is Heine-Borel. We denote the space of all \(n\)-manifolds with boundary by \(\MMM_n^\partial\).
    
    Let $\MMM_n^{\partial, PL}\subseteq \MMM_n^\partial$ be the set of those
  $(\f_i)_{i \in \N}$ such that for all $i,j$, $\f_i$ and $\f_j$ are locally PL-compatible.
\end{defn}

Notice that in the  previous definition each \(\f_i\) is an embedding of either \(\bDelta^n\) or \(\ubDelta^n\) into \(\U\) and so \(\overset{\circ}\f_i\) is either \(\f_i\rest \obDelta^n\) or \(\f_i\rest \uobDelta^n\). Similarly to Lemma \ref{lemma:FindEveryManifold}, one can prove that if $M$ is an $n$-manifold with boundary, then there is $\bar\f\in\MMM_n^\partial$ such
that $M$ is homeomorphic to $M(\bar\f)$. Conversely, for every
  $\bar\f\in\MMM_n^\partial$, $(M(\bar\f),\bar\f)$ is an $n$-manifold with boundary.

\begin{thm}
    $\MMM_n^\partial$ and $\MMM_n^{\partial, PL}$ are standard Borel spaces.  
\end{thm}
\begin{proof}
Using Lemma \ref{lemma:EmbIsClosed}, one can show that $\SSS_n^\partial$ is a standard Borel space. Then it is enough to use an argument similar to Proposition \ref{prop:ManifoldsBorel} and Lemma \ref{lemma:AlgPLManifoldsBorel}, respectively, to prove that  \(\MMM_n^\partial\) and  \(\MMM_n^{\partial,PL}\) are Borel subsets of \(\SSS_n^\partial\).   
\end{proof}

\begin{thm}[Borel triangulation of PL-manifolds with boundary]%
  \label{thm:GetTrianulation2}
  There is a Borel $\xi\colon \MMM^{\partial,PL}_n\to \TTT$
  such that for all $\bar\f\in \MMM^{\partial,\PL}_n$, $\xi(\bar\f)$ is a compatible
  triangulation of~$M(\bar\f)$.
\end{thm}
\begin{proof}
    It follows from similar arguments to those of the proofs of Lemma \ref{prop:TriangulationStep} and Theorem \ref{thm:GetTrianulation}.
\end{proof}

We now say that $\bar\f,\bar\f'\in\MMM^{\partial, PL}_n$ are PL-homeomorphic, and write $\bar\f\homeo^{\partial,PL}_n\bar\f'$, if 
  there is a homeomorphism $h\colon M(\bar\f)\to M(\bar\f')$
  such that for all $\tau\colon\bDelta^n\to M(\bar\f)$,
  if $i$ and $j$ are such that $|\tau|\subset |\overset{\circ}{\f}_i|$,
  $|h\circ\tau|\subset |\overset{\circ}{\f}{}'_j|$
  and $\f_i^{-1}\circ\tau$ is PL, then
  $(\f'_j)^{-1}\circ h\circ\tau$ is also PL. 
  Using  Theorems \ref{thm:GetTrianulation2} and \ref{thm:BorelClassificationOfPLtoISO} we obtain the following result which says that PL \(n\)-manifolds are classifiable by countable structures.

\begin{thm}[$\homeo^{\partial,PL}\,\le_B\ \cong_\AAA$]
\label{thm:PLman_withboundary}
    There is a Borel map
    $\eta\colon \MMM^{\partial,PL}_n\to\AAA$ 
    which reduces $\homeo^{\partial,PL}_n$
    to~$\cong_\AAA$.
\end{thm}

We are now able to get a classification by countable structures for \(3\)-manifolds with boundary as well. The main observation is that the boundary \(\partial M\) of a \(3\)-manifold \(M\) with boundary is a \(2\)-manifold without boundary and that one can extend a triangulation of the latter to the whole \(M\) in a Borel way. 

We denote by \(\homeo^\partial_3\) the homeomorphism relation between elements of \(\MMM_3^\partial\).
Given $\bar\f\in\MMM_3^\partial$,
one canonically obtains $\bar\psi\in \MMM_2$
such that $\partial M(\bar\f)=M(\bar\psi)$.
This assignment $\bar\f\mapsto\bar\psi$ 
is Borel.
Let $\partial\colon \MMM_3^\partial\to \MMM_2$
be this Borel map. Recall now that every topological \(3\)-manifold \(M\) with boundary admits a collar, that is, an open neighborhood of \(\partial M\) homeomorphic to \(\partial M \times [0,1)\). Denote by \(c:\MMM_2 \to \MMM_3^\partial\) the Borel map assigning to each \(M\) the element \(M \times [0,1)\). 

\begin{thm}[$\homeo_3^\partial\ \le_B\ \cong_\AAA$]
\label{thm:3man_withboundary}
  There is a Borel $\xi\colon \MMM_3^\partial\to \AAA$ 
  which reduces $\homeo_3^\partial$ to $\cong_\AAA$.
\end{thm}
\begin{proof}
    Let \(M \in \MMM_3^\partial\). By \cite{BS25} and the proof of Theorem \ref{thm:MainBorelReduction2}, there is a Borel procedure producing a triangulation \(T_{\partial(M)}\) of \(\partial(M)\). Using the standard triangulation of products of simplicial complexes, \(T_{\partial(M)}\) extends canonically to a triangulation of the collar \(c(\partial(M))\). Arguing as in the first part of the proof of Theorem \ref{thm:GetTrianulation}, this triangulation can then be extended to a triangulation \(T\) of the whole manifold \(M\). Both extension procedures are Borel, and therefore yield a Borel map
    \(
    \eta\colon \MMM_3^\partial \to \TTT
    \)
    assigning to each \(M\in\MMM_3^\partial\) a triangulation \(T\).
    Then we have that \(\bar \f,\bar \f' \in \MMM_3^\partial\) are homeomorphic if and only if \(\eta(\bar \f)\) and \(\eta(\bar\f')\) are PL-homeomorphic, and by Theorem \ref{thm:BorelClassificationOfPLtoISO} we obtain the desired reduction. 
\end{proof}






\vspace{10pt}

\subsection{Borel classification of open subsets of~$\R^2$ and $\R^3$ up to homeomorphism}
\label{ssec:OpenSubsets}


In this section we will prove that there is a Borel reduction of the
homeomorphism relation on open subsets of $\R^n$ for $n\in \{2,3\}$ to
the isomorphism relation on sorted complemented algebras.  First, let
us define the Borel space of open subsets of $\R^n$.

\begin{Def} \label{def:SpaceOfOpenSets}
  Let $O(\R^n)$ be the set of open subsets of~$\R^n$. Let $K(S^n)$ be
  the space of compact subsets of the $n$-dimensional sphere $S^n$.
  Let $h\colon \R^n\to S^n$ be the canonical embedding given by the
  inverse of the stereographic projection, and denote by
  $\infty\in S^n$ the only point not in the range of $h$. Then the set
  $K^{\infty}(S^n):=\{K\in K(S^n)\mid\infty\in K\}$ is a closed subset
  of $K(S^n)$ and hence a compact Polish space.  Let
  $g\colon O(\R^n) \to K^{\infty}(S^n)$ be defined by
  $g(O)=S^n\setminus h[O]$. 
  We equip $O(\R^n)$ with the metric
  $$d(O,O')=d_H(g(O),g(O'))$$
  where $d_H$ is the Hausdorff metric on $K^{\infty}(S^n)$.
  This makes $O(\R^n)$ into a compact space.
  Let $\homeo_n^O$ be the homeomorphism relation on~$O(\R^n)$.
\end{Def}

The following is an application of the Whitney Decomposition Theorem.

\begin{prop}\label{prop:FromOpenToMan}
  There is a Borel $\xi\colon O(\R^n)\to \MMM^{PL}_n$ such that for all
  $O\in O(\R^n)$, $O\homeo M(\xi(O))$. 
\end{prop}
\begin{proof}
  For a closed cube $C\subset\R^n$ and $\delta\in\R_+$, denote by
  $C[\delta]$ the closed cube whose side length is $(1+\delta)$ times
  that of~$C$.
  Let $O\in O(\R^n)$.
  The Whitney Decomposition Theorem 
  (for a precise statement that we rely on, see \cite[Appendix J]{grafakos2014classical})
  gives a collection 
  of dyadic cubes $\CC(O)$ such that 
  $$D=\{C[1/8]\mid C\in\CC(O)\}$$
  is a countable cover of $O$ such that for every $x\in O$ there
  is $\e>0$ such that $B(x,\e)$ intersects only finitely many
  elements of $D$. Suppose $D=(C_i)_{i\in\N}$ is an enumeration of $D$. 
  Every $C[1/8]$ is canonically PL-homeomorphic
  to $\bDelta^n$ with respect to a rectilinear triangulation of both $\bDelta^n$ and $\R^n$, so for each $i$ fix $\f_i\colon \bDelta^n\to C_i$
  to be such a canonical PL-homeomorphism. Clearly $\bar\f$
  satisfies \ref{item:m1}--\ref{item:m3} for $O=M$
  and the maps $\f_i$, $\f_j$ are PL-compatible
  for $i,j\in\N$.
  It remains to show that $\CC(O)$ can be obtained from $O$ in a constructive
  way so that the map $O\mapsto \bar\f$ is, in fact, Borel.
  We refer the reader to the proof of Whitney Decomposition
  in \cite[Appendix J]{grafakos2014classical} which gives a 
  sufficiently constructive proof for this purpose.
\end{proof}


\begin{Cor}[$n=2,3$, $\homeo^O_n\ \le_B\ \cong_\AAA$]%
  \label{cor:RedFromOpenToMan}
  Let $n\in\{2,3\}$.
  There is a Borel map 
  $\xi\colon O(\R^n)\to \AAA$ reducing $\homeo_n^O$
  to~$\cong_\AAA$.
\end{Cor}
\begin{proof}
  Let $\xi_1\colon O(\R^n)\to \MMM^{PL}_n$ be as given by Proposition \ref{prop:FromOpenToMan},
  and $\xi_2\colon \MMM_n^{PL}\to\AAA$
  given by Corollary~\ref{cor:HomeoPL_on_Man_red_cong}. Let $\xi=\xi_2\circ\xi_1$.
  Suppose $O,O'\in O(\R^n)$. Then $O\homeo O'$
  if and only if $\xi_1(O)\homeo \xi_1(O')$
  (by the choice of $\xi_1$)
  if and only if $\xi_1(O)\homeo^{PL}_n\xi_1(O')$
  (by Fact~\ref{fact:MoiseBing} and the choice of~$n$)
  if and only if $\xi_2(\xi_1(O))\cong \xi_2(\xi_1(O'))$ by the choice of $\xi_2$.
\end{proof}




\vspace{10pt}

\subsection{Borel classification of Cantor sets in~$\R^3$ up to 
conjugacy}
\label{ssec:BorelClassCantor}


\begin{Def}
  Let $\CCC(\R^3)$ be the space of Cantor subsets of $\R^3$ defined by
  $$\CCC(\R^3)=\{K\in K(\R^3)\mid K\text{ is perfect and totally disconnected}\}.$$
  By \cite{GKB13} this is a Polish space. The conjugacy relation on
  $\CCC(\R^3)$ is defined by $C\simeq C'$ iff there is a homeomorphism
  $h\colon \R^3\to \R^3$ such that $h[C]=C'$.
\end{Def}

The following answers Question 1.1 of \cite{GKB13}.

\begin{thm}[$\simeq\ \le_B\ \cong_\AAA$]%
  \label{thm:CantorBorelRed}
  There is a Borel $f\colon \CCC(\R^3)\to \AAA$ which reduces 
  $\simeq$ to $\cong_\AAA$.
\end{thm}
\begin{proof}
  By Corollary~\ref{cor:RedFromOpenToMan} it is enough to show that
  there is a Borel reduction from $\simeq$ to the $\homeo$-relation on
  $O(\R^3)$.  As noted in the proof of Corollary \ref{cor:CantorSets}, the
  Cantor sets $C,C'$ are conjugate if and only if
  $\R^3\setminus C\homeo \R^3\setminus C'$. It is straightforward to see
  that the map $C\mapsto \R^3\setminus C$ is a Borel map from $\CCC(\R^3)$
  to $O(\R^3)$ (see Definition~\ref{def:SpaceOfOpenSets}).
\end{proof}


\vspace{10pt}
\changelocaltocdepth{1}

\section{Final results and remarks}
\label{sec:Final}
In this last section we show that many equivalence relations previously studied are Borel bireducible with the isomorphism on countable graphs, denoted by $\cong_{\mathcal{G}}$.
We start proving that the latter is Borel reducible to the homeomorphism between open subset of $\R^n$, for $n\ge 2$. The result is known, but we state it for the sake of completeness. 

\begin{prop}[$\cong_{\mathcal{G}}\ \le_B\ \homeo_n^{O}$]%
    \label{thm:CongToHomeo}
    $\cong_{\mathcal{G}}$ is Borel reducible to $\homeo_n^{O}$ for all $n\ge 2$.
\end{prop}
\begin{proof}
    By \cite{CamGao2001}, $\cong_{\mathcal{G}}$
    is Borel reducible to the homeomorphism
    relation $\homeo$ on $K(2^\N)$. 
    Let $h\colon 2^\N\to \R$  be the standard
    ``$\frac{1}{3}$-Cantor set embedding''
    into $\R$ and let $\iota\colon\R\to\R^n$ be the standard inclusion. Let $\xi\colon K(2^\N)\to O(\R^n)$
    be given by 
    $\xi(K)=\R^n\setminus (\iota\circ h)[K].$
    This map is easily seen to be Borel. Let us show that it
    is a reduction. If $K\homeo K'$, then
    $(\iota\circ h)[K]\homeo (\iota\circ h)[K']$
    and this homeomorphism extends to $\R^2$
    by \cite[Chapter 13, Theorem 7]{Mo77}, which in turn trivially
    extends to $\R^n$. Thus, $\xi(K)\homeo^O_n \xi(K')$. If, on the other hand, 
    $\xi(K)\homeo^O_n \xi(K')$, then this homeomorphism extends
    to a homeomorphism between $(\iota\circ h)[K]$ and $(\iota\circ h)[K']$ by
    \cite[Theorem~4.1]{CM83} and the remark right after its proof. Conjugating with $\iota\circ h$ gives a homeomorphism from $K$ to~$K'$.
\end{proof}

\begin{Cor}
    For $n \in \{2,3\}$, the relations \(\homeo^O_n\), \(\homeo^{PL}_n\), \(\homeo^{\partial,PL}_n\), \(\homeo_n\), \(\homeo_3^\partial\), \(\cong_{\AAA}\),  and $\simeq$ are Borel bireducible with \(\cong_{\mathcal{G}}\).
\end{Cor}
\begin{proof}
    By classical results we have 
    $\cong_\AAA\ \le_B\ \cong_{\mathcal{G}}$ (see \cite{Friedman1989}). Let $n \in \{2,3\}$. Then by Corollary \ref{cor:RedFromOpenToMan}, Proposition \ref{thm:CongToHomeo} and its proof, we obtain $\cong_{\mathcal{G}}\ \le_B\ \homeo^O_n\ \le_B\ \homeo^{PL}_n\ \le_B\ \cong_\AAA$, and so \(\homeo^O_n\), \(\homeo^{PL}_n\) and \(\cong_\AAA\) are Borel bireducible with \(\cong_{\mathcal{G}}\). 

    The same reduction defined in the proof of Proposition \ref{thm:CongToHomeo} shows that \(\cong_{\mathcal{G}}\ \le_B\ \homeo^{\partial,PL}_n\), \(\cong_{\mathcal{G}}\ \le_B\ \homeo_n\), \(\cong_{\mathcal{G}}\ \le_B\ \homeo_3^\partial\), and so by Theorems \ref{thm:PLman_withboundary}, \ref{thm:MainBorelReduction3}, \ref{thm:MainBorelReduction2},  and \ref{thm:3man_withboundary} we have that they are Borel bireducible.

    Finally,  by \cite[Theorem 5.4]{GKB13} and using Theorem~\ref{thm:CantorBorelRed} we obtain $\cong_{\mathcal{G}}\ \le_B\ \simeq\ \le_B\ \cong_{\mathcal{G}}$.
\end{proof}


\subsubsection*{Manifolds given not as an atlas}

We have shown that if we parametrize $n$-manifolds in terms of
countable atlases $\bar\f=(\f_i)_{i\in\N}$ satisfying
\ref{item:m1}--\ref{item:m3}, then they form a standard Borel space
and that for $n\in \{2,3\}$ there is a Borel reduction of
homeomorphism on $n$-manifolds into isomorphism of countable
structures. We also showed, that in the special case of open subsets
of $\R^n$, $n\in \{2,3\}$, we can find such a reduction from the space
$O(\R^n)$ of open subsets by first constructing an atlas satisfying
\ref{item:m1}--\ref{item:m3} and then applying the above result.  One
may ask, what about a general case, where the manifold is neither
given as an atlas nor as a subset of~$\R^3$ or $\R^2$. One answer is
that every $2$- and $3$-manifold can be isometrically embedded as a
closed subset of $\R^N$ for some large enough $N$ by the Nash
embedding theorem.  Thus, the Effros space $F(\R^N)$ contains a
homeomorphic copy of every $2$- and $3$-manifold. It is not clear to
the authors whether $\{F\in F(\R^N)\mid F\text{ is a 3-manifold}\}$ is
a Borel set, but it is possible to find $\MM_3\subset F(\R^N)$ such that
every 3-manifold is represented in $\MM_3$ up to homeomorphism. Let
$$\MM_3=\{F\in F(\R^N)\mid \forall q\in D(F)(B(q,1)\cap F\homeo \bDelta^3)\}.$$
Every $3$-manifold can be made smooth and Heine-Borel, and so the
metric can be locally modified such that $B(x,1)$ is
homeomorphic to $\bDelta^3$ for every point $x$. A similar trick was used by Hjorth and
Kechris \cite{HjoKec2000} to parametrize complex manifolds. Thus, $\MM_3$
contains a homeomorphic copy of every $3$-manifold. This set $\MM_3$ is a
Borel subset of $F(\R^N)$ and using a similar technique as in
Section~\ref{ssec:OpenSubsets} we can reduce the homeomorphism on it
to the homeomorphism relation on~$\MMM_3$. All of this of course
applies to $2$-manifolds as well.

\subsubsection*{Higher dimensional manifolds}

If $n>3$, then there are examples of triangulated non-compact $n$-manifolds $T$ and $T'$
such that $|T|$ and $|T'|$ are homeomorphic but not PL-homeomorphic. This is
an obstruction to the Moise-Bing
theorem (Fact~\ref{fact:MoiseBing}), 
and therefore to Theorems \ref{thm:ClassificationOfManifolds} and \ref{thm:MainBorelReduction3}.
If we produce basis spaces $(X,\beta)$ and $(X',\beta')$ from such triangulations
according to the construction of Definition \ref{def:BetaFromT}, then the resulting
basis spaces might be non-equivalent even though the generated topologies
are homeomorphic as are
$|T|$ and $|T'|$. 
Then also the corresponding SCA algebras will be non-isomorphic. 
For a simpler
example of non-equivalent basis spaces which generate the same topology, see Example~\ref{ex:BasisSpacesNonEq}.
So we have:

\begin{Question}
  Is the homeomorphism relation on non-compact $n$-manifolds Borel
  reducible to the isomorphism on countable graphs for $n>3$?
\end{Question}

The upper bound problem in the preceding question is also addressed in ongoing work of Gompf and Panagiotopoulos \cite{GP26}. They establish, in work in progress, that topological $n$-manifolds, for $n>1$, are classifiable by countable structures. In particular, their result would answer the first question above affirmatively. Their approach is substantially different from ours and applies in considerably greater generality.

The same obstruction also leaves open the converse direction for  PL $n$-manifolds. While PL $n$-manifolds are classifiable by countable structures for all dimensions $n$, our methods do not establish that graph isomorphism is a lower bound for PL-homeomorphism.

\begin{Question}
    Is the isomorphism on countable graphs Borel
  reducible to the PL-homeomorphism relation on PL $n$-manifolds for $n>3$?
\end{Question}

The same question can be formulated for Heine-Borel simplicial complexes.

\begin{Question}
    Is the isomorphism on countable graphs Borel
  reducible to the PL-homeomorphism on Heine-Borel simplicial complexes?
\end{Question}

\subsubsection*{Relationship to other invariants}

We have produced a complete invariant of the homeomorphism of
non-compact $2$- and $3$-manifolds in the form of the sorted
complemented algebras (SCA),
Theorem~\ref{thm:ClassificationOfManifolds}. The elements of the
algebra $A_M$ are essentially just interiors of the compact and
co-compact polyhedra in a manifold~$M$ with ``rational'' vertex points. Thus, at
least in theory, all classical invariants of algebraic topology should
be definable in $A_M$, or at least, if $A_M$ is enhanced with some
extra structure. For example we can add also the lower-dimensional
polyhedra to $A_M$, or consider the closed polyhedra instead of their
open interiors etc. As a logical question, we can ask, under which
circumstances is the structure $A_M$ fully described by the
given invariants? For example, we should be able to find a formula $\f$
(possibly in $L_{\infty\omega}$) such that $A_M\models \f$ if and only if
$M$ is simply connected. The Perelman-Poincaré theorem
suggests that if $M$ is a compact $3$-manifold without boundary,
then $\f$ fully determines the isomorphism type of~$A_M$.
More generally, we can state an open research program:

\begin{RP}
  Use the fact that properties of $M$ are definable in $A_M$ to build
  ``model theory'' of $3$-manifolds.
\end{RP}




\bibliographystyle{alpha}
\bibliography{bibliography}

\end{document}